\documentclass[11pt]{article}

\usepackage[margin=1in]{geometry}

\usepackage{amssymb}
\usepackage{amsmath}
\usepackage{amsthm}
\usepackage{empheq}
\usepackage[english]{babel}
\usepackage[utf8]{inputenc}
\usepackage{latexsym}
\usepackage{mathrsfs}
\usepackage{bbm}
\usepackage{graphicx}
\usepackage{float}
\usepackage{subcaption}
\usepackage{xcolor}
\usepackage[hidelinks]{hyperref}
\usepackage{accents}
\usepackage{cases}
\usepackage{url}
\usepackage{comment}
\usepackage{cleveref}

\numberwithin{equation}{section}

\newtheorem{thm}{Theorem}[section]
\newtheorem{prop}[thm]{Proposition}
\newtheorem{defn}[thm]{Definition}
\newtheorem{lemma}[thm]{Lemma}

\newtheorem{remark}[thm]{Remark}

\newcommand{\nn}{\mathbb{N}}
\newcommand{\rr}{\mathbb{R}}

\newcommand{\eee}{\mathbb{E}}

\newcommand{\aaa}{\mathcal{L}}
\newcommand{\mm}{\mathcal{M}}
\newcommand{\hh}{\mathcal{H}}

\newcommand{\vphi}{\varphi}

\newcommand{\mf}{\mathcal{F}}
\newcommand{\pbm}{\mathbb{P}}
\newcommand{\rn}{\mathbb{R}}
\newcommand{\na}{\mathcal{L}}
\newcommand{\m}{\mathcal{M}}

\allowdisplaybreaks

\title{\textbf{Nonzero-sum stochastic games and mean-field games with impulse controls}}
\author{
Matteo Basei\footnote{Department of Industrial Engineering and Operations Research, University of California, Berkeley, USA, and EDF R\&D, Paris, France. Email: 
{\it matteo.basei@edf.fr}} \qquad 
Haoyang Cao\footnote{Department of Industrial Engineering and Operations Research, University of California, Berkeley, USA. Email: {\it cindy\_haoyang\_cao@berkeley.edu}}\qquad
Xin Guo\footnote{Department of Industrial Engineering and Operations Research, University of California, Berkeley, USA. Email: {\it xinguo@berkeley.edu}}
}
\date{\today}

\begin{document}

\maketitle

\begin{abstract}
\noindent We consider a general class of nonzero-sum $N$-player stochastic games with impulse controls, where players control the underlying dynamics with discrete interventions. We adopt a verification approach and provide sufficient conditions for the Nash equilibria (NEs) of the game. We then consider the limit situation of $N \to \infty$, that is, a suitable mean-field game (MFG) with impulse controls. We show that under appropriate technical conditions, the existence of unique NE solution to  the MFG, which is an $\epsilon$-NE approximation to the $N$-player game, with $\epsilon=O\left(\frac{1}{\sqrt{N}}\right)$. As an example, we analyze in details a class of two-player stochastic games which extends the classical cash management problem to the game setting. In particular, we present numerical analysis for the cases of the single player, the two-player game,  and the MFG, showing the impact of competition on the player's optimal strategy, with sensitivity analysis of the model parameters.
	\medskip \\
	
\end{abstract}
\vspace{0.5cm}

\section{Introduction}
\label{sec:introduction}
Recent development in the theory of Mean-Field Games (MFGs) has seen an exponential growth in studies of nonzero-sum stochastic games, especially in analyzing $N$-player games and their MFGs counterpart. This paper focuses on nonzero-sum stochastic $N$-player games and the corresponding MFGs in an  impulse control setting.

\paragraph{A motivating game problem.} The impulse control game is motivated by  the classical cash management problem in the seminal work \cite{Constantinides1978}.  Instead of a single-player game as in  \cite{Constantinides1978}, now consider $N$ players, each managing a flow of cash balance. For player $i \in \{1,\dots,N\}$, the uncontrolled cash balance is driven by 
\begin{equation*}
dX^i_t=b_i(X^i_{t-}) dt+\sigma_i(X^i_{t-})dW^i_t,\qquad X^i_{0-}=x_i,
\end{equation*}
where $W^i$ are independent real Brownian motions. Each player, say player $i$, chooses a sequence of random (stopping) times $(\tau_{i,1}, \tau_{i,2}, \cdots, \tau_{i,k}, \cdots)$ to intervene and exercise her control.
At each $\tau_{i,k}$,  the time of this player's $k$-th intervention, her control is denoted as $\tilde \xi_{i,k}$.  
Given the sequence $\{(\tau_{i,k}, \tilde \xi_{i,k})\}_{k\geq 1}$ for player $i$,  the dynamic of $X^i$ becomes
\begin{equation*}
dX^i_t=b_i(X^i_{t-}) dt+\sigma_i(X^i_{t-})dW^i_t+\sum_{\tau_{i,k}\le t}\delta(t-\tau_{i,k}) \tilde \xi_{i,k},\qquad X^i_{0-}=x_i,
\end{equation*}
with $\delta(\cdot)$ the Dirac function. The payoff for player $i$ is 
\begin{equation}\label{eqn-Nplayer}
\eee_x\left[\int_0^{\infty}e^{-rt}f_{i}(X_t)dt+\sum_{k=1}^{\infty}e^{-r\tau_{i,k}}\phi_i(\tilde \xi_{i,k})+\sum_{j\neq i}\sum_{k=1}^{\infty}e^{-r\tau_{j,k}}\psi_{i,j}(\tilde \xi_{j,k})\right].
\end{equation}
Here $X_t=(X_t^1, \cdots, X_t^N)$ with $x=(x_1, \cdots, x_N)$ the starting state, $r>0$ the discount rate, $f_i$  the running cost, $\phi_i$  the cost of control  for player $i$, and $\psi_{i,j}$  the cost for player $i$ incurred from player $j$'s control, subject to appropriate conditions to be specified in Section \ref{2-player}.
 The  goal of each player $i$ is to find the best policy to minimize her cost among a set of admissible game strategies, to be defined in Section \ref{sec:DefGames}.  

 \paragraph{Our work.}

Motivated by the above game example, this paper analyzes a class of stochastic games with impulse controls, for both $N$ players and its corresponding MFGs. For the $N$-player game, it establishes a general form of Quasi-Variational Inequalities (QVIs) and  provides the sufficient conditions for the Nash equilibria (NEs) of the game, via the verification theorem approach. For the corresponding MFG, it presents sufficient conditions for the existence of NEs and shows that the solution of the MFG is an $\epsilon$-NE approximation to the $N$-player game, with $\epsilon=\frac{1}{\sqrt{N}}$. Through sensitivity analysis and comparisons among $N=1, 2$ and $N=\infty$ (i.e., MFG), it analyzes the cash management game problem and the effect of competition in games and the collapse of MFG to the single-player game. In particular, it shows that in a game setting players have to take the opponents' strategies into consideration due to competition. Consequently, it is optimal (in the NE sense) that players choose to intervene less frequently; but once set to intervene, players will exert larger amount of controls. In some sense, competition induces more efficient control strategies from players.

Compared to games with regular controls and singular controls, impulse control is a more natural mathematical framework for applied problems allowing for a discontinuous state space. See  the examples of  cash management \cite{Constantinides1978}, inventory controls \cite{Harrison1983, HT1983, Sulem1986}, transaction cost in portfolio analysis \cite{Eastham1988, Morton1995, Korn1998, Korn1999, Bielecki2000, ksendal2002}, insurance model \cite{Jeanblanc-Picque1995, Candenillas2006}, liquidity risk \cite{LyVath2007}, exchange rates \cite{Mundaca1998, Bertola2016},  and real options \cite{Triantis1990, Mauer1994,BensoussanRoignant}. The presence of discontinuity makes the analysis of impulse control problems hard and even harder for stochastic games. From a PDEs perspective, the corresponding Hamilton-Jacobi-Bellman equation system is coupled with an additional non-local operator for which most PDEs techniques are not applicable. Indeed, there had been little progress in the theory of impulse controls after Bensoussan and Lions' classical work \cite{Bensoussan1982}, until the work of \cite{Guo2010} where
the non-local operator was found to be connected with the infinitesimal differential operator in the nonlinear PDEs via the payoffs in the action region and the waiting region. (See also  \cite{Bayraktar}.) 

\paragraph{Related work on impulse control games.}
Despite the rapid growth in recent literature on stochastic games and MFGs, few papers have studied the case of stochastic games with impulse controls.  
\cite{Cosso12,Azimzadeh} focus on zero-sum impulse games, \cite{CampiDeSantis} studies mixed impulse-stopping games, \cite{FerrariKoch} analyzes nonzero-sum stochastic games involving impulse controls.
The closest to our work is \cite{Ad2018}, which  studies a class of nonzero-sum {\it two-player} impulse control games. In the example proposed, players control the same {\it one-dimensional} diffusion process.  Interestingly,  their payoffs could be derived without having to deal directly with the non-local operator.  
 
 In comparison, our paper considers a general class of 
 {\it multi-dimensional $N$-player} impulse control games and their MFGs.
 The interaction among players in a discontinuous fashion requires a general mathematical framework, starting from the very definition of game formulation, to admissible game strategies, to the appropriate choices of filtrations, as detailed in Section \ref{sec:DefGames}. The structure of the game payoff also changes when two players do not necessarily control the same dynamics even for the special $N=2$ case: First, one needs to introduce the notion of common waiting region for analyzing general $N$-player games; secondly, one can no longer avoid dealing with the tricky non-local operator in the QVIs for impulse controls. Moreover, 
 \cite{Ad2018} does not consider the MFG of impulse controls and the relation between the MFG and 
 its $N$-player counterpart, nor did they study the impact of game competition.

\paragraph{Related work comparing $N$-player game and MFG.} For an introduction to MFG, we refer to the seminal works \cite{HuangMalhCaines,LasryLions} and to the recent books \cite{CarmonaDelarue1,CarmonaDelarue2}. In the regular control setting where the controls are absolutely continuous with respect to the Lebesgue measure, \cite{Carmona2015} solves the $N$-player game of systemic risk with no common noise and studies the MFG counterpart with common noise; \cite{Bensoussan2016} focuses on a class of linear-quadratic problems; \cite{Lacker2018} analyzes both the $N$-player and MFG  of optimal portfolio problem under the competition criterion, as well as the formulations of both types of games with relative criterion. In the singular control setting where controls are c\`adl\`ag type, \cite{Guo2018} explicitly solves the $N$-player game and MFG of a classical fuel followers problem. The $\epsilon$-NE approximation of MFGs was first established for regular controls in \cite{CarmonaDelarue1,CarmonaDelarue2} with $\epsilon=\frac{1}{N}$ and then for singular controls in \cite{Guo2018} and \cite{GuoJoon2018} with $\epsilon=\frac{1}{\sqrt{N}}$. Our result with $\epsilon=\frac{1}{\sqrt{N}}$ for impulse controls  is consistent with those for singular controls as both allow discontinuity in the state space.   

\section{$N$-Player Stochastic Games with Impulse Controls}
\label{sec:DefGames}

\subsection{Problem Formulation}

In this section, we provide the  mathematical definition for the $N$-player stochastic games with impulse controls. The idea is clear and intuitive: $N$ players intervening on a stochastic process by discrete-time intervention. However, the precise mathematical definition presents some non-trivial technicalities with the presence of discontinuous multi-dimensional controlled process.

\paragraph{Domain and underlying process.} Let ($\Omega$, $\mathcal{F}$, $\{\mathcal{F}_t\}_{t\geq 0}$, $\mathbb{P}$) be a filtered probability space and let $\{W_t\}_{t \geq 0}$ be an $M$-dimensional Brownian motion with natural filtration $\{\mathcal{F}_t\}_{t \geq 0}$. 
Let $S$ be a fixed non-empty subset of $\rr^d$, representing the set where the game takes place, in the sense that the game ends when the controlled process exits from $S$. For example, in portfolio optimization problems the game ends in case of bankruptcy, which may be modelled by choosing $S=(0,\infty)$.

For $t \geq 0$ and $\zeta \in L^2 (\mathcal F_t)$, we denote by $Y^{t,\zeta}=\{Y^{t,\zeta}_s\}_{s \geq t}$ a solution to the stochastic differential equation
\begin{equation}
\label{SDE}
\begin{cases}
d Y^{t,\zeta}_s = b(Y ^{t,\zeta}_s) ds + \sigma(Y ^{t,\zeta}_s) dW_s, \quad s \geq t,  \\
Y^{t,\zeta}_t=\zeta.
\end{cases}
\end{equation}
Here, $b: S \to \rr^d$ and $\sigma: S \to \rr^{d \times M}$ are given Lipschitz-continuous functions, i.e., there exists a constant $K>0$ such that for all $y_1,y_2 \in S$,
\begin{equation*}
|b(y_1) - b(y_2)| + |\sigma(y_1)-\sigma(y_2)| \le K|y_1-y_2|.
\end{equation*}
The equation in \eqref{SDE} models the underlying process when none of the players intervenes. {Since we are going to stop the process as soon as it exits from $S$, in our framework it is enough to have the functions $b$ and $\sigma$ defined only on $S$.}

\paragraph{Interventions of the players and impulse controls.} $N$ players, indexed by $i \in \{1,\dots, N\}$, can intervene on the process in \eqref{SDE} by means of discrete-time interventions. Namely, if player $i$ intervenes with impulse $\delta \in Z_i$, where $Z_i$ is a fixed subset of $\rr^{l_i}$, the process is shifted from state $x$ to state $\Gamma^i(x, \delta)$, where $\Gamma^i : S \times Z_i \to S$ is a given function. In most applied settings, the process shifts with a simple translation, i.e., $\Gamma^i(x, \delta) = x+\delta$. 

The interventions of player $i$ are described by the sequence $\{ (\tau_{i,k}, \xi_{i,k}) \}_{k\geq 1}$ (impulse control), where $\{ \tau_{i,k} \}_{k \geq 1}$ represent the intervention times and $\{ \xi_{i,k} \}_{k \geq 1}$ the corresponding amount of adjustment. {Mathematically, $\tau_{i,k}$ is a stopping time with respect to a suitable filtration $\{\widetilde {\mathcal{F}}_t\}_{t \geq 0}$ (see Remark \ref{RemFiltration} below for details), with $\tau_{i,k+1}\geq\tau_{i,k}$, and $\xi_{i,k}$ is a $\widetilde {\mathcal{F}}_{\tau_{i,k}}$-measurable variable, for each $k \geq 1$ and $i \in \{1,\dots,N\}$}. 

Intervening has a cost or a gain, both for the acting player and for all her opponents. 
Namely, if $x$ is the current state and player $i$ intervenes with an impulse $\delta$, her cost is $\phi_i(x,\delta)$, whereas the cost for player $j \neq i$ is $\psi_{j,i}(x,\delta)$, for given functions $\phi_i,\psi_{j,i}: S \times Z_i \to \rr$. 
For the game to be well defined,  it is necessary to have $\phi_i>0$. That is, intervening corresponds to a cost, otherwise the game degenerates and the players could improve their payoff by continuously intervening. 

\paragraph{Action regions, impulse functions, strategies.}  As seen, players' interventions on the underlying process are modelled by impulse controls.  In the model we propose here, impulse controls originate from a precise strategy that each player preliminarily fixes. 
 
\begin{defn}\label{defPhi}
	A strategy for player $i \in \{1,\dots, N\}$ is a pair $\vphi_i =(A_i,\xi_i)$, where $A_i$ is a fixed {closed} subset of $\rr^d$ (action region) and $\xi_i: S \to Z_i$ is a continuous function (impulse function). We denote by $\Phi_i$ the set of strategies for player $i$. 
\end{defn}

Strategies determine the behaviour of the players, as follows. Fix a starting point $x \in S$ and an $N$-tuple of strategies $\vphi= (\vphi_1, \dots, \vphi_N)$, where $\vphi_i=(A_i,\xi_i) \in \Phi_i$ is the strategy of player $i$ and the sets $A_i$ are pairwise disjoint, that is, $A_i\cap A_j=\emptyset$ for $i\neq j$. Then, $N$ impulse controls $\{ (\tau_{i,k}^{x;\vphi}, \xi^{x;\vphi}_{i,k}) \}_{k \geq 1}$ (the players' interventions), a right-continuous process $X^{x;\vphi}$ (the controlled process), a stopping time $\tau_S^{x;\vphi}$ (the end of the game) are uniquely defined by the following two rules.
\begin{enumerate}
	\item Player $i$ intervenes if and only if the process enters the set $A_i$, in which case the impulse is given by $\xi_i(y)$, where $y$ is the current state. Recall that choosing $\xi_i(y)$ as the intervention impulse means that player $i$ shifts the process from state $y$ to state $\Gamma^i(y,\xi_i(y))$, as introduced earlier.
	\item The game ends when the process exits from $S$.
\end{enumerate}
More precisely,  $\{ (\tau_{i,k}^{x;\vphi}, \xi^{x;\vphi}_{i,k}) \}_{k \geq 1}$, $X^{x;\vphi}$, $\tau_S^{x;\vphi}$ are defined in the following Definition \ref{controls}, where we use the conventions $\inf \emptyset = \infty$ and $[\infty,\infty)=\emptyset$.

\begin{defn}
	\label{controls}
	Let $x \in S$ and $\vphi =(\vphi_1,\dots,\vphi_N)$, where $\vphi_i =(A_i,\xi_i) \in \Phi_i$ is a strategy for player $i\in \{1,\dots,N\}$. Assume that $A_i\cap A_j=\emptyset$, for $i\neq j$. For $k \in \{0, \dots, \bar k\}$, where $\bar k = \sup \{k \in \nn \cup \{0\} : \widetilde \tau_k < \alpha^S_k \}$, we define by induction $\widetilde\tau_0=0$, $x_0=x$, $\widetilde X^0 = Y^{ \widetilde\tau_{0}, x_{0}}$, $\alpha_0^S =\infty$, and
	\begin{align*}
	& \alpha^O_k = \inf \{ s >  \widetilde\tau_{k-1} : \widetilde X^{k-1}_s \notin O \},  && \text{{\footnotesize[exit time from $O \subseteq S$]}} \\	
	&  \widetilde\tau_k = \min \{ \alpha^{A_1}_k, \dots, \alpha^{A_N}_k \},  && \text{{\footnotesize[intervention time]}} \\
	& m_k = \mathbbm{1}_{\{\widetilde\tau_k=\alpha^{A_1}_k\}} + \dots + N \mathbbm{1}_{\{\widetilde\tau_k=\alpha^{A_N}_k\}}, && \text{{\footnotesize[index of the player interv.~at $ \widetilde \tau_k$]}} \\
	&  \widetilde\xi_k = \xi_{m_k} \big( \widetilde X^{k-1}_{ \widetilde \tau_k} \big), && \text{{\footnotesize[impulse]}} \\
	& x_k=\Gamma^{m_k} \big( \widetilde X^{k-1}_{ \widetilde\tau_k},  \widetilde\xi_k \big), && \text{{\footnotesize[starting point for the next step]}} \\
	& \widetilde X^k= \widetilde X^{k-1} \mathbbm{1}_{[0,  \widetilde \tau_k[} +  Y^{ \widetilde\tau_{k},x_{k}} \mathbbm{1}_{[ \widetilde \tau_k, \infty[}. && \text{{\footnotesize[contr.~process up to the $k$-th interv.]}}
	\end{align*} 
	Let $\bar k_i$ be the number of interventions by player $i \in \{1,\dots,N\}$ before the end of the game, and, in the case where $\bar k_i \neq 0$, let $\eta(i,k)$ be the index of her $k$-th intervention ($1 \leq k \leq \bar k_i$):
	\begin{gather*}
	\bar k_i = {\displaystyle\sum}_{1 \leq h \leq \bar k} \mathbbm{1}_{\{m_h=i\}},
	\qquad\quad
	\eta(i,k) = \min \Big\{ l \in \nn : {\displaystyle\sum}_{1 \leq h \leq l} \mathbbm{1}_{\{m_h=i\}} = k \Big\}.
	\end{gather*}
	Assume now that the times $\{\tilde \tau_k\}_{0\leq k \leq \bar k}$ never accumulate strictly before $\alpha^S_{\bar k}$. That is, we assume that  $\lim_{k \to \bar k}\tilde \tau_{k} = \alpha^S_{\bar k}$ in the event $\{\bar k=+\infty\}$,  with the convention $\alpha^S_\infty= \sup_k \alpha^S_k$. The controlled process $X^{x;\vphi}$ and the exit time $\tau^{x;\vphi}_S$ are defined by 
	\begin{equation*}
	X^{x;\vphi} := \widetilde X^{\bar k}, 
	\qquad\qquad
	\tau_S^{x;\vphi} := \alpha^S_{\bar k} = \inf \{ s \geq 0 : X^{x;\vphi}_s \notin S \},
	\end{equation*}
	with the convention $\widetilde X^{\infty} = \lim_{k \to +\infty}\widetilde X^{k}$. Finally, the impulse controls $\{(\tau^{x;\vphi}_{i,k}, \xi^{x;\vphi}_{i,k})\}_{k \geq 1}$, with $i \in \{1, \dots, N \}$, are defined by
	\begin{equation}
	\label{times}
	\tau^{x;\vphi}_{i,k} :=  
	\begin{cases}
	\widetilde \tau_{\eta(i,k)}, & k \leq \bar k_i,
	\\
	\tau^{x;\vphi}_S, & k > \bar k_i,
	\end{cases} 
	\qquad\qquad
	\xi^{x;\vphi}_{i,k} :=  
	\begin{cases}
	\widetilde \xi_{\eta(i,k)}, & k \leq \bar k_i,
	\\
	0, & k > \bar k_i.
	\end{cases}  
	\end{equation}
\end{defn} 

Notice that, if player $i$ intervenes a finite number of times, i.e., $\bar k_i=\bar k_i(\omega)$ is finite, then the tail of the control is conventionally set to $(\tau_{i,k}, \xi_{i,k}) = (\tau_S, 0)$ for $k > \bar k_i$. 
The following lemma characterizes precisely the controlled process $X^{x;\vphi}$.  
\begin{lemma}
	\label{lemmaprocess}
	Let $x \in S$ and $\vphi =(\vphi_1,\dots,\vphi_N)$, where $\vphi_i =(A_i,\xi_i) \in \Phi_i$ is a strategy for player $i\in \{1,\dots,N\}$. Let $X=X^{x;\vphi}$, $\tau_S =\tau^{x;\vphi}_S$, $\tau_{i,k}=\tau^{x;\vphi}_{i,k}$, $\xi_{i,k}=\xi^{x;\vphi}_{i,k}$ be as in Definition \ref{controls}, for $i\in \{1,\dots,N\}$ and $k \geq 1$. Then, 
	\begin{itemize}
		\item[-] $X$ admits the following representation, with $\widetilde \tau_k$, $x_k$ as in Definition \ref{controls} and $Y$ as in \eqref{SDE}:
		\begin{equation}
		\label{representation}
		X_s = \sum_{k=0}^{\bar k -1} Y^{ \widetilde \tau_k,x_k}_s \mathbbm{1}_{ [ \widetilde \tau_k, \widetilde \tau_{k+1}[}(s) + Y^{ \widetilde \tau_{\bar k},x_{\bar k}}_s \mathbbm{1}_{ [ \widetilde \tau_{\bar k}, \infty[ }(s).
		\end{equation}	
		\item[-] $X$ is right-continuous. More precisely, $X$ is continuous in $[0,\infty) \, \setminus \, \{\tau_{i,k} : \tau_{i,k} < \tau_S \}$ and discontinuous in $\{\tau_{i,k} : \tau_{i,k} < \tau_S \}$, where 
		\begin{equation}
		\label{prop1}
		X_{\tau_{i,k}} = \Gamma^i \big(X_{(\tau_{i,k})^-}, \xi_{i,k} \big),
		\qquad
		\xi_{i,k} = \xi_i \big(X_{(\tau_{i,k})^-} \big),
		\qquad
		X_{(\tau_{i,k})^-} \in \partial A_i.
		\end{equation}
		\item[-] $X$ never exits from the set $(A_1 \cup \dots \cup A_N)^c$.
	\end{itemize}
\end{lemma}

\begin{proof}
	We just prove the first property in \eqref{prop1}, the other ones being immediate. Let $i \in \{1,\dots,N\}$, $k \geq 1$ with $\tau_{i,k}< \tau_S$ and set $\sigma = \eta(i,k)$, with $\eta$ as in Definition \ref{controls}. By \eqref{times}, \eqref{representation} and Definition \ref{controls}, we have
	\begin{multline*}
	X_{\tau_{i,k}} 
	= X_{ \widetilde \tau_{\sigma}} 
	= Y ^{ \widetilde\tau_\sigma, x_\sigma}_{ \widetilde\tau_\sigma} = x_{\sigma}
	= \Gamma^i \big( \widetilde X^{\sigma-1}_{  \widetilde\tau_{\sigma} } ,  \widetilde\xi_\sigma \big) 
	\\
	= \Gamma^i \big( \widetilde X^{\sigma-1}_{  ( \widetilde\tau_{\sigma})^- } ,  \widetilde\xi_\sigma \big) 
	= \Gamma^i \big( X_{  ( \widetilde\tau_{\sigma})^- } ,  \widetilde\xi_\sigma \big)
	= \Gamma^i \big( X_{  (\tau_{i,k})^- } , \xi_{i,k} \big),
	\end{multline*}
	where the fifth equality is by the continuity of the process $\widetilde X ^{\sigma-1}$ in $[ \widetilde \tau_{\sigma-1}, \infty)$ and the next-to-last equality follows from $\widetilde X ^{\sigma-1} \equiv X$ in $[0,  \widetilde\tau_{\sigma})$. 
\end{proof}

\begin{remark}
\label{RemFiltration}
For $x \in S$ and $\vphi \in \Phi_x$, let $\{\mathcal{F}^{X}_{t}\}_{t \geq 0}$ denote the natural filtration of the process $X=X^{x;\vphi}$. Then, by construction, $\tau_{i,k}$ is a stopping time with respect to the filtration $\{\mathcal{F}^{X}_{t^-}\}_{t \geq 0}$ and $\xi_{i,k}$ is a $\mathcal{F}^{X}_{\tau_{i,k}}$-measurable random variable, for $i \in \{1,\dots,N\}$ and $k \in \nn$. 
\end{remark}

\begin{remark}
In single-player impulse control problems (e.g., \cite{ksendal2002}), the optimal intervention times are recursively defined by 
\begin{equation}
\label{simple}
\tau_{k+1}= \inf \{s \geq \tau_{k} : X_t^{k} \in A\},
\end{equation}
for a suitable set $A$, where $X^{k}$ represents the controlled process after the $k$-th intervention. Notice that this procedure cannot be directly extended to $N$-player impulse games: In a game setting, the intervention times of player $i$ also depend on her opponents' past interventions, so that \eqref{simple} would not be well defined in this case. To overcome this technical difficulty and provide a rigorous framework, {we have introduced} the definition of strategy.
\end{remark}

\paragraph{Objective functions.} Each player aims at minimizing her objective function, made up of four terms: a continuous-time running cost in $[0,\tau_S]$, the discrete-time costs associated to her own interventions, the discrete-time costs associated to her opponents' interventions, a terminal cost if the game ends. 

More precisely, let $f_i, h_i: S \to \rr^d$ be given functions, and let $\rho_i>0$ be strictly positive constants, for $i \in \{1,\dots,N\}$. For more technical details on the existence and uniqueness of the solution to impulse control problems, see \cite{Guo2010}. The functional that player $i$ aims at minimizing is defined as follows. 

\begin{defn}
	Let $x \in S$ and $ \vphi = (\vphi_1,\dots,\vphi_N)$ be a $N$-tuple of strategies. For $i\in \{1,\dots,N\}$, provided that the right-hand side exists and is finite, we set 
	\begin{multline}
	\label{defJ}
	J^i(x;\vphi) := 
	\eee_x \bigg[ \int_0^{\tau_S} e^{-\rho_i s} \, f_i(X_s) ds 
	+ \sum_{\substack{k \in \nn \\ \tau_{i,k} < \tau_S}} e^{-\rho_i \tau_{i,k}} \, \phi_{i} \big( X_{(\tau_{i,k})^-}, \xi_{i,k} \big)
	\\
	+ \sum_{\substack{1 \leq j \leq N \\ j \neq i}} \sum_{\substack{k \in \nn \\ \tau_{i,k} < \tau_S}} e^{-\rho_i \tau_{j,k}} \, \psi_{i,j} \big( X_{(\tau_{j,k})^-}, \xi_{j,k} \big)
	+ e^{-\rho_i \tau_S} \, h_i (X_{\tau_S})\mathbbm{1}_{ \{ \tau_S < +\infty \} } \bigg],
	\end{multline}
	with $X = X^{x;\vphi}$, $\tau_S = \tau^{x;\vphi}_S$, $\{(\tau_{i,k}, \xi_{i,k})\}_{k \geq 1} = \{(\tau_{i,k}^{x;\vphi},\xi^{x;\vphi}_{i,k})\}_{k \geq 1}$ as in Definition \ref{controls}.
\end{defn}

The subscript in the expectation denotes, as in control theory, conditioning with respect to starting point $X_t^{x;\vphi}=x$. To shorten the notations, we will often omit the initial state  and write $\eee$. Also, notice that in the summations we only consider times strictly smaller than $\tau_S$: indeed, since the game ends in $(\tau_S)^-$, interventions in the form $\tau_{i,k}=\tau_S$ are meaningless for the game. 

\paragraph{Admissible strategies and Nash equilibria.} Before defining a Nash equilibrium (NE)  for the game, we define, for each starting point $x \in S$, the set $\Phi_x$ of admissible strategies, i.e., strategies as in Definition \ref{defPhi} with additional properties assuring that the game is well defined.
\begin{defn}
	\label{AdmStrat}
	Let $x \in S$ and $\vphi_i = (A_i,\xi_i)$ be a strategy for player $i \in \{1, \dots, N\}$. We say that the $N$-tuple $\vphi=(\vphi_1, \dots, \vphi_N)$ is $x$-admissible, written as $\vphi \in \Phi_x$, if: 
	\begin{enumerate}
		\item the sets $A_1,\dots,A_N$ are pairwise disjoint, that is, $A_i \cap A_j = \emptyset$ for $i \neq j$;
		\item for $i,j \in \{1,\dots,N\}$ with $i \neq j$, the following random variables are in $L^1(\Omega)$:
		\begin{equation}
		\label{L1}
		\begin{gathered}
		\int_0^{\tau_S} e^{-\rho_i s} |f_i|(X_s)ds, 
		\qquad
		e^{-\rho_i \tau_S}|h_i|(X_{\tau_S}),
		\\
		\sum_{ \tau_{i,k} < \tau_S } e^{-\rho_i \tau_{i,k}} |\phi_{i}| ( X_{(\tau_{i,k})^-}, \xi_{i,k} ),
		\qquad
		\sum_{ \tau_{i,k} < \tau_S } e^{-\rho_i \tau_{i,k}} |\psi_{i,j}| ( X_{(\tau_{j,k})^-}, \xi_{j,k} );
		\end{gathered}
		\end{equation}
		\item for each $i \in \{1,\dots,N\}$ and $p \in \nn$, the random variable $\|X\|_\infty = \sup_{t\geq0}|X_t|$ is in $L^p(\Omega)$:
		\begin{equation}
		\label{L1process}
		\eee[\|X\|^p_\infty] < \infty;
		\end{equation}
		\item for $i \in \{1,\dots, N\}$, we have
		\begin{equation}
		\label{LimImp}
		\lim_{k \rightarrow + \infty} \tau_{i,k} = \tau_S.
		\end{equation}
	\end{enumerate}
\end{defn}

The first condition in Definition \ref{AdmStrat} \textcolor{black}{ensures that there are no conflicts due to two or more players willing to intervene at same time (see Remarks \ref{rmk:disjoint} and \ref{rmk:disjoint2} below for further comments on this condition).} The second condition assures that the functionals $J^i(x;\vphi)$ in \eqref{defJ} are well-defined, for each $i \in \{1,\dots,N\}$. The third condition will be used in the proof of the verification theorem where sufficient conditions for the NEs are specified. Finally, the fourth condition prevents the players from accumulating the interventions before the end of the game.

We now provide the definition of NE and payoffs. Given a tuple of strategies $\vphi= (\vphi_1, \dots, \vphi_N)$, an index $i \in \{1,\dots,N\}$ and a strategy $\bar\vphi \in \Phi_i$, we denote by $(\vphi^{-i},\bar\vphi)$ the $N$-tuple we get when substituting the $i$-th component of $\vphi$ by $\bar\vphi$, that is 
\begin{equation*}
{(\vphi^{-i},\bar\vphi) := (\vphi_1, \dots, \vphi_{i-1}, \bar\vphi, \vphi_{i+1}, \dots, \vphi_N).}
\end{equation*}
\begin{defn}
	\label{DefNash}
	Given $x \in S$, we say that the admissible $N$-tuple of strategies $\vphi^* \in \Phi_x$ is a NE of the game if
	\begin{equation*}
	J^i(x;\vphi^*) \leq J^i(x;(\vphi^{*,-i},\vphi_i)),  
	\end{equation*}
	for each $i \in \{1,\dots,N\}$ and each $\vphi_i \in \Phi_i$ such that $(\vphi^{*,-i},\vphi_i) \in \Phi_x$. Finally, if $x \in S$ and $\vphi^* \in \Phi_x$ is a NE, then the payoff associated with the equilibrium $\vphi^*$ for player $i \in \{1,\dots,N\}$ is
	\begin{equation*}
	V_i(x) := J^i(x;\vphi^*).
	\end{equation*}	
\end{defn}

\begin{remark}\label{rmk:disjoint}
{If the action regions are not pairwise disjoint (i.e., two or more players would like to intervene at the same time), one sets specific rules deciding which player has the priority. For example, player $i$ may have priority over player $j$ whenever $i>j$; otherwise, priority may be given to the player who is the farthest away from the state he would shift the process to.
	
\textcolor{black}{Several formulation are possible to handle priorities among players.	
Our formulation is based on the idea that,} since priority rules practically partition the conflict regions, it is not restrictive to assume that the action regions are pairwise disjoint, as we now detail.

Let $\tilde A_1, \dots ,\tilde A_N$ denote the action regions before any priority rules is set, so that the sets $\tilde A_i$ are possibly non-disjoint. Let $C$ denote the region where two or more players would like to simultaneously intervene, $C:= \cup_{i \neq j} (\tilde A_i \cap \tilde A_j)$. Deciding which player has the priority corresponds to choosing a partition $C_1,\dots,  C_N$ of $C$, with the additional property that $C_i \subseteq \tilde A_i$: namely, if a point belongs to the conflict region, $x \in C$, then it is player $i$ who intervenes, where $i$ is the only index such that $x \in C_i$. The actual action regions $A_i$ are then defined by $A_i = (\tilde A_i \setminus C) \cup C_i$, clearly pairwise disjoint. Hence, whatever the priority rule is, we get a $N$-uple of pairwise disjoint action regions, so that condition 
$1$ in Definition \ref{AdmStrat} is not restrictive.

The advantage of this formulation is twofold. On  one hand, no specific priority rules are embedded in the model, which provides more flexibility with respect to, e.g., the approach in \cite{Ad2018} \textcolor{black}{(if $N=2$, choosing $A_1 = \tilde A_1$ and $A_2 = \tilde A_2 \setminus \tilde A_1$ in our setting retrieves the priority rules in \cite{Ad2018})}. On the other hand, this helps relieving the notational burden in the paper, since we do not have to deal with multiple intersections of the action regions when defining the controlled process. 
}
\end{remark} 

\begin{remark}\label{rmk:disjoint2}
\textcolor{black}{We remark that the present setting allows two or more players to intervene, one right after the other, at a same instant $t \geq 0$. For example, if the present state is $X_{t-}=x \in A_{i_1}$, then player $i_1$ intervenes and move it to $x'$. If $x'$ happens to be in $A_{i_2}$, for some $i_2$, then player $i_2$ immediately intervenes, moving again the state to $x''$. Overall, the state jumped from $X_{t-}= x$ to $X_{t}= x''$. In order to have a well-defined process, only finitely many players can intervene in $t$, which is guaranteed by condition \eqref{LimImp} above. }
\end{remark} 

\subsection{Verification Theorem}

In this section we establish a verification theorem for the games defined in Section \ref{sec:DefGames}, providing sufficient conditions to determine the payoffs and an NE. This verification theorem  links  the impulse games with a suitable system of quasi-variational inequalities (QVI). Note that a special case of this verification theorem for $N=2$ was presented in \cite{Ad2018}. 

In Section \ref{ssec:Qvi} we heuristically introduce the system of QVIs, providing the intuition behind each equation involved. These arguments are made rigorous in Section \ref{ssec:Verif1}, with the precise statement and proof of the verification theorem.

\subsubsection{The Quasi-Variational Inequalities}
\label{ssec:Qvi}

We start by heuristically guessing an expression for a NE $\vphi^*=(\vphi^*_1,\dots,\vphi^*_N)$ and for the corresponding payoffs $V_i$ of the game.

Consider a game as in Section \ref{sec:DefGames}. Assume for a moment that the payoffs $V_i$, $i \in\{1,\dots,N\}$ are known. Moreover, assume that for every $i$ there exists a (unique) function $\xi_i:S \to Z_i$ such that 
\begin{equation}
\label{def:deltai}
\{\xi_i(x)\} = \arg\min_{\delta \in Z_i} \big\{ V_i(\Gamma^i(x,\delta))+\phi_i(x,\delta)\big\}, 
\end{equation}
for each $x \in S$. We define the intervention operators by
\begin{equation}
\label{MH}
\begin{aligned}
\mm_i V_i(x) &= V_i\big(\Gamma^i(x,\xi_i(x))\big) + \phi_i \big(x,\xi_i(x)\big), 
\\
\hh_{i,j} V_i(x) &= V_i \big(\Gamma^j(x,\xi_j(x))\big) + \psi_{i,j} \big(x,\xi_j(x)\big), 
\end{aligned}
\end{equation}
for $x \in S$ and $i,j \in \{1,\dots,N\}$, with $i \neq j$. 

The functions in \eqref{def:deltai} and \eqref{MH} are intuitive. If $x$ is the current state of the process, and player $i$ (resp.~player $j$) intervenes with impulse $\delta$, the payoff for player $i$ can be represented as $V_i(\Gamma^i(x,\delta)) + \phi_i(x,\delta)$ (resp.~$V_i(\Gamma^j(x,\delta)) + \psi_{i,j}(x,\delta)$), that is, as the sum of the intervention cost and the payoff in the new state. As a consequence, $\xi_i(x)$ in \eqref{def:deltai} is the impulse that player $i$ would use in case she decides to intervene. Similarly, $\mm_iV_i(x)$ (resp.~$\hh_{i,j}V_i(x)$) represents the payoff for player $i$ when player $i$ (resp.~player $j \neq i$) takes the best immediate action and behaves optimally afterwards. 

Notice that it is not always optimal to intervene, so $\mm_iV_i(x) \geq V_i(x)$, for each $x \in S$, and that player $i$ should intervene (with impulse $\xi_i(x)$) only if $\mm_iV_i(x) = V_i(x)$. As a consequence, provided that an explicit expression for $V_i$ is available, an NE is heuristically given by $\vphi^*=(\vphi^*_1,\dots,\vphi^*_N)$, where $\vphi_i^*=(A_i^*,\xi_i^*)$ is given, for each $i \in \{1,\dots,N\}$, by
\begin{equation*}
A^*_i = \{ \mm_iV_i - V_i = 0 \}, \qquad\quad \xi^*_i = \xi_i.
\end{equation*}
Practically, this means that player $i$ intervenes when the process enters the region $\{ \mm_iV_i - V_i = 0 \}$, i.e, when $\mm_i V_i(x) = V_i(x)$. When this happens, her impulse is $\xi_i(x)$, where $x$ is the current state. The verification theorem in the next section will give a rigorous proof to this heuristic argument. 

To complete the argument, we need to determine the payoffs $V_i$. Assume that $V_i$ are smooth enough so that we can define 
\begin{equation}
\label{operatorL}
\aaa V_i = b \cdot \nabla V_i + \frac{1}{2} \mbox{tr} \left( \sigma \sigma^t D^2 V_i\right), 
\end{equation}
where $b, \sigma$ are as in \eqref{SDE}, $\sigma^t$ denotes the transpose of $\sigma$ and $\nabla V_i, D^2 V_i$ are the gradient and the Hessian matrix of $V_i$, respectively. Then $V_i$ should satisfy the following quasi-variational inequalities (QVIs), where $i,j \in \{1, \dots, N\}$:
\begin{subequations}
	\label{pbNEW}
	\begin{empheq}[left = \empheqlbrace]{align}
	& V_i = h_i,  && \text{in} \,\,\, \partial S, \label{pbNEW-1} \\
	& \mm_jV_j - V_j \geq 0, && \text{in} \,\,\, S, \label{pbNEW-2}  \\
	& \hh_{i,j}V_i-V_i=0, && \text{in} \,\,\, {\textstyle \bigcup_{j \neq i}} \{\mm_jV_j - V_j = 0\}, \label{pbNEW-3}  \\
	& \min\big\{\aaa V_i -\rho_i V_i + f_i, \mm_iV_i-V_i \}=0, && \text{in} \,\,\, {\textstyle \bigcap_{j \neq i}} \{\mm_jV_j - V_j > 0\}  \label{pbNEW-4}  .
	\end{empheq}
\end{subequations}
Notice that there is a small abuse of notation in \eqref{pbNEW-1}, as $V_i$ is not defined in $\partial S$, so that \eqref{pbNEW-1} means $\lim_{y \to x} V_i(y) = h_i(x)$, for each $x \in \partial S$.

Intuition behind each in \eqref{pbNEW}:  the terminal condition \eqref{pbNEW-1} is obvious, and \eqref{pbNEW-2}, already stated above, is a standard condition in impulse control theory. As for \eqref{pbNEW-3}, if player $j$ intervenes (i.e.,~$\mm_jV_j - V_j =0$), by the definition on NE, we expect no losses for player $i \neq j$, that is $\hh_{i,j}V_i - V_i =0$. Meanwhile, if all the players except $i$ are not intervening (hence, $\mm_jV_j - V_j >0$ for all $j \neq i$), then player $i$ faces a classical one-player impulse control problem, so that $V_i$ satisfies the corresponding QVI of $\min\big\{\aaa V_i -\rho_i V_i + f_i, \mm_iV_i-V_i \}=0$, which is \eqref{pbNEW-4}. In short, the latter condition says that $ \aaa V_i-\rho_i V_i + f_i = 0$ when she does not intervene, whereas $\aaa V_i-\rho_i V_i + f_i \geq 0$ when she intervenes.

\begin{remark}\label{rmk:qvi}
For any player $i$, the region where she chooses not to intervene, as in $\eqref{pbNEW-4}$ when $\min\{\aaa V_i-\rho_iV_i+f_i,\mm_iV_i-V_i\}=\aaa V_i-\rho_iV_i+f_i=0$, is decided by not just player $i$ but all $N$ players; it is indeed the common non-action region $C$. On $C$, it is necessary to have $\aaa V_i-\rho_iV_i+f_i=0$ for all $i\in\{1,\dots, N\}$. The condition that $\mm_iV_i-V_i\geq 0$, however, needs an extra verifying step: it is not entirely player $i$'s decision to wait, yet this choice has to be the best one she can make at a NE. This marks the subtlety of the NE and one crucial difference between the single-player control problem and the multi-player game.
\end{remark}

\subsubsection{Statement and Proof}
\label{ssec:Verif1}

We now provide a rigorous proof of the results heuristically introduced in the previous section. Notations and assumptions from Section \ref{sec:DefGames} are adopted from now on.

\begin{thm}[\textbf{Verification Theorem}] 
	\label{thm:verification}
	Let $V_1,\dots,V_N$ be functions from $S$ to $\rr$, assume that $$\{\xi_i(x)\} = \arg\min_{\delta \in Z_i} \big\{ V_i(\Gamma^i(x,\delta))+\phi_i(x,\delta)\big\}$$ holds and set $\mathcal D_i := \{ \mm_i V_i - V_i > 0 \}$. Moreover, for $i \in \{1,\dots,N\}$ assume that:
	\begin{itemize}
		\item[(i)] $V_i$ is a solution to \eqref{pbNEW-1}-\eqref{pbNEW-4};
		\item[(ii)] $V_i \in C^2(\cap_{j \neq i} \mathcal D_j \setminus \partial \mathcal D_i) \cap C^1(\cap_{j \neq i} \mathcal D_j) \cap C(\overline{\cap_{j \neq i} \mathcal D_j})$ and it has polynomial growth; 
		\item[(iii)] $\partial \mathcal D_i $ is a Lipschitz surface, and $V_i$ has locally bounded derivatives up to the second order in some neighbourhood of $\partial \mathcal D_i$.
	\end{itemize}
	Finally, let $x \in S$ and define $\vphi^*=(\vphi_1^*,\dots,\vphi_N^*)$, with 
	\begin{equation*}
	\vphi_i^*:=(A^*_i,\xi^*_i), \qquad A^*_i:=\{ \mm_i V_i - V_i = 0 \}, \qquad \xi^*_i:=\xi_i,
	\end{equation*}
	where $i \in \{1,\dots,N\}$ and the function $\xi_i$ is as in \eqref{def:deltai}. Then, provided that $\vphi^* \in \Phi_x$, 
	\begin{equation*}
	\mbox{$\vphi^*$ is an NE and $V_i(x)=J^i(x;\vphi^*)$ for $i \in \{1,\dots,N\}$.}
	\end{equation*} 	
\end{thm}

\begin{proof}
	Let $x \in S$, $i \in \{1,\dots,N\}$ and $\vphi_i \in \Phi_i$ such that $(\vphi^{*,-i},\vphi_i) \in\Phi_x$. Notice that $(\vphi^{*,-i},\vphi_i)$ corresponds to the case where all the players except player $i$ behave optimally. By Definition \ref{DefNash}, we have to prove that
	\begin{equation*}
	V_i(x)= J^i(x;\vphi^*), 
	\qquad 
	V_i(x) \leq J^i(x;(\vphi^{*,-i},\vphi_i)). 
	\end{equation*}
	
	\textit{Step 1: $V_i(x) \leq J^i(x;(\vphi^{*,-i},\vphi_i))$.} To simplify the notations, we omit the dependence on $i, x, \vphi$ and write
	\begin{equation}
	\label{shortnotations}
	X=X^{x;(\vphi^{*,-i},\vphi_i)}, \qquad\quad 
	\tau_{j,k}=\tau^{x;(\vphi^{*,-i},\vphi_i)}_{j,k}, \qquad\quad 
	\xi_{j,k}=\xi^{x;(\vphi^{*,-i},\vphi_i)}_{j,k}.
	\end{equation}	
	The properties in Lemma \ref{lemmaprocess} imply that, for $j \neq i$, $s \geq 0$, $\tau_{j,k}< \infty$, 
	\begin{subequations}
		\label{prop2}
		\begin{align}
		& (\mm_j V_j - V_j) \big( X_s \big) >0, \label{prop2-M} \\
		& (\mm_j V_j - V_j) \big( X_{ ( \tau_{j,k})^-} \big) =0, \label{prop2-tau} \\
		& \xi_{j,k} = \xi_j \big( X_{ (\tau_{j,k})^-} \big). \label{prop2-d}
		\end{align}
	\end{subequations}
{We first approximate $V_i$ with regular functions. Since (ii) and (iii) hold, by \cite[proof of Thm.~10.4.1 and App.~D]{ksendalSDE} there exists a sequence of functions $\{V_{i,m}\}_{m \in \nn}$ such that:
	\begin{itemize}
		\item[(a)] $V_{i,m} \in C^2(\cap_{j \neq i} \mathcal D_j) \cap C^0(\overline{\cap_{j \neq i} \mathcal D_j})$ for each $m \in \nn$ (in particular, the function $\aaa V_{i,m}$ is well-defined in $\cap_{j \neq i} \mathcal D_j$);
		\item[(b)] $V_{i,m} \to V_i$ as $m \to \infty$, uniformly on the compact subsets of $\overline{\cap_{j \neq i} \mathcal D_j}$;
		\item[(c)] $\{\aaa V_{i,m}\}_{m \in \nn}$ is locally bounded in $\cap_{j \neq i} \mathcal D_j$ and $\aaa V_{i,m} \to \aaa V_i$ as $m \to \infty$, uniformly on the compact subsets of $\cap_{j \neq i} \mathcal D_j \setminus \partial \mathcal{D}_i$.
	\end{itemize}
For each $r >0$ and $\ell \in \nn$, we set
\begin{equation}
\label{stopping}
\tau_{r,\ell} = \tau_r \land \tau_{1,\ell} \land \dots \land \tau_{N,\ell},
\end{equation}
	where $\tau_r = \inf \{ s > 0 : X_s \notin B(0,r) \} $ is the exit time from the ball with radius $r$. By \eqref{prop2-M} we have that $X_s \in \cap_{j \neq i} \mathcal D_j$ for each $s > 0$. Since $V_{i,m} \in C^2(\cap_{j \neq i} \mathcal D_j)$ by (a), for each $m \in \nn$ we can apply It\^o's formula to the process $e^{-\rho_i t}V_{i,m}(X_t)$ over the interval $[0, \tau_{r,\ell})$. Taking the conditional expectations, we get
	\begin{multline}
	\label{itoAPPROX}		
	V_{i,m}(x) = \eee_x \bigg[ - \int_{0}^{\tau_{r,\ell}} e^{-\rho_i s} (\aaa V_{i,m} - \rho_iV_{i,m}) (X_s) ds - \sum_{\tau_{i,k}< \tau_{r,\ell}} e^{-\rho_i \tau_{i,k}} \Big( V_{i,m} \big( X_{\tau_{i,k}} \big) - V_{i,m} \big( X_{ ( \tau_{i,k} )^-} \big) \Big)
	\\
	- \sum_{j\neq i} \sum_{\tau_{j,k}< \tau_{r,\ell}} e^{-\rho_i \tau_{j,k}} \Big( V_{i,m} \big( X_{\tau_{j,k}} \big) - V_{i,m} \big( X_{ ( \tau_{j,k} )^-} \big) \Big) + e^{-\rho_i \tau_{r,\ell}}V_{i,m}(X_{(\tau_{r,\ell})^-}) \bigg].
	\end{multline}	
	Notice that \eqref{itoAPPROX} is well defined: since $\tau_{r,\ell} \leq \tau_r$, $X$ belongs to the compact set $\overline{B(0,r)}$, where the continuous function $V_{i,m}$ is bounded; moreover, the two summations consist in a finite number of terms since $\tau_{r,\ell} \leq \tau_{i,\ell}$ for each $i \in \{1,\dots,n\}$. Also, notice that in \eqref{itoAPPROX} we need to write $V_{i,m}(X_{(\tau_{r,\ell})^-})$, as we have a jump at time $\tau_{r,\ell}$. We now pass to the limit in \eqref{itoAPPROX} as $m \to \infty$: since $X$ belongs to the compact set $\overline{B(0,r)}$, by the uniform convergence in (b) and (c) we get}
	\begin{multline}
	\label{ito}		
	V_i(x) = \eee_x \bigg[ - \int_{0}^{\tau_{r,\ell}} e^{-\rho_i s} (\aaa V_i - \rho_iV_i) (X_s) ds - \sum_{\tau_{i,k}< \tau_{r,\ell}} e^{-\rho_i \tau_{i,k}} \Big( V_i \big( X_{\tau_{i,k}} \big) - V_i \big( X_{ ( \tau_{i,k} )^-} \big) \Big)
	\\
	- \sum_{j\neq i} \sum_{\tau_{j,k}< \tau_{r,\ell}} e^{-\rho_i \tau_{j,k}} \Big( V_i \big( X_{\tau_{j,k}} \big) - V_i \big( X_{ ( \tau_{j,k} )^-} \big) \Big) + e^{-\rho_i \tau_{r,\ell}}V_i(X_{(\tau_{r,\ell})^-}) \bigg].
	\end{multline}	
	We now estimate each term in the right-hand side of \eqref{ito}. As for the first term, since $(\mm_jV_j - V_j)(X_s) > 0$  for each $j \neq i$ by (\ref{prop2-M}), from (\ref{pbNEW-4}) it follows that 
	\begin{equation}
	\label{magg1}
	(\aaa V_i-\rho_iV_i) (X_s) \geq -f_i(X_s),
	\end{equation}
	for all $s \in [0,\tau_S]$. Let us now consider the second term: by (\ref{pbNEW-2}) and the definition of $\mm_i V_i$ in \eqref{MH}, for every stopping time $\tau_{i,k}<\tau_S$ we have 
	\begin{align}
	\label{magg2}
	V_i \big( X_{ ( \tau_{i,k} )^-} \big) & \leq \mm_i V_i \big( X_{ ( \tau_{i,k} )^-} \big)  \nonumber \\
	& = \inf_{\delta \in Z_i} \big\{ V_i \big( \Gamma^i \big( X_{ ( \tau_{i,k} )^-} , \delta \big) \big) + \phi_i \big( X_{ ( \tau_{i,k} )^-} , \delta \big) \big\} \nonumber \\
	& \leq V_i \big( \Gamma^i \big( X_{ ( \tau_{i,k} )^-} , \xi_{i,k} \big) \big) + \phi_i \big( X_{ ( \tau_{i,k} )^-} , \xi_{i,k} \big) \nonumber \\
	& = V_i \big( X_{ \tau_{i,k} } \big) + \phi_i \big( X_{ ( \tau_{i,k} )^-} , \xi_{i,k} \big).
	\end{align}
	As for the third term, let us consider any stopping time $\tau_{j,k}<\tau_S$, with $j \neq i$. By \eqref{prop2-M} we have $(\mm_j V_j - V_j) \big( X_{ ( \tau_{j,k} )^-} \big) = 0$; hence, the condition in (\ref{pbNEW-3}), the definition of $\hh_{i,j} V_i$ in \eqref{MH} and the expression of $\xi_{j,k}$ in \eqref{prop2-d} imply that
	\begin{align}
	\label{magg3}
	V_i \big( X_{ ( \tau_{j,k} )^-} \big) & = \hh_{i,j} V_i \big( X_{ ( \tau_{j,k} )^-} \big) \nonumber \\ 
	& = V_i \big( \Gamma^j \big( X_{ ( \tau_{j,k} )^-} , \xi_{j} \big( X_{ ( \tau_{j,k} )^-}) \big) \big) + \psi_{i,j} \big( X_{ ( \tau_{j,k} )^-} , \xi_{j} \big( X_{ ( \tau_{j,k} )^-})\big) \nonumber \\
	& = V_i \big( \Gamma^j \big( X_{ ( \tau_{j,k} )^-} , \xi_{j,k} \big) \big) + \psi_{i,j} \big( X_{ ( \tau_{j,k} )^-} , \xi_{j,k} \big) \nonumber \\
	& = V_i \big( X_{  \tau_{j,k} } \big) + \psi_{i,j} \big( X_{ ( \tau_{j,k} )^-} , \xi_{j,k} \big). 
	\end{align}
	By \eqref{ito} and the estimates in \eqref{magg1}-\eqref{magg3} it follows that
	\begin{multline*}
	V_i(x) \leq \eee_x \bigg[ \int_{0}^{\tau_{r,\ell}} e^{-\rho_i s} f_i(X_s) ds + \sum_{\tau_{i,k}< \tau_{r,\ell}} e^{-\rho_i \tau_{i,k}} \phi_i \big( X_{ ( \tau_{i,k} )^-} , \xi_{i,k} \big)
	\\
	+ \sum_{j\neq i}\sum_{\tau_{j,k}< \tau_{r,\ell}} e^{-\rho_i \tau_{j,k}} \psi_{i,j} \big( X_{ ( \tau_{j,k} )^-} , \xi_{j,k} \big) + e^{-\rho_i \tau_{r,\ell}} V_i(X_{\tau_{r,\ell}}) \bigg]. 
	\end{multline*}
	Thanks to the conditions in \eqref{L1} and \eqref{L1process} together with the polynomial growth of $V_i$ in (ii), we now use the dominated convergence theorem and pass to the limit, first as $r \to \infty$ and then as $\ell \to \infty$, {so that the stopping times $\tau_{r,\ell}$ converge to $\tau_S$ by \eqref{LimImp}.} In particular, for the fourth term we notice that {by (ii) and \eqref{L1process} we have }
	\begin{equation}
	\label{domcon}
	V_i(X_{(\tau_{r,\ell})^-}) \leq C(1+|X_{(\tau_{r,\ell})^-}|^p) \leq C(1+\|X\|_\infty^p) \in L^1(\Omega),
	\end{equation}
	for suitable constants $C>0$ and $p \in \nn$. Therefore, the corresponding limit for the fourth term immediately follows by the continuity of $V_i$ in the case $\tau_S < \infty$ and by \eqref{domcon} itself in the case $\tau_S = \infty$ (as a direct consequence of \eqref{L1process}, we have $\|X\|^p_\infty < \infty$ a.s.). Hence, 
	\begin{multline*}
	V_i(x) \leq \eee_x \bigg[ \int_{0}^{\tau_{S}} e^{-\rho_i s} f_i(X_s) ds + \sum_{\tau_{i,k}< \tau_{S}} e^{-\rho_i \tau_{i,k}} \phi_i \big( X_{ ( \tau_{i,k} )^-} , \xi_{i,k} \big)
	\\
	+ \sum_{j \neq i} \sum_{\tau_{j,k}< \tau_{S}} e^{-\rho_i \tau_{j,k}} \psi_{i,j} \big( X_{ ( \tau_{j,k} )^-} , \xi_{j,k} \big) + e^{-\rho_i \tau_{S}} h_i(X_{(\tau_{S})^-})\mathbbm{1}_{ \{ \tau_{S} < +\infty \} } \bigg] = J^i(x;(\vphi^{*,-i},\vphi_i)).
	\end{multline*}
	
	\textit{Step 2: $V_i(x) = J^i(x;\vphi^*)$.} Similar as in Step 1, except that  all the inequalities are equalities by the properties of $\vphi^*$. 
\end{proof}

\subsection{Example: Two-Player Cash Management Game}
\label{2-player}

Now let us revisit the cash management game in Section \ref{sec:introduction}, using the notations introduced in Section \ref{sec:DefGames}. We here consider the two-player game: $N=2$, $b_i=0$, $\sigma_i=\sigma>0$. The uncontrolled cash level of the two players $X_t = (X^1_t,X^2_t)$ is 
 $${dX^i_t=\sigma dW^i_t, \qquad X^i_{0-}=x_i},$$ 
 where $W$ is a two-dimensional standard Brownian motion and $x \in \rr^2$. Let
 $$
 \vphi=(\vphi_1,\vphi_2), \qquad \vphi_1=(A_1,\xi_1), \qquad  \vphi_2=(A_2,\xi_2), 
 $$
 denote the strategies of the players for this game, as in Definition \ref{defPhi}. Since player $i \in \{1,2\}$ intervenes by shifting her own component $X^i$ of the cash level, we have 
 $$
 \Gamma^i(x,\delta)=x+\delta, \qquad
 Z_1=\{(\delta_1,0): \delta_1 \in\rr\}, \qquad
Z_2=\{(0,\delta_2): \delta_2 \in \rr\}.
$$
This means that player 1 (resp.~player 2) intervenes by moving the process from state $(x_1,x_2)$ to state
$$
\xi_1(x_1,x_2)=\big(x_1+ \tilde\xi_1(x_1,x_2), \,\,x_2\big) \qquad \Big(\text{resp. } \xi_2(x_1,x_2)=\big(x_1, \,\,x_2 +\tilde \xi_2(x_1,x_2)\big) \,\Big),
$$
for suitable functions $\tilde \xi_i$. Notice that, as a consequence,  the controlled process $X_t=(X^1_t,X^2_t)$ satisfies
 \begin{equation*}
 dX^i_t=\sigma dW^i_t+\sum_{\tau_{i,k}\le t}\delta(t-\tau_{i,k})\tilde \xi_{i,k}, \qquad X^i_{0-}=x_i,
 \end{equation*}
where 
$$
\tilde \xi_{i,k} = \tilde \xi_i \Big(X^1_{(\tau_{i,k})^-},\, X^2_{(\tau_{i,k})^-}\Big).
$$

Let now $i,j\in\{1,2\}$ with $j\neq i$. The cost function for player $i$ under the control policy $\vphi=(\vphi_1,\vphi_2)$ is given by
\begin{equation*}
J^i(x;\vphi)=\eee_x\left[\int_0^{\infty}e^{-rt}f_i(X_t)dt+\sum_{k\geq1}e^{-r\tau_{i,k}}\phi_i(\xi_{i,k})+\sum_{k\geq1}e^{-r\tau_{j,k}}\psi_{i,j}(\xi_{j,k})\right],
\end{equation*}
where 
\begin{empheq}[left=\empheqlbrace]{align*}
&f_i(x)=h\left|x_i-\frac{1}{N}\sum_{j=1}^Nx_j\right|,\quad x\in\rn^2,N=2,\\
&\phi_i(\xi)=K+k|\xi|,\quad\xi\in\rn,\\
&\psi_{i,j}(\xi)= c,\quad\xi\in\rn,
\end{empheq}
for positive constants $h,K,k,c$. The goal of player $i$ is to minimize the cost $J^i$: we are interested in finding $\vphi^*=(\vphi_1^*,\vphi_2^*)$ such that Definition \ref{DefNash} holds.

By the symmetry of the problem structure, we seek for an NE where the action regions take the form of 
$$A_1=\{x:x_1-x_2\geq u\} ,\qquad A_2=\{x:x_2-x_1\geq u\}$$ 
for some $u>0$, with appropriate impulse functions such that
$$\xi_1(x)=(U-x_1+x_2, \, 0),\qquad \xi_2(x)=(0, \, U-x_2+x_1)$$ 
for some $U<u$. Recall that this means that player 1 (resp.~player 2) intervenes when $X^1_t-X^2_t \geq u$ (resp.~$X^2_t-X^1_t \geq u$) and shifts her component so as to have $X^1_t-X^2_t = U$ (resp.~$X^2_t-X^1_t = U$). Note that $A_1\cap A_2=\emptyset$ and that $C=\{x:-u<x_1-x_2<u\}$ is the common waiting region, i.e., where no player intervenes. 

By the same symmetry argument, we look for payoffs in the form of
$$V_i(x_1,x_2)=w_i(x_i-x_j),\quad i,j\in\{1,2\},\quad i\neq j,$$
for some functions $w_i$. In this case, player 1 and player 2 are indistinguishable, therefore it suffices to study the payoff of player 1. Now the waiting region for player 1 is $D_1=\{x:x_1-x_2<u\}$, and define $D_{-1}=\{x:x_2-x_1<u\}=\{x:x_1-x_2>-u\}$. By the corresponding QVI and the regularity requirement in the Verification Theorem \ref{thm:verification}, the function $w_1:\mathbb{R}\to\mathbb{R}$ need to satisfy the following system of equations and inequalities:
\begin{subequations}
\begin{empheq}[left=\empheqlbrace]{align}
&w_1(s)=\begin{cases}
w_1(u)+k(s-u),&s\geq u;\\
\frac{h_2}{r}s+c_1e^{\lambda_2s}+c_2e^{-\lambda_2s}, &0\leq s\leq u;\\
-\frac{h_2}{r}s+\left(c_1+\frac{h_2}{r\lambda_2}\right)e^{\lambda_2s}+\left(c_2-\frac{h_2}{r\lambda_2}\right)e^{-\lambda_2s},&-u\leq s\leq 0;\\
w_1(-u),&s\leq -u;
\end{cases}\label{eq:vfn-form}\\
&\dot{w}_1(u)=\dot{w}_1(U)=k;\label{eq:vfn-ac-c1}\\
&w_1(u)=w_1(U)+K+k(u-U);\label{eq:vfn-ac-c0}\\
&w_1(-u)=w_1(-U)+c;\label{eq:vfn-na-c0}\\
&\m w_1(s) - w_1(s)\geq 0,\quad \forall s\in D_{-1};\label{eq:vfn-nonloc}
\end{empheq}
\end{subequations}
where $h_2=\frac{h}{2},\hspace{1pt}\sigma_2=\sqrt{2}\sigma,\hspace{1pt}\lambda_2=\frac{\sqrt{2r}}{\sigma_2}$ and $(c_1,c_2,u,U)$ remain to be determined. Accordingly, $w_2(s) = w_1(-s)$ for any $s\in\mathbb{R}$. 

Now, similar argument as in \cite{Constantinides1978} shows that when $h_2-rk>0$, $c>0$, there exists a solution $w_1$ to \Crefrange{eq:vfn-form}{eq:vfn-na-c0} satisfying $c_1<0$, $c_2>0$, $0<U<u$\label{thm:ext-ne-1}. Moreover, if such solution $w_1$ as above satisfies \eqref{eq:vfn-nonloc}, then an NE to the cash management problem  $\vphi^*=(\vphi_1^*,\vphi_2^*)$ is characterized by
\begin{equation}\label{eq:2p-ne-1}\tag{NE-1}
\begin{cases}
A_1^*=\{x:x_1-x_2\geq u\},\quad A_2^*=\{x:x_2-x_1\geq u\},\\
\xi_{1}^*(x)=(U-x_1+x_2, \, 0), \quad
\xi_{2}^*(x)=(0, \, U-x_2+x_1).\\
\end{cases}
\end{equation} 
The corresponding payoffs are given by $$V_1(x_1,x_2)=w_1(x_1-x_2),\quad V_2(x_1,x_2)=w_1(x_2-x_1).$$
The optimality of \eqref{eq:2p-ne-1} can be easily verified by checking the conditions in the Verification Theorem \ref{thm:verification}.

Figure \ref{subfig:2p-vfn-id} shows the equilibrium payoffs for both players if they adopt the control policy specified in \eqref{eq:2p-ne-1}. Figure \ref{subfig:2p-ctrl-id} illustrates the control policy, with $h=2$, $K=3$, $k=1$, $r=0.5$, $\sigma = \frac{\sqrt{2}}{2}$ and $c = 1$,  where the thresholds can be solved as $U=0.686$ and $u=5.658$, with $c_1=-0.003$ and $c_2=1.972$. 
\begin{figure}[H]
\centering
\begin{subfigure}[b]{0.45\textwidth}
\includegraphics[width=\textwidth]{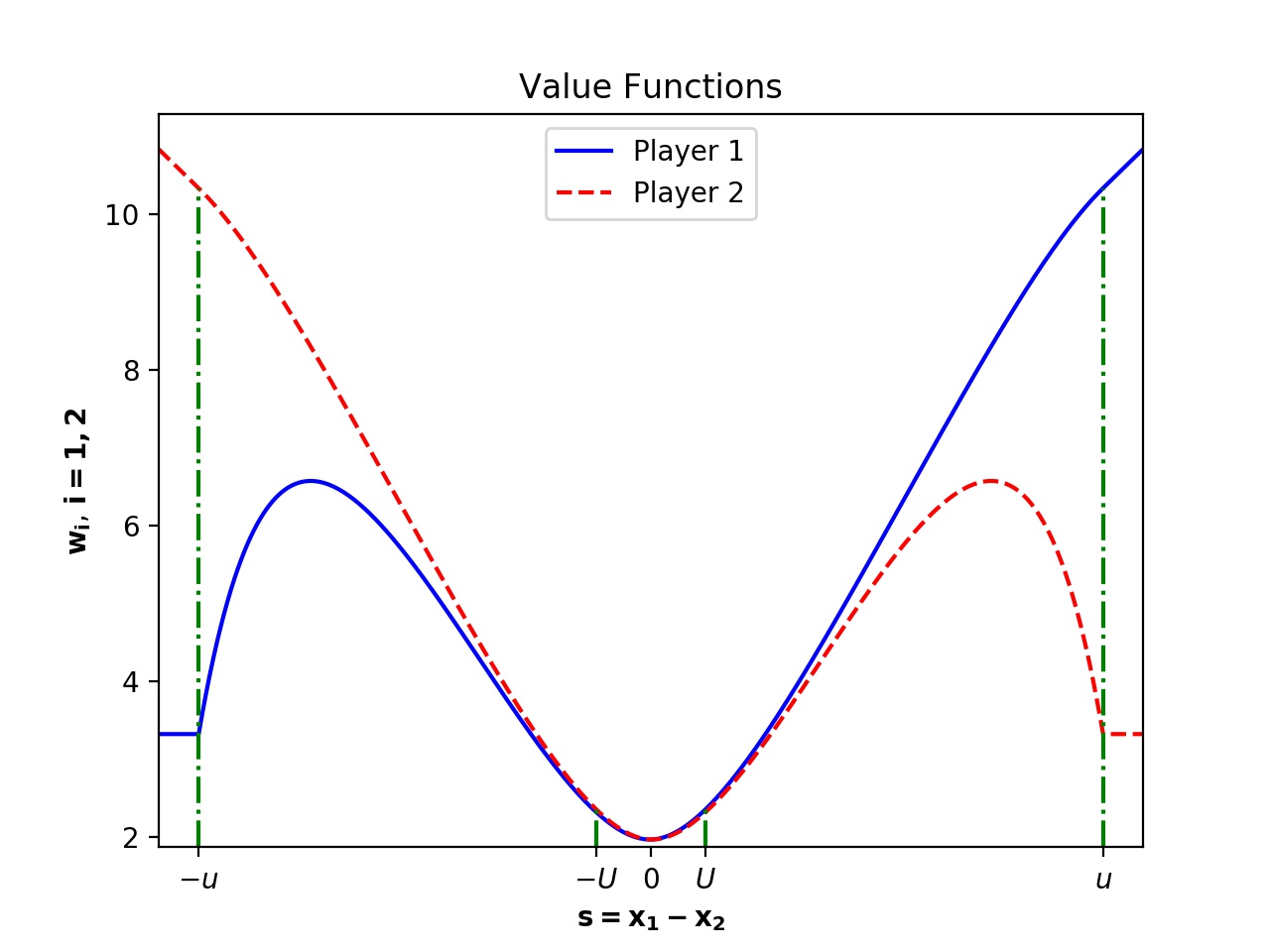}
\caption{Payoffs after change of variables.}
\label{subfig:2p-vfn-id}
\end{subfigure}
\begin{subfigure}[b]{0.45\textwidth}
\centering
\includegraphics[width=0.8\textwidth]{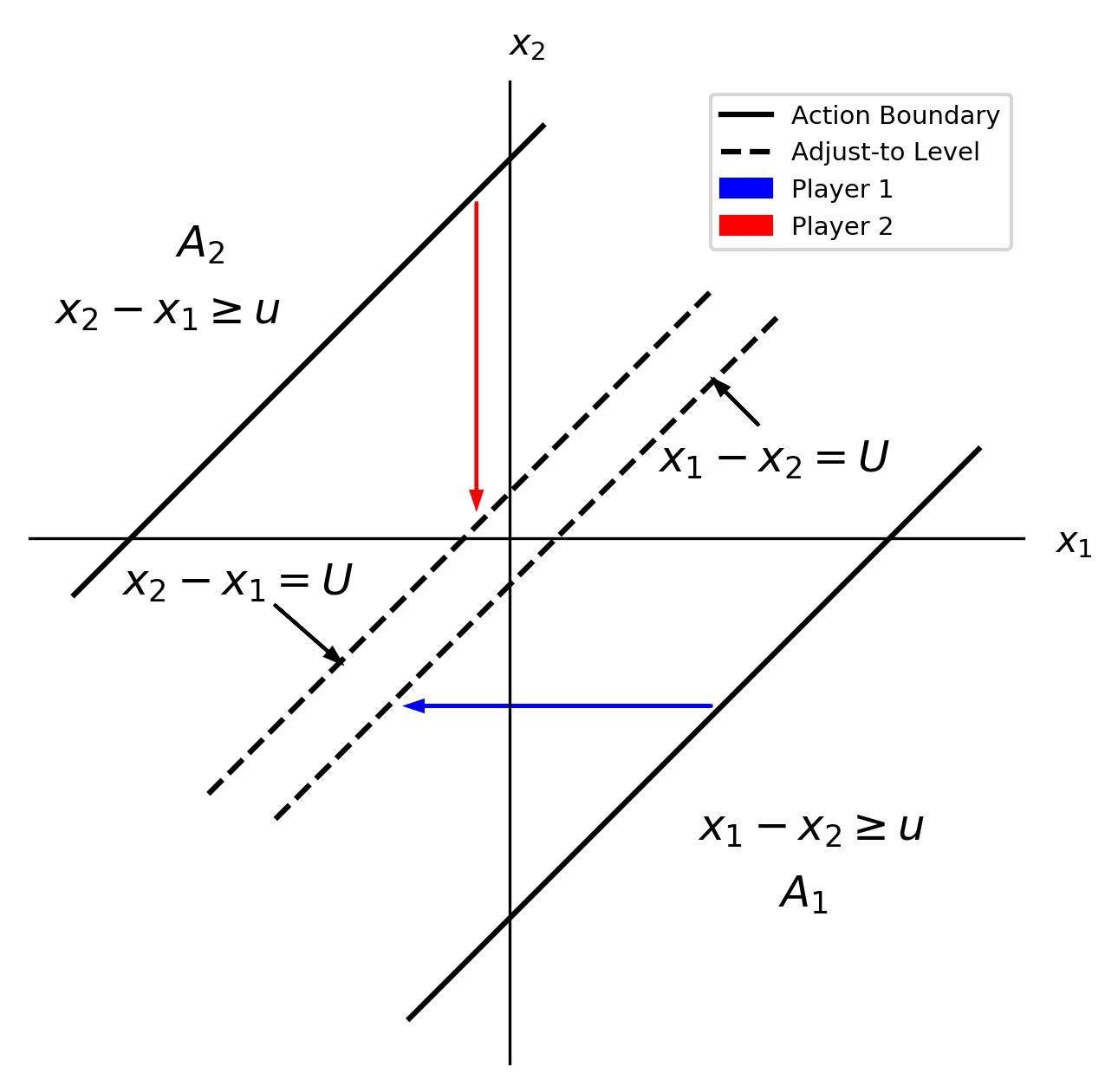}
\caption{Control policy.}
\label{subfig:2p-ctrl-id}
\end{subfigure}
\caption{A Nash equilibrium and the payoffs}
\end{figure}

\paragraph{Multiple NEs.}
In general, NE for nonzero-sum games may not be unique.  In this example, an alternative NE can be derived by switching action regions between the two players. For instance, if  player 1 is to dictate the game whereas player 2 is a complete follower, the action region for player 1 can be characterized by $A_1=\{x\in\rn^2:|x_1-x_2|>u\}$, and  $A_2=\emptyset$. That is, let $s=x_1-x_2$ then $V_1(x)=w_1(x_1-x_2)$ and $V_2(x)=w_2(x_1-x_2)$, where
\begin{subequations}
\begin{empheq}[left=\empheqlbrace]{align}
&w_1(s)=\begin{cases}
w_1(u)+k(s-u),&s\geq u;\\
\frac{h_2}{r}s+c_1e^{\lambda_2s}+\left(c_1+\frac{h_2}{r\lambda_2}\right)e^{-\lambda_2s}, &0\leq s\leq u;\\
w_1(-s),&s\leq 0;
\end{cases}\label{eq:vfn-dic-fm}\\
&\dot{w}_1(u)=\dot{w}_1(U)=k;\label{eq:vfn-dic-c1}\\
&w_1(u)=w_1(U)+K+k(u-U);\label{eq:vfn-dic-c0}\\
&\mm w_1(s) - w_1(s)\geq 0,\quad s\in\rn;\label{eq:vfn-dic-nonloc}
\end{empheq}
\end{subequations}
\begin{subequations}
\begin{empheq}[left=\empheqlbrace]{align}
&w_2(s)=\begin{cases}
w_2(u),&s\geq u;\\
\frac{h_2}{r}s+c_2e^{\lambda_2s}+\left(c_2+\frac{h_2}{r\lambda_2}\right)e^{-\lambda_2s}, &0\leq s\leq u;\\
w_2(-s),&s\leq 0;
\end{cases}\label{eq:vfn-flw-fm}\\
&w_2(u)=w_2(U)+c;\label{eq:vfn-flw-c0}\\
&\mm w_2(s) - w_2(s)\geq 0,\quad -u\leq s\leq u.\label{eq:vfn-flw-nonloc}
\end{empheq}
\end{subequations}
Now, assume that $h_2-rk>0$, $c>0$, then again one can show that there exists a solution $w_1$ satisfying \Crefrange{eq:vfn-dic-fm}{eq:vfn-dic-nonloc} with $c_1\in(-\frac{h_2}{r\lambda_2},0)$ as well as $0<U<u$, and $w_2$ satisfying \Crefrange{eq:vfn-flw-fm}{eq:vfn-flw-c0}.
Moreover, if  such solution $w_2$  satisfies \eqref{eq:vfn-flw-nonloc}, then an NE to the cash management problem in Section \ref{2-player} $\vphi^*=(\vphi_1^*,\vphi_2^*)$ is characterized by
\begin{equation}\label{eq:2p-ne-2}\tag{NE-2}
\begin{cases}
A_1^*=\{x:|x_1-x_2|\geq u\},\quad A_2^*=\emptyset;\\
\xi_{1}^*(x)=
\begin{cases}
(U-x_1+x_2, \, 0), & \mbox{if}  \ \ \hspace{1pt}x_1-x_2\geq u,\\
(-U-x_1+x_2, \, 0), & \mbox{if} \ \ \hspace{1pt}x_1-x_2\leq -u.
\end{cases}
\end{cases}
\end{equation} 
Notice that we do not need to define $\xi^*_2$, as player 2 never intervenes.

 Figure \ref{subfig:2p-vfn-dic} shows the payoffs and Figure \ref{subfig:2p-ctrl-dic} demonstrates the NE, 
 under the same values of $h$, $K$, $k$, $r$ and $\sigma$, with thresholds $U=0.993$ and $u=1.999$, and  $c_1=-0.101$ and $c_2=-0.133$.
\begin{figure}[h]
\centering
\begin{subfigure}[b]{0.45\textwidth}
\centering
\includegraphics[width=\textwidth]{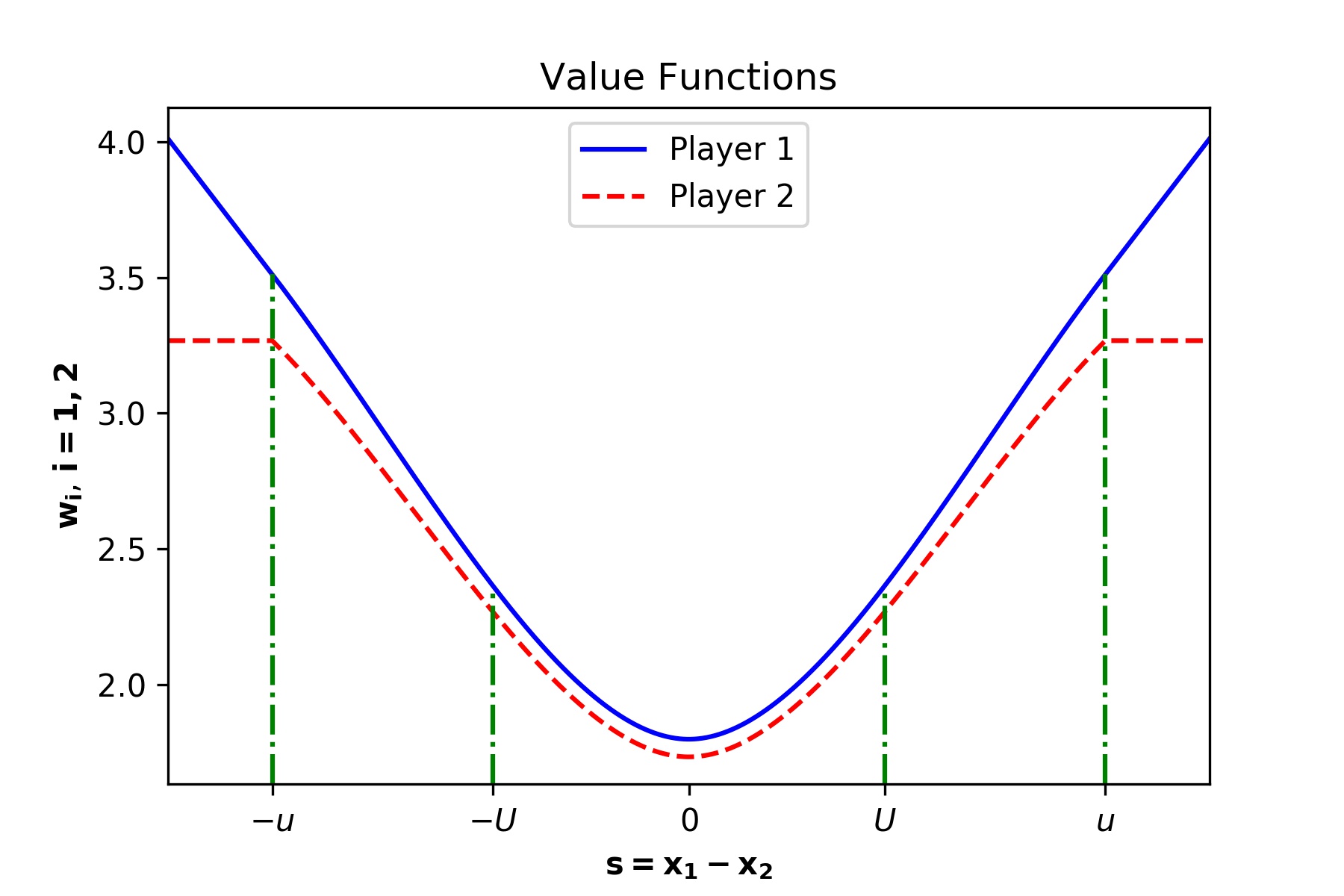}
\caption{Player 1 is the sole controller.}
\label{subfig:2p-vfn-dic}
\end{subfigure}
\begin{subfigure}[b]{0.45\textwidth}
\centering
\includegraphics[width=\textwidth]{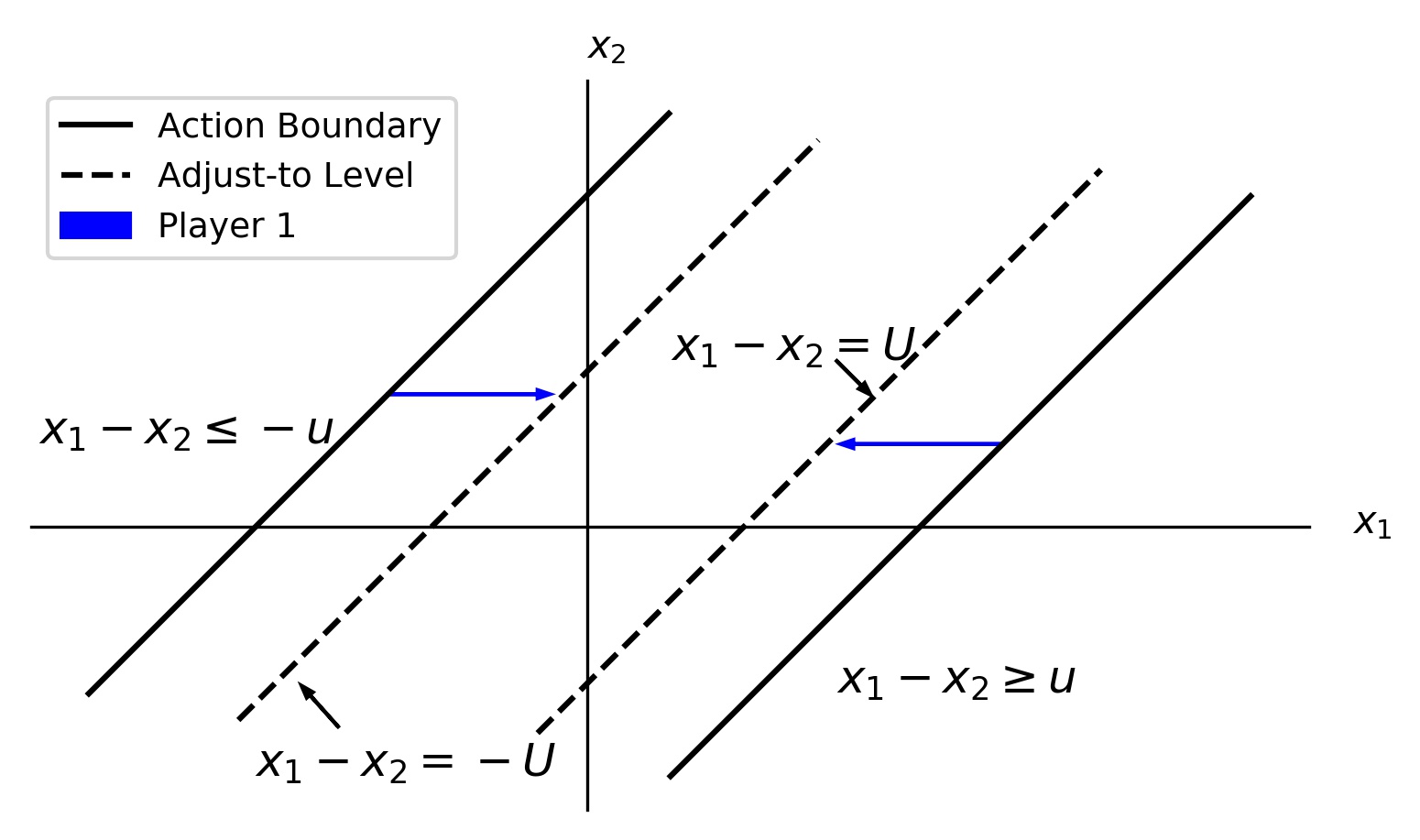}
\caption{Control policy.}
\label{subfig:2p-ctrl-dic}
\end{subfigure}
\caption{Alternative Nash equilibrium and the payoffs}
\end{figure}

{      
\section{Mean-Field Games (MFGs) with Impulse Controls}
As seen from the previous section,  it is  difficult to solve analytically the general $N$-player impulse control game. 
We will now introduce an MFG framework for the impulse control game and show that this MFG provides a reasonable approximation to the $N$-player game. More precisely, we show that under appropriate technical conditions, the existence of unique NE solution to  the MFG, which is an $\epsilon$-NE approximation to the $N$-player game, with $\epsilon=O\left(\frac{1}{\sqrt{N}}\right)$.

\subsection{Formulation of MFGs with impulse controls}
Given the $N$-player stochastic game formulation (\ref{eqn-Nplayer}), its natural MFG formulation goes as follows.
Let $(\Omega, \mf,\{\mf_t\}_{t\geq0},\pbm)$ be a filtered probability space supporting an $\{\mf_t\}_{t\geq0}$-adapted standard Brownian motion $W$.  Consider an infinite number of rational and indistinguishable players who interact through the cost structure consisting of a running cost $f$ and the cost of control $\phi$. For each player, her uncontrolled state process is given by $$dX_t=b(X_t)dt+\sigma(X_t)dW_t, \quad X_{0-}\sim\mu.$$
Each player seeks for the optimal impulse control policy $\vphi^*$ among the set of admissible impulse controls $\mathcal A$ to minimize the total discounted cost. Controls $\vphi\in\mathcal A$ are here represented by $\vphi =(A,\xi)$, where $A$, a closed subset of $\rn$, is called the action region and $\xi:\rn\to\rn$ is a measurable function. Under the control policy $\vphi$, the state process becomes
$$dX_t=b(X_{t-})dt+\sigma(X_{t-})dW_t+\sum_{n\geq1}\delta(t-\tau_n)\xi_n,\quad X_{0-}\sim\mu,$$
where $\mu$ denotes the initial distribution of the state, and
\begin{align*}
&\tau_1=\inf\{t\geq0:X_{t-}\in A\},\quad \tau_{n}=\inf\{t>\tau_{n-1}:X_{t-}\in A\},\,n\geq2;\\
&\xi_n=\xi(X_{\tau_n-})\in\mf_{\tau_n},\,n\geq1.
\end{align*}
Therefore, the control policy $\vphi$ can also be characterized by a sequence of stopping times and the associated random variables, $\vphi=\{(\tau_n,\xi_n)\}_{n\geq1}$. The optimization problem faced by individual player is given by,
\begin{equation}\label{eq:smfg-gen}\tag{MFG}
\begin{gathered}
V(x|m)=\inf_{\vphi\in\mathcal A}J_{\infty}(x,\vphi|m),\\
J_{\infty}(x,\vphi|m)=\eee_{\mu}\left[\left.\int_0^{\infty}e^{-rt}f(X_t,m)dt+\sum_{n=1}^\infty e^{-r\tau_{n}}\phi(\xi_{n})\right|X_{0-}=x\right],\\
dX_t=b(X_{t-})dt+\sigma(X_{t-})dW_t+\sum_{\tau_n\le t} \delta(t-\tau_{n})\xi_{n},\qquad X_{0-}\sim\mu,
\end{gathered}
\end{equation}
where $m$ denotes the mean information:
\begin{equation*}
m=\limsup_{t \to \infty} \eee_{\mu}[X_t].
\end{equation*}
{\color{black} Finally, for $x,m \in \rr$, we define the function $\mm V(x|m)$ in the usual way:
\begin{equation*}
\mm V(x|m) = \inf_{\delta \in \rr} \, \big\{ V(x+\delta|m) + \phi(\delta) \big\}.
\end{equation*} 
}

Compared to (\ref{eqn-Nplayer}) for $N$-player games,  individual players in an MFG now lose sight of individual opponents, hence the term  $\psi_{i,j}$ has disappeared in the MFG formulation. \textcolor{black}{As for the intervention costs, we chose them in the form $\phi(\xi_n)$ in order to have simpler notations in this section; however, one may also consider intervention costs in the form $\phi(\xi_n, X_{(\tau_n)^-},m)$.}

\begin{defn}\label{def:smfg-sol}
A pair of control policy and mean information $(\vphi^*=(A^*,\xi^*),m^*)$, with $\vphi^*\in\mathcal A$ and $m^*\in\rn$, is said to be a solution to the \eqref{eq:smfg-gen} if 
\begin{itemize}
\item $V(x|m^*) = J_{\infty}(x,\vphi^*|m^*)$,
\item $\displaystyle{m^*=\limsup_{t\to\infty}\eee_{\mu}\left[X^{*}_t\right]}$, where
\begin{equation*}
dX_t^{*}=b(X_{t-}^{*})dt+\sigma(X_{t-}^{*})dW_t+\sum_{\tau_n^*\le t} \delta(t-\tau_n^*)\xi_n^*,\quad X_{0-}^*\sim\mu,
\end{equation*}
such that
\begin{align*}
&\tau_1^*=\inf\{t\geq0:X_{t-}^*\in A^*\},\quad \tau_{n}^*=\inf\{t>\tau_{n-1}^*:X_{t-}^*\in A^*\},\,n\geq2;\\
&\xi_n^*=\xi^*(X_{(\tau^*_{n})^-}^*)\in\mf_{\tau_n^*},\,n\geq1.
\end{align*}
\end{itemize}
\end{defn}

\subsection{Solution to symmetric MFGs with impulse controls}
We first analyze the existence of the solution to (\ref{eq:smfg-gen}), under  some technical assumptions. In particular, we impose symmetry on the dynamics and cost functions.
\begin{description}
	\item[(A1)]The uncontrolled dynamics is symmetric in the sense that 
	\begin{equation*}
		dX_t=\sigma dW_t,\quad X_{0-}\sim \mu,
	\end{equation*}
	where $\sigma>0$ is a constant, and $\mu$, with $\int_\rn |x|\mu(dx)<\infty$, is symmetric around its mean.
	\item[(A2)]The cost of control satisfies
	\begin{empheq}[left=\empheqlbrace]{align}
		&K:=\inf_{\xi\in\rn}\phi(\xi)>0,\\
		&\phi\in\mathcal{C}(\rn\setminus\{0\}),\\
		&\lim_{|\xi|\to\infty}\phi(\xi)=+\infty,\quad\sup_{\xi\in\rn}\frac{\phi(\xi)}{1+|\xi|}\leq k\,\text{ for some }k>0,\\
		&\phi(\xi_1)+\phi(\xi_2)\geq \phi(\xi_1+\xi_2)+K,\quad \forall \xi_1,\xi_2\in\rn,\\
		&\phi(-\xi)=\phi(\xi),\quad\forall\xi\in \rn. \label{eq:sym-ctrl-cost}
	\end{empheq}
	\item[(A3)]The running cost $f:\rn\times\rn\to\rn$, jointly continuous with $f(x,m)\geq 0$, satisfies the following properties, for any fixed $m\in\rn$.
	\begin{description}
		\item[(A3-1)] There exists $C_f=C_f(m)>0$ such that, for each $x,y \in \rr$, $$|f(x,m)-f(y,m)|<C_f|x-y|;$$
		\item[(A3-2)] For each $\delta\in\rn$, $$f(m+\delta,m)=f(m-\delta, m).$$
	\end{description}
	\item[(A4)] \textcolor{black}{For each $m \in \rr$ fixed, the strategy $\tilde\vphi(m) = (A_m,\xi_m)$ defined  by
	\begin{equation}
	\label{AxiMFG}
	A_m= \{ x \in \rr: \mm V(x|m) - V(x|m)  = 0 \}, \qquad \xi_m(x)= \arg\min_{\delta \in \rr}\{ V(x+\delta|m) + \phi(\delta) \},
	\end{equation}
	is admissible and is the unique strategy $\vphi$ such that $V(x|m)=J_{\infty}(x,\vphi|m)$.}
\end{description}

\begin{remark}
\label{remonA4}
\textcolor{black}{We remark that, when $m \in \rr$ is fixed, problem \eqref{eq:smfg-gen} becomes a standard single-player impulse control problem, whose typical solution has the form \eqref{AxiMFG} in Assumption (A4): see \cite{ksendal2007} for an introduction to single-player impulse problems and e.g.~\cite{basei}, \cite{CadenillasLaknerPinedo}, \cite{CadenillasZapatero}, \cite{Korn}, \cite{MitchellFengMuthuraman} and \cite{MundacaOksendal} for some applications having solutions in the form \eqref{AxiMFG}.}
\end{remark}


\begin{thm}\label{thm:sys}
Under Assumptions (A1)--(A4), \eqref{eq:smfg-gen} admits a solution in the sense of Definition \ref{def:smfg-sol}.
\end{thm}

\begin{proof}
Theorem \ref{thm:sys} is proved using a three-step approach (solution for generic $m$, mean information update, fixed point of the composite function). 

First, if the mean information $m$ is given, we solve the corresponding optimal control problem. \textcolor{black}{By (A4), the unique optimal strategy is given by} $\tilde\vphi(m)=(A_m,\xi_m)\in\mathcal A$. Define a mapping from space of mean information $\rn$ to the admissible control set $\mathcal A$, as $$\Gamma_1:\rn\to\mathcal A,\qquad\quad \Gamma_1(m)=\tilde \vphi(m)=(A_m,\xi_m).$$

Next, we update the mean information $m$. Given any $\vphi=(A,\xi)\in\mathcal A$, under Assumption (A1), define the corresponding controlled process
$$dX_t^\vphi=\sigma dW_t+\sum_{n\geq1}\delta(t-\tau_n)\xi_n,$$
where
$$\begin{aligned}
	&\tau_1=\inf\{t\geq0:\,X_{t-}^{\vphi}\in A\},\quad \tau_n=\inf\{t>\tau_{n-1}:\,X^{\vphi}_{t-}\in A\},\,n\geq 2;\\
	&\xi_n=\xi(X^{\vphi}_{\tau_n-}),\,n\geq1.
\end{aligned}$$
Define a mapping from the admissible control set $\mathcal A$ to the extended real line $[-\infty,\infty]$, as $$\Gamma_2:\mathcal A\to [-\infty,\infty],\qquad\quad \Gamma_2(\vphi)=\limsup_{t\to\infty}\eee_\mu\left[X_t^\vphi\right].$$

Finally, define the composite mapping $\Gamma:\rn\to[-\infty,\infty]$, as
$$
\Gamma:m\overset{\Gamma_1}{\mapsto}\tilde\vphi(m)\overset{\Gamma_2}{\mapsto}\limsup_{t\to\infty}\eee_\mu\left[X^{\tilde\vphi(m)}_t\right].
$$
This is where we utilize the symmetric cost structures. \textcolor{black}{By Lemma \ref{symm} below}, the waiting region $D_m=\rn\setminus A_m$ and the optimal control function $\xi_m$ are symmetric with respect to $m$. Given the symmetry and $L^1$ condition of $\mu$ in Assumption (A1), let $m^*=\eee_{\mu}[X_{0-}]\in\rn$, then by symmetry, $\Gamma(m^*)=\limsup_{t\to\infty}\eee_\mu\left[X^{\tilde\vphi(m^*)}_t\right]=m^*$. Therefore, $(\vphi^*=\tilde\vphi(m^*),m^*)$ is a solution to \eqref{eq:smfg-gen} in the sense of Definition \ref{def:smfg-sol}.
\end{proof}

\textcolor{black}{
\begin{lemma}
\label{symm}
Let (A1)-(A4) hold and let $m \in \rr$ be fixed. Then: 
\begin{itemize}
	\item[(i)] the function $V(\cdot \, | m)$ in (\ref{eq:smfg-gen}) is symmetric with respect to $m$, i.e., $V(m+x|m) = V(m-x|m)$ for each $x \in \rr$; 
	\item[(ii)] the continuation region $D_m = \{x \in \rr: \mm V(x|m) - V(x|m) > 0 \}$ is symmetric with respect to $m$, i.e., $m-x \in D_m$ if and only if $m+x \in D_m$, for each $x \in \rr$;
	\item[(iii)] the optimal impulse function $\xi_m = \arg\min_{\delta \in \rr} \{ V ( \,\cdot \, +\delta) + \phi(\delta)\}$ satisfies $\xi_m(m+x) = -\xi_m(m-x)$, for each $x \in \rr$. As a consequence, if $x,x'$ are symmetric with respect to $m$, then $x+\xi_m(x), x'+\xi_m(x')$ are symmetric as well (i.e., optimal interventions preserve symmetry).
\end{itemize}
\end{lemma}
}

\begin{proof}
\color{black}{
(i) By \cite{Guo2010}, the function $V(\,\cdot \, | m)$ is the only viscosity solution to the QVI 
\begin{equation}
\label{visc1}
\min \{ \mathcal{L} V(x|m) - rV(x|m)+ f(x,m), \, \mm V (x|m) - V(x|m) \} = 0, 
\end{equation}
with $\mathcal{L} = \frac{\sigma^2}{2} \frac{d^2}{dx^2}$. We will prove that $x \mapsto V(2m -x | m )$ is a viscosity solution to \eqref{visc1}, so that by uniqueness we get $V(x|m)=V(2m-x|m)$ for each $x \in \rr$, i.e., the claim in (i). To prove that $\tilde V(x) := V(2m-x|m)$ is a viscosity subsolution to \eqref{visc1} (the supersolution argument is similar), let us consider $x_0 \in \rr$ and $\tilde \vphi \in C^2(\rr)$ such that $\tilde V - \tilde \vphi$ has a local maximum at $x_0$ and $(\tilde V - \tilde \vphi)(x_0)=0$. If we set $x_1=2m-x_0$ and  $\vphi(x) := \tilde \vphi(2m- x)$, we see that $V(\,\cdot \, | m) - \vphi$ has a local maximum at $x_1$ and $(V(\,\cdot \, | m)-\vphi)(x_1)=0$. Since $V(\,\cdot \, | m)$ is viscosity subsolution to \eqref{visc1}, we have that
\begin{align*}
0 &\geq \min \big\{ \mathcal{L} \vphi(x_1) - rV(x_1|m) + f(x_1,m), \,\, \mm V (x_1|m) - V(x_1|m) \big\} \\
&= \min \big\{ \mathcal{L} \tilde \vphi(x_0) - r\tilde V(x_0)+ f(x_0,m), \,\, \mm \tilde V (x_0) - \tilde V(x_0) \big\},
\end{align*}
where in the last step we have used the definitions of $\tilde V, \tilde \vphi, x_1$ and the  symmetry properties in \eqref{eq:sym-ctrl-cost} and (A3-2).
	
(ii) For each $x \in \rr$, we have 
\begin{equation}
\begin{aligned}
\label{visc2}
\mm V(m-x|m) &= \inf_{\delta \in \rr} \Big\{V(m-x+\delta|m) + \phi(\delta) \Big\}= \inf_{\tilde \delta \in \rr} \Big\{ V(m-(x+\tilde \delta)|m) + \phi(-\tilde \delta) \Big\} \\
&=\inf_{\tilde \delta \in \rr} \Big\{ V(m+(x+\tilde \delta)|m) + \phi(\tilde \delta) \Big\} =  \mm V(m+x|m)
\end{aligned}
\end{equation}
where in the second equality we have used  the change of variable $\tilde \delta = - \delta$, while in the second-to-last equality we have used (i) and \eqref{eq:sym-ctrl-cost}. The claim in (ii) immediately follows from \eqref{visc2} and (i).

(iii) By arguing as in (ii), for every $x \in \rr$ we have
\begin{equation*}
\xi_m(m+x) = \arg\min_{\delta \in \rr} \{ V (m+x+\delta|m) + \phi(\delta)\} = -\arg\min_{\tilde\delta \in \rr} \{ V (m-x+\tilde\delta|m) + \phi(\tilde\delta)\} = -\xi_m(m-x).
\end{equation*}
}
\end{proof}

\subsection{MFGs vs $N$-player games}
Next,  we will demonstrate that the solution to \eqref{eq:smfg-gen} is an approximation to the corresponding $N$-player game with identical players under the symmetric setting. Here the state process of the $N$-player game on $\rn^N$ is denoted by $\{\mathbf X_t\}_{t\geq0}=\{(X_t^1,\dots,X^N_t)\}_{t\geq0}$, with player $i$ only controlling her own state process $X^i$, $i=1,\dots, N$.

Recall the notations in Section $2$: for player $i$, with $i=1,\dots, N$, $f_i:\rn^N\to\rn$ denotes the running cost, $\phi_i$ denotes the individual cost of control and $\psi_{i,j}$ denotes the cost of control done by another player $j$. To be consistent with the MFGs setting, we assume 
\begin{description}
	\item[(A1$'$)]Player $i$'s uncontrolled state process is given by 
	\begin{equation*}
		dX^i_t=\sigma dW^i_t,\quad X^i_{0-}\sim \mu,
	\end{equation*}
	where $(W^1, \dots, W^N)$ denotes the $N$-dimensional standard Brownian motion, $\sigma>0$ is a constant and $\mu$ is symmetric around its mean.
	\item[(A2$'$)] The cost of individual control $\phi_i$ equals $\phi$ that satisfies Assumption (A2) and the costs of other players' control, $\psi_{i,j}$ are identical for all $i$ and $j$. 
	\item[(A3$'$)] The cost function $f_i$ takes the form of 
	$$\begin{aligned}
	&f_i(\mathbf x)=g\left(x_i-\frac{1}{N}\sum_{j=1}^Nx_j\right),\,\forall \mathbf x\in\rn^N,\\
	&f(x,m)=g(x-m),\,\forall x\in\rn, \end{aligned}$$
for some function $g:\rn\to\rn^+$, $g(x)=g(-x)$,  with $$|g(x)-g(y)|\leq L|x-y|,\,\forall x,y\in\rn,$$ where $L>0$ is the Lipschitz constant.
\end{description} 
Note that the running cost of \eqref{eq:smfg-gen} under Assumption (A3$'$) also satisfies Assumption (A3).
\begin{defn}[\textbf{$\epsilon$-Nash equilibrium}]
A strategy $\vphi^*=(\vphi_1^*,\dots,\vphi_N^*)$ is a called an $\epsilon$-Nash equilibrium to the $N$-player game introduced in Section \ref{sec:DefGames} if $$\eee_{\mu}\left[J^i(\mathbf X_{0-};\vphi^*)\right]\leq \eee_{\mu}\left[J^i(\mathbf X_{0-};(\vphi^{*,-i},\vphi_i))\right]+\epsilon,\quad \forall \vphi_i\in\Phi_i(\mathbf X_{0-})\text{ s.t. }(\vphi^{*,-i},\vphi_i)\in\Phi(\mathbf X_{0-}),$$
where $X_{0-}^i\overset{i.i.d.}{\sim}\mu$, $i=1,\dots, N$.
\end{defn}

Let $(\tilde{\vphi}^*,m^*)$ be a solution to the \eqref{eq:smfg-gen} under Assumptions (A1)-(A4), where $m^*$ is the expectation of initial distribution $\mu$, the control policy $\tilde{\vphi}^*$ is characterized by action region $\tilde A^*$ and impulse function $\tilde \xi^*:\rn\to\rn$; denote the corresponding waiting region as $\tilde D^*$ and the state process on $\rn$ under $\tilde \vphi^*$ as $\tilde X$. As illustrated in Theorem \ref{thm:sys}, the waiting region $\tilde D^*$ as well as the impulse function $\tilde\xi^*$ will be symmetric around $m^*$. Define the following priority sets
$$P_i=\{\mathbf x\in\rn^N:|x_i-m^*|>|x_j-m^*|\hspace{1pt},\,\forall j>i;\,{\hspace{1pt}}|x_i-m^*|\geq|x_k-m^*|,\,\forall k>i\},\hspace{1pt}\quad\forall i\in\{1,\dots, N\}.$$
Then for $i=1,\dots, N$, define the action region for Player $i$ as $A_i^*=\{\mathbf x\in\rn^N:x_i\in\tilde A^*\}\cap P_i$ and her impulse function $\xi_i^*:\rn^N\to\rn$ such that $$\xi_i^*(\mathbf x)=\tilde\xi^*(x_i),\,\forall \mathbf x\in\rn^N.$$
Denote $\vphi_i^*=(A_i^*,\xi_i^*)$ and $\vphi^*=(\vphi_1^*,\dots,\vphi_N^*)$. 
\begin{thm}
\label{thm:eps}
Let Assumptions (A1-A3) and (A1$'$-A3$'$) hold. Suppose that, under $\tilde\vphi^*$, we have $\tilde D^*\subset[m^*-u^*,m^*+u^*]$ for some positive constant $u^*$ and that $\tilde X^*_t\in\tilde D^*$ almost surely for all $t\geq0$. Then $\vphi^*$ is an $\epsilon$-NE for the $N$-player cash management game introduced in Section \ref{2-player} for generic $N$, with $\epsilon=O\left(\frac{1}{\sqrt{N}}\right)$.
\end{thm}
\begin{proof}
Fix $i\in\{1,\dots,N\}$. Consider $\bar{\vphi}=(\vphi^{*,-i},\vphi_i)$ such that $\vphi_i\in\Phi_i(\mathbf x)$ and $\bar{\vphi}\in\Phi(\mathbf x)$. For $j\neq i$, $\bar{\vphi}_j=\vphi_j^*$ whose action region independent from the strategy of player $i$. 

We first look at the running cost. 
$$X^i_t-\frac{1}{N}\sum_{j=1}^NX^j_t=\left(X^i_t-m^*\right)-\frac{1}{N}\left(X^i_t-m^*\right)-\frac{\sum_{j\neq i}(X^j_t-m^*)}{N},$$
so that
$$\begin{aligned}
\left|f_i(\mathbf X_t)-f(X^i_t,m^*)\right|&=\left|g\left(X^i_t-\frac{1}{N}\sum_{j=1}^NX^j_t\right)-g\left(X^i_t-m^*\right)\right|\\
&\leq L\left(\frac{1}{N}|X^i_t-m^*|+\frac{\sum_{j\neq i}|X^j_t-m^*|}{N-1}\right).
\end{aligned}$$

Note that $$\left|\frac{\sum_{j\neq i}(X^j_t-m^*)}{N-1}\right|\leq u^*$$ and by the i.i.d. assumption,
$$\eee_\mu\left|\frac{\sum_{j\neq i}(X^j_t-m^*)}{N-1}\right|\leq \left(\eee_\mu\left|\frac{\sum_{j\neq i}(X^j_t-m^*)}{N-1}\right|^2\right)^{\frac{1}{2}}\Rightarrow \eee_\mu\left|\frac{\sum_{j\neq i}(X^j_t-m^*)}{N-1}\right|=O\left(\frac{1}{\sqrt N}\right).$$
Without loss of generality, let us consider $\vphi_i$ such that $$\eee_\mu\left[\int_0^\infty e^{-rt}L|X^i_t-m^*|dt\right]<M,$$ for some sufficiently large $M>0$. Then
\begin{align*}
&\left|\eee_\mu\left[\int_0^\infty e^{-rt}f_i(\mathbf X_t)dt\right]-\eee_\mu\left[\int_0^\infty e^{-rt}f(X^i_t,m^*)dt\right]\right|\\
&\qquad\quad\leq\frac{1}{N}\eee_\mu\left[\int_0^\infty e^{-rt}L|X^i_t-m^*|dt\right]+\eee_\mu\left[\int_0^\infty e^{-rt}L\left|\frac{\sum_{j\neq i}(X^j_t-m^*)}{N-1}\right|dt\right]\\
&\qquad\quad=\frac{1}{N}\eee_\mu\left[\int_0^\infty e^{-rt}L|X^i_t-m^*|dt\right]+\int_0^\infty e^{-rt}L\eee_\mu\left[\left|\frac{\sum_{j\neq i}(X^j_t-m^*)}{N-1}\right|\right]dt \quad (\text{Fubini's})\\
&\qquad\quad=O\left(\frac{1}{N}\right)+O\left(\frac{1}{\sqrt N}\right)\\
&\Rightarrow \eee_\mu\left[\int_0^\infty e^{-rt}f_i(\mathbf X_t)dt\right]=\eee_\mu\left[\int_0^\infty e^{-rt}f(X^i_t,m^*)dt\right]+  O\left(\frac{1}{\sqrt N}\right)
\end{align*}
and therefore
\begin{empheq}{align*}
&\eee_\mu\left[J^i(\mathbf X_{0-};\bar{\vphi})\right]\\
&=\eee_\mu\left[\int_0^\infty e^{-rt}f_i(\mathbf X_t)dt+\sum_{n\geq1}e^{-r\tau_{i,n}}\phi(\xi_{i,n})+\sum_{j\neq i}\sum_{n\geq 1}e^{-r\tau_{j,n}}\psi_{i,j}(\xi_{j,n}^*)\right]\\
&=\eee_\mu\left[\int_0^\infty e^{-rt}f(X^i_t,m^*)dt+\sum_{n\geq1}e^{-r\tau_{i,n}}\phi(\xi_{i,n})\right]+\eee_\mu\left[\sum_{j\neq i}\sum_{n\geq 1}e^{-r\tau_{j,n}}\psi_{i,j}(\xi_{j,n}^*)\right]+O\left(\frac{1}{\sqrt{N}}\right)\\
&\geq V(\mu) +\eee_\mu\left[\sum_{j\neq i}\sum_{n\geq 1}e^{-r\tau_{j,n}}\psi_{i,j}(\xi_{j,n}^*)\right]+O\left(\frac{1}{\sqrt{N}}\right)\\
&=\eee_\mu\left[J^i(\mathbf X_{0-};\vphi^*)+O\left(\frac{1}{\sqrt{N}}\right)\right],
\end{empheq}
where we have denoted by $V(\mu)$ the payoff of the impulse MFG with initial distribution $\mu$.
\end{proof}
}

\begin{remark}
Among the assumption of Theorem \ref{thm:eps}, we ask that $\tilde D^*\subset[m^*-u^*,m^*+u^*]$ is bounded. Heuristically, if $f(x,m) \to +\infty$ as $x \to \infty$ and diverges at a greater rate than the intervention costs, we expect that, for $|x|$ big enough, intervening is cheaper than keeping the state as it is (in other words, we expect that the continuation region is bounded). \textcolor{black}{For examples of bounded continuation regions in impulse control theory, see e.g.~the references in Remark \ref{remonA4}, i.e.~\cite{basei}, \cite{CadenillasLaknerPinedo}, \cite{CadenillasZapatero}, \cite{Korn}, \cite{MitchellFengMuthuraman}, \cite{MundacaOksendal}. For a two-player stochastic impulse game with bounded continuation regions, see \cite{FerrariKoch}. In Section \ref{ssec:explicit} below, we will provide an example of impulse MFG where this condition is satisfied.}
\end{remark}

\subsection{Explicit Solutions: MFGs for Cash Management Problems}
\label{ssec:explicit} 
In this section, we explicitly solve the MFGs cash management problems, i.e., a (slightly more general) mean-field counterpart to the two-player game in Section \ref{2-player}. 

\begin{description}
\item[(A1$''$)]The uncontrolled dynamics follow{}
\begin{equation}
dX_t=\sigma dW_t,\quad X_{0-}\sim \mu,{}
\end{equation}
where $\sigma>0$ is a constant.
\item[(A2$''$)]The cost of control satisfies
\begin{equation}
\phi(\xi)=\begin{cases}
K^++k^+\xi,&\xi\geq0,\\
K^--k^-\xi, &\xi<0,
\end{cases}
\end{equation}
where $K^\pm,k^\pm>0$.
\item[(A3$''$)]The running cost takes the form 
$$f(x,m)=C\left(x-\alpha(m)\right),$$
where $C:\rr \to \rr$ is defined by
\begin{equation}
C(x)=\max\{hx,-px\}, \qquad\quad x \in \rr,
\end{equation}
with parameters $h,p >0$ satisfying
$$ h-k^+r>0,\quad p-k^-r>0, $$
and where $\alpha:\rn\to\rn$ is a contraction, i.e., 
$$|\alpha(x) - \alpha(y)| \leq  k |x-y|, \qquad\quad 0 < k < 1, \qquad x,y \in \rr.$$
\end{description}
Here the function $\alpha$ can be interpreted as target level depending on the mean information $m$. The assumption that $\alpha$ is a contraction mapping is for analytical tractability.
\begin{thm}\label{thm:smfg-cash}
Under Assumptions (A1--A4), \eqref{eq:smfg-gen} admits a unique analytical solution in the sense of Definition \ref{def:smfg-sol}.
\end{thm}
\begin{proof}
We will in fact explicitly derive the solution. Let us fix $m\in\rn$. Then the corresponding QVI for the control problem is
\begin{equation}\label{eq:qvi-stat-mfg}
\min\left\{\na V-rV+C(x-\alpha(m)),\m V-V\right\}=0.
\end{equation}
Similar to \cite{Constantinides1978}, one can  find an optimal policy characterized by the vector $(d,D, U,u)$ with $d<D<0<U<u$, by smooth-fit principle. The payoff corresponding to \ref{thm:smfg-cash} has to satisfy (we set $\lambda = \frac{\sqrt{2r\sigma^2}}{\sigma^2}$)
\begin{empheq}[left=\empheqlbrace]{align}
&V(x)=\begin{cases}
V\left(u+\alpha(m)\right)-k^-\left(u-x+\alpha(m)\right),&x-\alpha(m)\geq u,\\
\vspace{1pt}\\
\frac{h}{r}\left(x-\alpha(m)\right)+c_1\exp\left\{\lambda\left(x-\alpha(m)\right)\right\}\\
\hspace{45pt}+c_2\exp\left\{-\lambda\left(x-\alpha(m)\right)\right\},&0\leq x-\alpha(m)\leq u,\\
\vspace{1pt}\\
-\frac{p}{r}\left(x-\alpha(m)\right)+\left(c_1+\frac{h+p}{2r\lambda}\right)\exp\left\{\lambda\left(x-\alpha(m)\right)\right\}\\
\hspace{45pt}+\left(c_2-\frac{h+p}{2r\lambda}\right)\exp\left\{-\lambda\left(x-\alpha(m)\right)\right\},&d\leq x-\alpha(m)\leq 0;\\
\vspace{1pt}\\
V\left(d+\alpha(m)\right)+k^+\left(d-x+\alpha(m)\right),&x-\alpha(m)\leq d;
\end{cases}\label{eq:mfg-reg-sym}\\
&\dot{V}\left(U+\alpha(m)\right)=\dot{V}\left(u+\alpha(m)\right)=k^-,\quad\dot{V}\left(D+\alpha(m)\right)=\dot{V}\left(d+\alpha(m)\right)=-k^-; \label{eq:mfg-ctrl-1st-sym}\\
&
\begin{aligned}V\left(u+\alpha(m)\right)=K^-+k^-(u-U)+V\left(U+\alpha(m)\right), \\
V\left(d+\alpha(m)\right)=K^++k^+(D-d)+V\left(D+\alpha(m)\right).\end{aligned}\label{eq:mfg-ctrl-0th-sym}
\end{empheq}
Recall that $K^\pm,k^\pm>0$ and $h-rk^-,p-rk^+>0$. By \cite{Constantinides1978}, there exists a 6-tuple $(c_1,c_2,d,D,U,u)$ satisfying $\eqref{eq:mfg-reg-sym}$, $\eqref{eq:mfg-ctrl-1st-sym}$ and $\eqref{eq:mfg-ctrl-0th-sym}$ such that $d<D<0<U<u$ and
	\begin{empheq}[left=\empheqlbrace]{align}
	&c_1= \frac{h+p}{r\lambda}\frac{(e^{-\lambda u}-e^{-\lambda U})[\cosh(\lambda d)-\cosh(\lambda D)]}{(e^{\lambda u}-e^{\lambda U})(e^{-\lambda d}-e^{-\lambda D})-(e^{-\lambda u}-e^{-\lambda U})(e^{\lambda d}-e^{\lambda D})}
	\in(-\frac{h+p}{2r\lambda},0),\\
	&c_2=c_1\frac{e^{\lambda u}-e^{\lambda U}}{e^{-\lambda u}-e^{-\lambda U}}\in\left(0,\frac{h+p}{2r\lambda}\right),\\
	&K^--\left(\frac{h}{r}-k^-\right)(u-U)-2c_1(e^{\lambda u}-e^{\lambda U})=0,\\
	&\lambda \left(c_1e^{\lambda u}-c_2e^{-\lambda u}\right)+\left(\frac{h}{r}-k^-\right)=0,\\
	&K^+-\left(\frac{p}{r}-k^+\right)(D-d)-2\left(c_1+\frac{h+p}{2r\lambda}\right)(e^{\lambda u}-e^{\lambda U})=0,\\
	&\lambda \left[\left(c_1+\frac{h+p}{2r\lambda}\right)e^{\lambda u}-\left(c_2-\frac{h+p}{2r\lambda}\right)e^{-\lambda u}\right]-\left(\frac{p}{r}-k^+\right)=0,
	\end{empheq}
where the  thresholds $d, D, u, U$ only depend on $K^\pm, k^\pm, h, p, r, \sigma$.
The optimal simple control policy $\vphi^{*} =\left((-\infty,\alpha(m)+d]\cup[\alpha(m)+u,+\infty),\xi^*\right)= \{(\tau_{n}^*,\xi_{n}^*)\}_{n\geq 1}$ is given by 
\begin{equation}\label{eq:mfg-ctrl}
\begin{aligned}
\xi^*(x)&=\begin{cases}
U-x+\alpha(m), &\text{if } x-\alpha(m)\geq \alpha(m)+u,\\
D-x+\alpha(m), &\text{if } x-\alpha(m)\leq \alpha(m)+d,\\
0,&\text{otherwise}.
\end{cases}\\
\tau_{1}^*&=\inf\{t\geq 0:|X_{t-}-\alpha(m)|\not\in(\alpha(m)+d,\alpha(m)+u)\}, \\
\tau_{n}^*&=\inf\{t>\tau_{n-1}^*:|X_{t-}-\alpha(m)|\not\in(\alpha(m)+d,\alpha(m)+u)\},\quad n\geq 2;\\
\xi_n^{*}&=\xi^*(X_{\tau_n-})
\end{aligned}
\end{equation}
Assume that the initial position $X_{0-}$ follows any given distribution $\mu$. Recall from \eqref{eq:smfg-gen} that 
\begin{align*}
V(x)&=\inf_{\vphi}\eee_x\left[\int_{0}^{\infty}e^{-rt}f(X_t,m)dt+\sum_{n\geq 1}e^{-r\tau_{n}}\phi(\xi_{n})\right]\\
&=\inf_{\vphi}\eee_{\mu}\left[\int_{0}^{\infty}e^{-rt}f(X_t,m)dt+\sum_{n\geq 1}e^{-r\tau_{n}}\phi(\xi_{n})\middle\vert X_{0-}=x\right].
\end{align*}

Denote the updated mean information as $\bar{m}=\limsup_{t\to\infty}\eee_x\left[{X}_{t}\right]$. We will show that this $\bar{m}$ is well-defined and invariant with respect to $x$. 

Notice that $\bar{m}=\lim\sup_{t\to\infty}\eee_x[{X}_{t}]=\lim\sup_{n\to\infty}\eee_x[X_{\tau_{n}^*}]$ by symmetry and a Fubini argument. Note that $\eee_x[X_{\tau_{n}^*}]=\alpha(m)+U\pbm\{X_{\tau_{n}^*}=\alpha(m)+U\}+D\left[1-\pbm\{X_{\tau_{n}^*}=\alpha(m)+U\}\right]$. For simplification, denote $\pbm\{X_{\tau_{n}^*}=\alpha(m)+U|X_{0-}=x\}$ as $p_n(x)$. Then, by the strong Markovian property of $\bar{X}_t$
\begin{align*}
&q_1\equiv\pbm\{X_{\tau_{n+1}^*}=\alpha(m)+U|X_{\tau_{n}^*}=\alpha(m)+U\}=\frac{U-d}{u-d}, \forall n\in\nn,\\
&q_2\equiv\pbm\{X_{\tau_{n+1}^*}=\alpha(m)+U|X_{\tau_{n}^*}=\alpha(m)+D\}=\frac{D-d}{u-d}, \forall n\in\nn,\\
&p_{n+1}(x)=p_n(x)q_1+\left[1-p_n(x)\right]q_2.
\end{align*}
Therefore, we have $$p_{n+1}(x)=q_2+(q_1-q_2)p_n(x)\Rightarrow p_n(x)-\frac{q_2}{1-q_1+q_2}=(q_1-q_2)^{n-1}\left[p_1(x)-\frac{q_2}{1-q_1+q_2}\right].$$ Hence, $\lim_{n\to\infty}p_n(x)=\frac{q_2}{1-q_1+q_2}$ and this is independent of the initial position $x$. We then have $\bar{m}=\alpha(m)+\frac{uD-dU}{u-U+D-d}$. Define the update of mean information $\Gamma:m\mapsto\bar{m}$ as $\Gamma(m)=\alpha(m)+\frac{uD-dU}{u-U+D-d}$. Since $\alpha$ is assumed a contraction mapping, so is $\Gamma$. Denote the fixed point of $\Gamma$ as $m^*$ and let $\vphi^*=\vphi(m^*)$ be be as in \eqref{eq:mfg-ctrl}. Then $(\vphi^*,m^*)$ is a solution to the (\ref{eq:smfg-gen}) in the sense of Definition \ref{def:smfg-sol}.
\end{proof}

\begin{remark}
Notice that in the example of Section \ref{2-player}, the cash management setting is under a symmetric cost structure with $h=p$, $K^{\pm}=K$, $k^\pm = k$. Its mean-field counterpart, by considering a similar derivation as in the proof of Theorem \ref{thm:smfg-cash}, has a symmetric solution structure with $d=-u$ and $D=-U$. Therefore $\frac{uD-dU}{u-U+D-d}=0$, and  $m^*=\eee_{\mu}[X_{0-}]$ is a solution. 
\end{remark}

\section{Sensitivity Analysis}

\textcolor{black}{In this section, we come back to the \emph{symmetric} cash management problem, and compare the solutions in the cases $N=1$ (`monopoly', described in \cite{Constantinides1978} and here recalled in Section 1) and $N=2$ (`duopoly', introduced in Section \ref{2-player} as the multi-player extension of the problem in \cite{Constantinides1978}).} 

Namely, we want to study see how  parameters $h$, $K$, $k$, $r$ and $\sigma$ influence the control policies and the thresholds $d$, $D$, $U$, $u$, we conduct a series of sensitivity analysis. We start with $h=2$, $K=3$, $k=1$, $r=0.5$, $\sigma = \frac{\sqrt{2}}{2}$ and $c = 1$.

We shall see similar behaviors for both the monopoly and the duopoly cases in terms of the thresholds and policy changes with respect to the underlying parameter changes.  One distinction is that  the thresholds and policy changes are more sensitive to parameter changes in the duopoly case due to competition. 

Finally, we will study the sensitivity analysis for the MFG counterpart.

\subsection{Duopoly vs Monopoly.}
Putting the thresholds for the duopoly and those of the monopoly together in Figure \ref{fig:2v1-ctrl}, one can see that due to competition in a game setting, players take the opponents' strategies into consideration. 

\textcolor{black}{We notice that the continuation region gets bigger in the duopoly case. Equivalently, interventions on the underlying process are less frequent in the duopoly case than in the monopoly case. Also, notice that the intervention size when the process reaches the lower or upper threshold are equal, due to the symmetric structure of the problem. In summary, the player of the monopoly makes frequent but cautiou interventions whereas each player in the duopoly intervenes less often but each time with bolder moves.}

\begin{figure}[H]
\centering
\includegraphics[width=\textwidth]{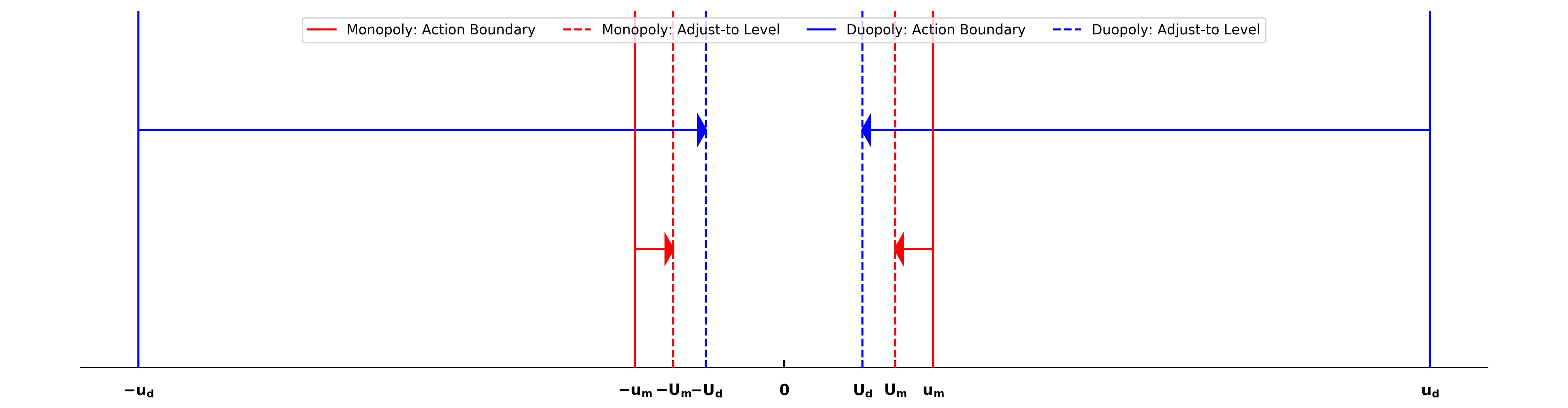}
\caption{Thresholds: Duopoly vs. Monopoly}
\label{fig:2v1-ctrl}
\end{figure}

\paragraph{Running Cost $h$.} When the  running cost $h$ increases, players have the incentive to intervene more frequently to prevent controlled process from deviating too far away from the target level. 
See Figure \ref{subfig:2v1-act-h}. On the other hand, the presence of the cost of control makes players more cautious when exercising controls.  Thus, an increased running cost encourages the players to intervene more frequently but with smaller amount of adjustment. See Figure \ref{subfig:2v1-jump-h}.

\begin{figure}[H]
\centering
\begin{subfigure}[b]{0.4\textwidth}
\centering
\includegraphics[width=\textwidth]{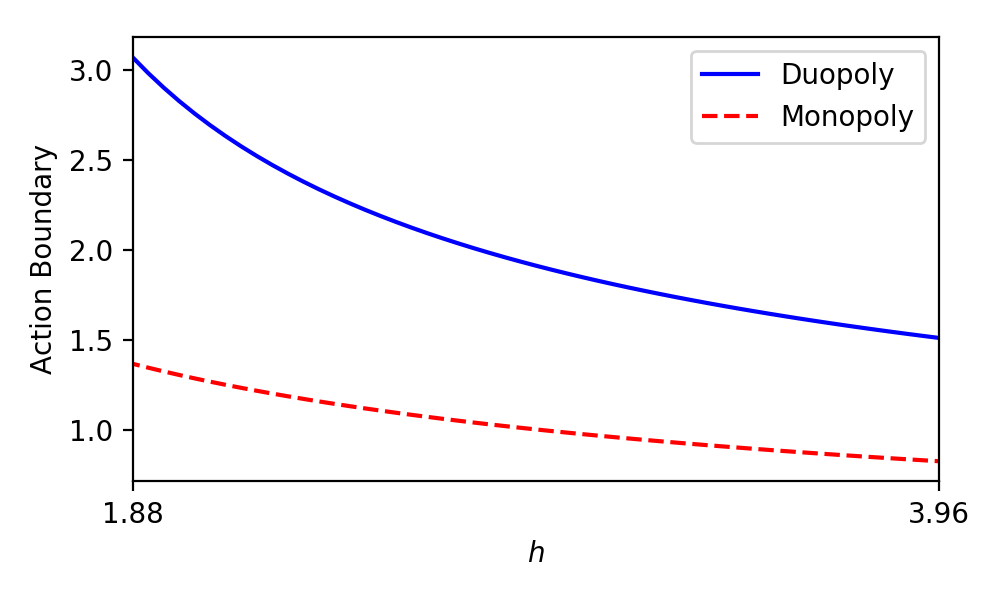}
\caption{Action Boundary of Monopoly and Duopoly}
\label{subfig:2v1-act-h}
\end{subfigure}
\begin{subfigure}[b]{0.4\textwidth}
\centering
\includegraphics[width=\textwidth]{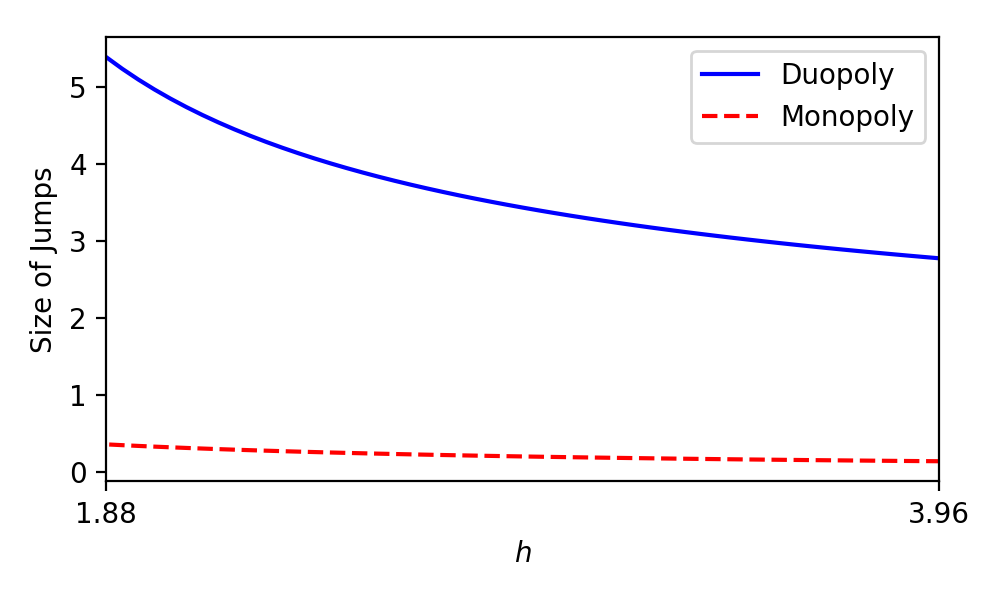}
\caption{Amount of Adjustment of Monopoly and Duopoly}
\label{subfig:2v1-jump-h}
\end{subfigure}
\caption{Sensitivity w.r.t. $h$}
\label{fig:2v1-h}
\end{figure}
\paragraph{Cost of Control $K$ and $k$.}
The parameter $K$ is the fixed cost when players choose to intervene. High fixed cost $K$ discourages the player from intervening too frequently. Therefore players have the incentive to tolerate a larger deviation from the target; and once a player chooses to intervene, the size of control needs to be  bigger to compensate for less frequent controls. Meanwhile, 
a decreasing frequency of intervention leads to an increasing action boundary $u$. See Figure \ref{subfig:2v1-act-K} and  Figure \ref{subfig:2v1-jump-K}.  For the per unit control cost $k$, similar results are shown in Figure \ref{fig:2v1-k}.
\begin{figure}[H]
\centering
\begin{subfigure}[b]{0.4\textwidth}
\centering
\includegraphics[width=\textwidth]{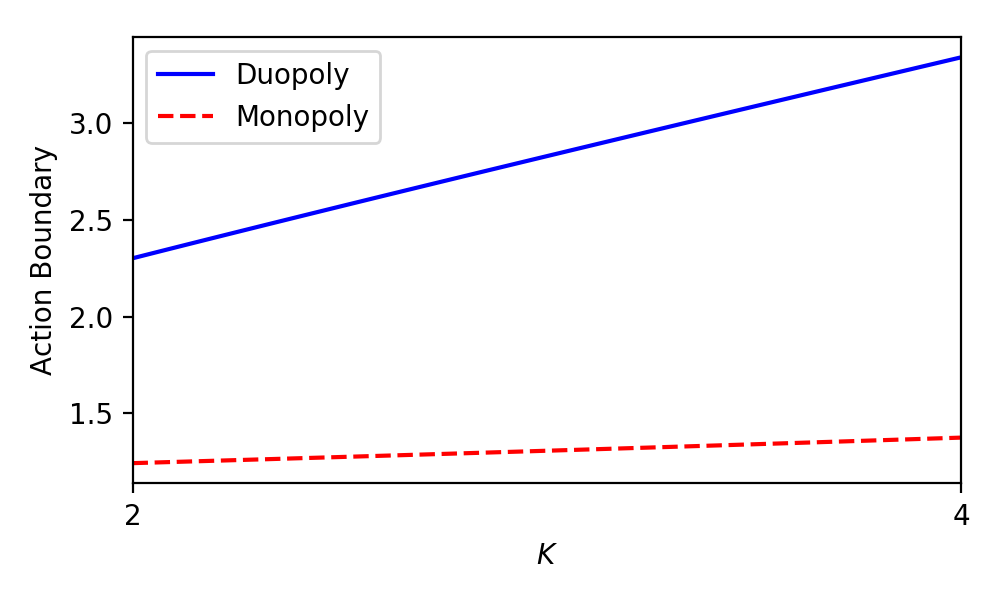}
\caption{Action Boundary of Monopoly and Duopoly}
\label{subfig:2v1-act-K}
\end{subfigure}
\begin{subfigure}[b]{0.4\textwidth}
\centering
\includegraphics[width=\textwidth]{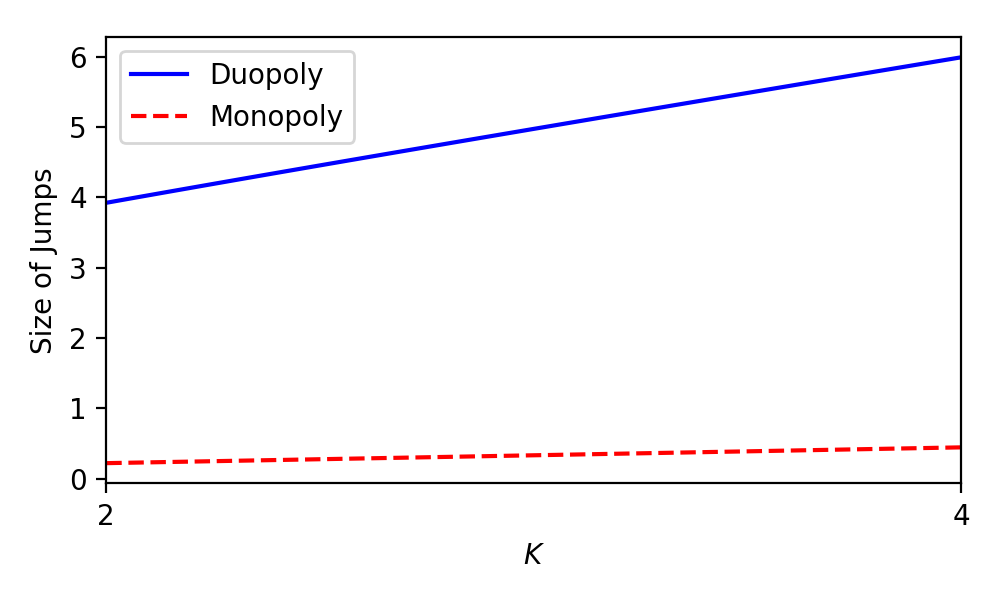}
\caption{Amount of Adjustment of Monopoly and Duopoly}
\label{subfig:2v1-jump-K}
\end{subfigure}
\caption{Sensitivity w.r.t. $K$}
\label{fig:2v1-K}
\end{figure}
\begin{figure}[H]
\centering
\begin{subfigure}[b]{0.4\textwidth}
\centering
\includegraphics[width=\textwidth]{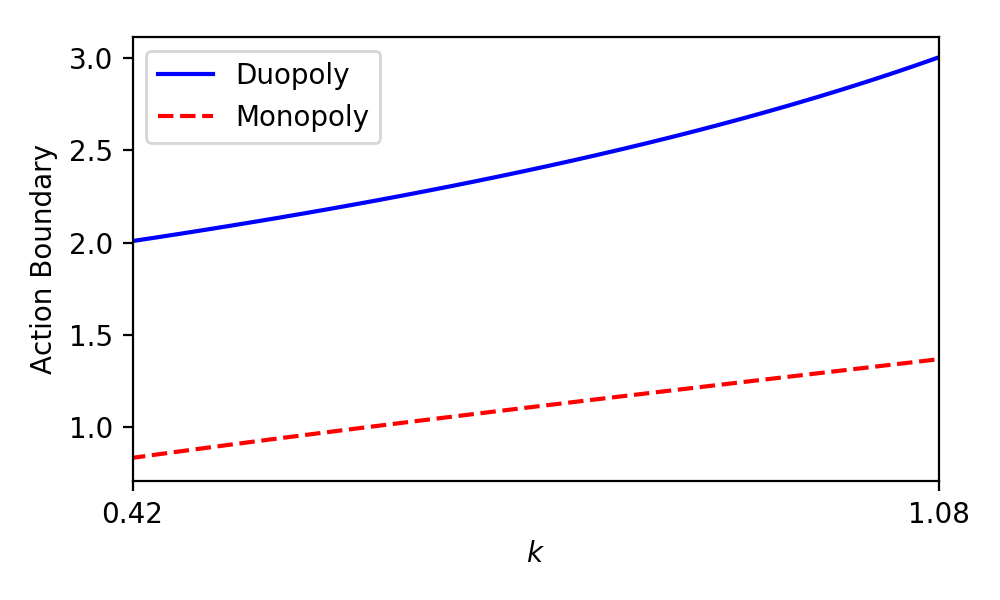}
\caption{Action Boundary of Monopoly and Duopoly}
\label{subfig:2v1-act-kk}
\end{subfigure}
\begin{subfigure}[b]{0.4\textwidth}
\centering
\includegraphics[width=\textwidth]{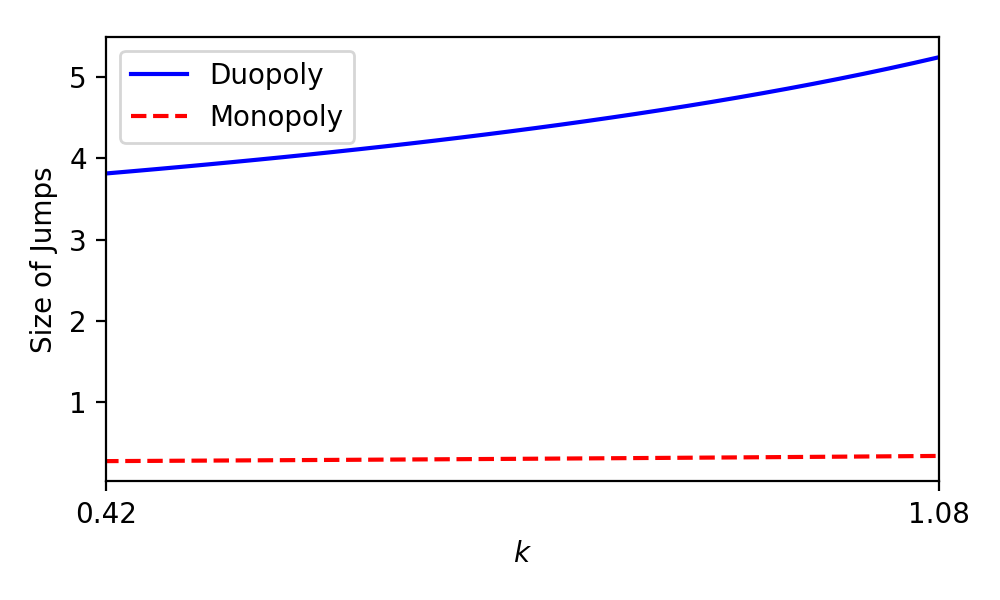}
\caption{Amount of Adjustment of Monopoly and Duopoly}
\label{subfig:2v1-jump-kk}
\end{subfigure}
\caption{Sensitivity w.r.t. $k$}
\label{fig:2v1-k}
\end{figure}
\paragraph{Discount Rate $r$.}    When the discount rate $r$ increases,  players are more tolerant with a larger deviation  from the target level as the penalty is discounted by a larger factor. That is,
a higher discount rate $r$ effectively reduces both the running and control cost, hence resulting in a decreased intervention frequency with an increased size of controls, 
as shown in Figure \ref{fig:2v1-r}.

\begin{figure}[H]
\centering
\begin{subfigure}[b]{0.4\textwidth}
\centering
\includegraphics[width=\textwidth]{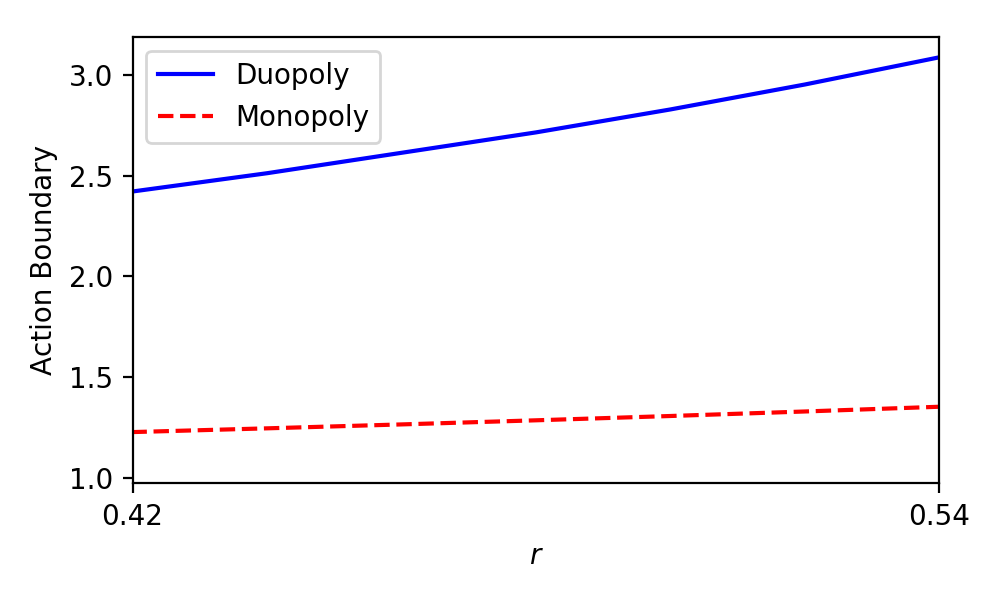}
\caption{Action Boundary of Monopoly and Duopoly}
\label{subfig:2v1-act-r}
\end{subfigure}
\begin{subfigure}[b]{0.4\textwidth}
\centering
\includegraphics[width=\textwidth]{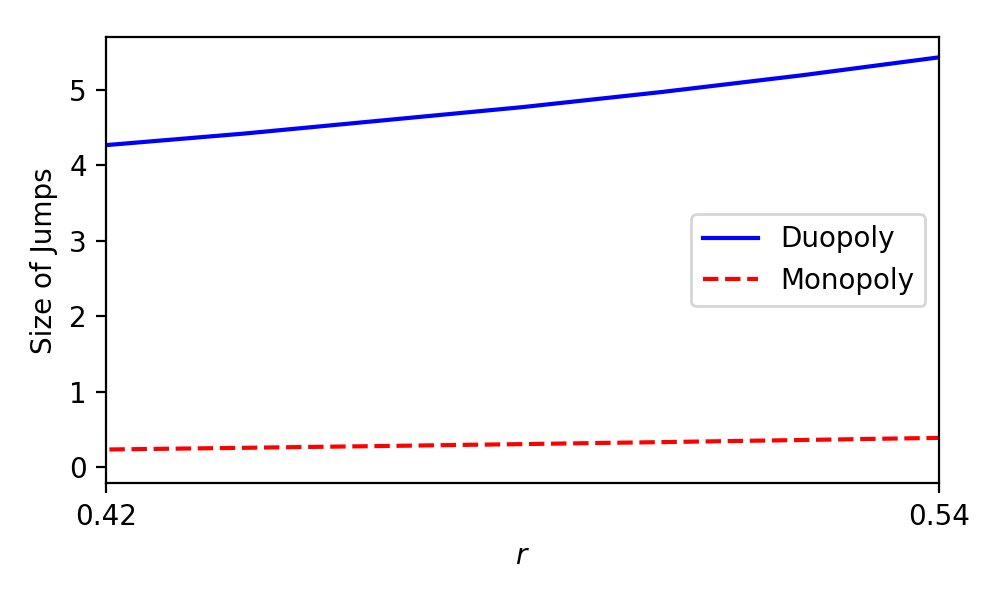}
\caption{Amount of Adjustment of Monopoly and Duopoly}
\label{subfig:2v1-jump-r}
\end{subfigure}
\caption{Sensitivity w.r.t. $r$}
\label{fig:2v1-r}
\end{figure}
\paragraph{Volatility $\mathbf\sigma$.} When the volatility $\sigma$ is bigger, players tend to intervene less as the controlled process is more likely to move
closer to the target level with a higher volatility. Therefore, a higher volatility allows players to intervene less frequently with a larger amount of adjustment, as shown in Figure \ref{fig:2v1-s}.
\begin{figure}[H]
\centering
\begin{subfigure}[b]{0.4\textwidth}
\centering
\includegraphics[width=\textwidth]{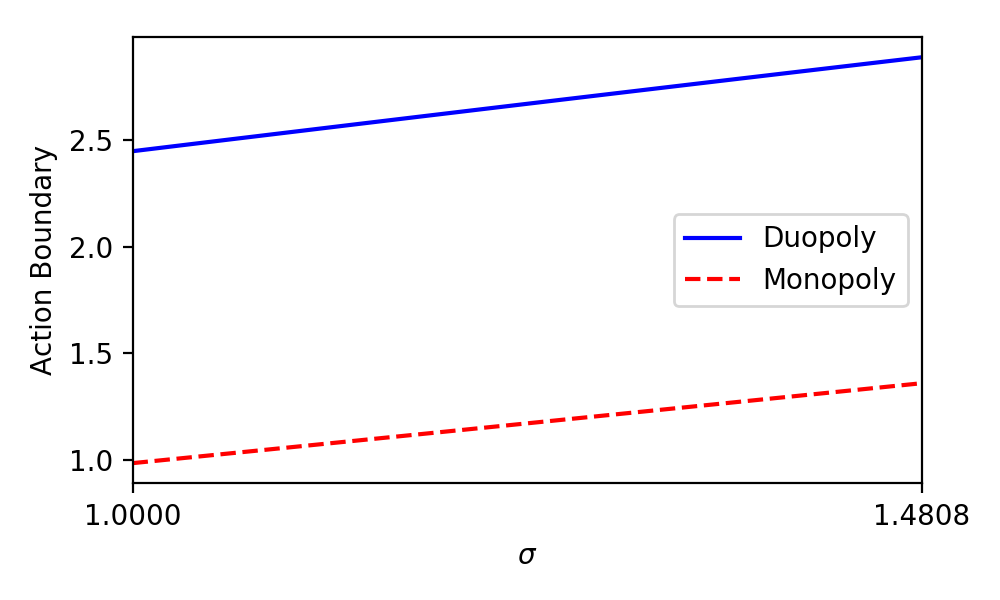}
\caption{Action Boundary of Monopoly and Duopoly}
\label{subfig:2v1-act-s}
\end{subfigure}
\begin{subfigure}[b]{0.4\textwidth}
\centering
\includegraphics[width=\textwidth]{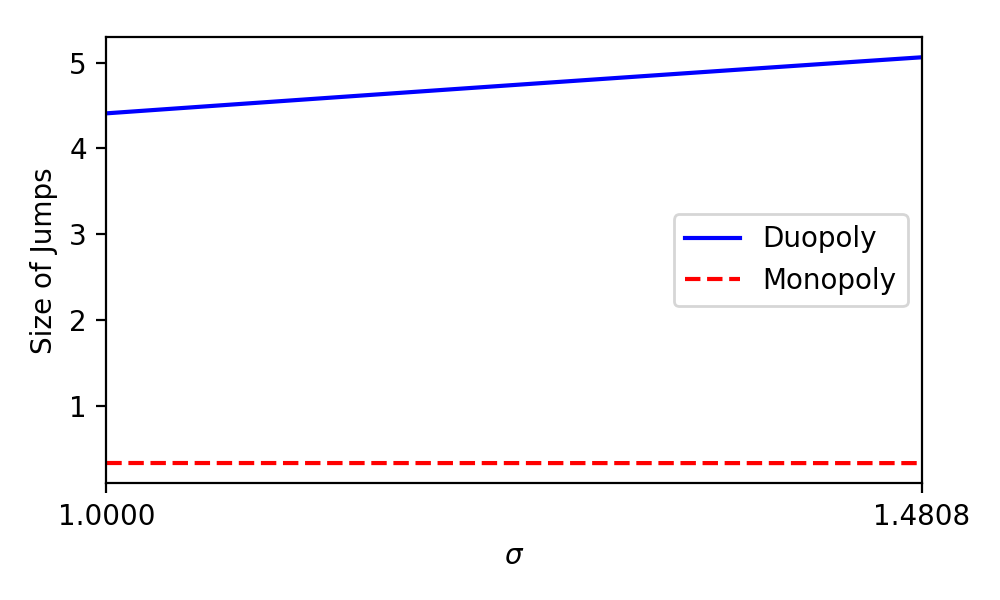}
\caption{Amount of Adjustment of Monopoly and Duopoly}
\label{subfig:2v1-jump-s}
\end{subfigure}

\caption{Sensitivity w.r.t. $\sigma$}
\label{fig:2v1-s}
\end{figure}

\subsection{Sensitivity Analysis of the MFG}
Here we present the sensitivity analysis of the MFG solution with respect the model parameters.
In particular, we look into the impact of the list of parameters, namely $h$, $p$, $K^\pm$, $k^\pm$, $r$ and $\sigma$, on the optimal impulse control policy and on the value of the mean information in the solution. Here we take one particular form of the $\alpha$ function  $\alpha(m) = \alpha m$ with $0<\alpha<1$.

\paragraph{Running cost $h$ and $p$.} Parameters relevant to running costs are holding cost $h$ and penalty cost $p$. When running cost increases, players have the incentive to intervene more frequently to prevent the state from deviating too far away from the target level $\alpha$; at the meantime, players pay extra precaution for each intervention so the jump size reduces. Figure \ref{fig:h} shows the impact of increasing holding cost $h$ with $p=2$, $K^-=3$, $k^-=0.5$, $K^+=3.25$, $k^+=1.5$, $r=0.5$, $\sigma=1$. It primarily affects the upper action boundary $u$, causing it to decreases. Meanwhile the gap between the $u$ and $U$ decreases as well. As for the mean information, eventually, increasing $h$ causes the mean information to decrease. Figure \ref{fig:p} illustrates the effect of increasing $p$ with $h=1$, $K^-=3$, $k^-=1$, $K^+=3.25$, $k^+=1.5$, $r=0.5$, $\sigma=1$. As opposed to increasing $h$, increasing $p$ primarily impacts $d$, causing to increase, while the gap between $d$ and $D$ decreases. Mostly it leads the mean information to increase.

\begin{figure}[H]
\centering
\begin{subfigure}[t]{0.48\textwidth}
\centering
\includegraphics[width=\textwidth]{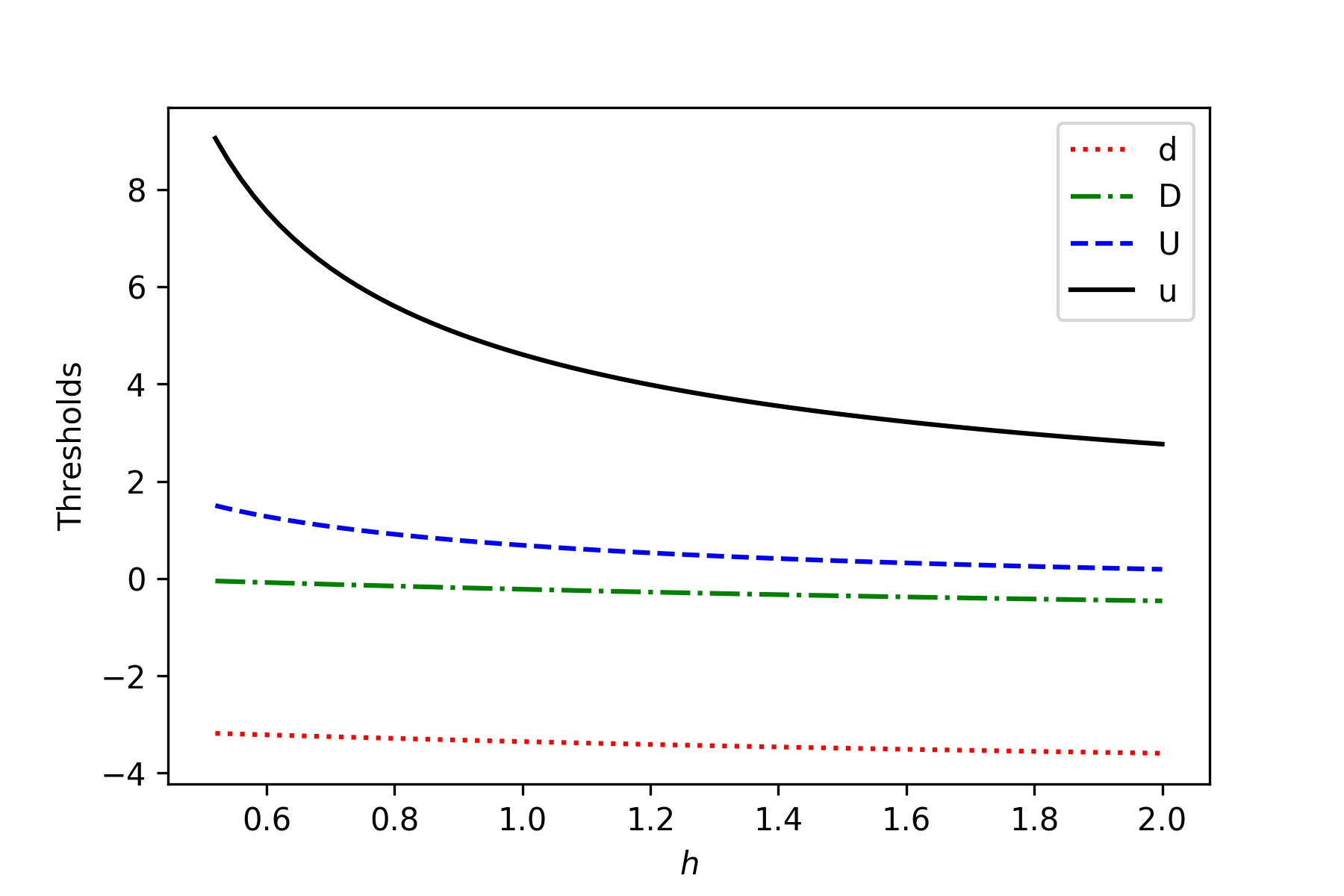}
\caption{All thresholds}
\label{subfig:all_h}
\end{subfigure}
~
\begin{subfigure}[t]{0.48\textwidth}
\centering
\includegraphics[width=\textwidth]{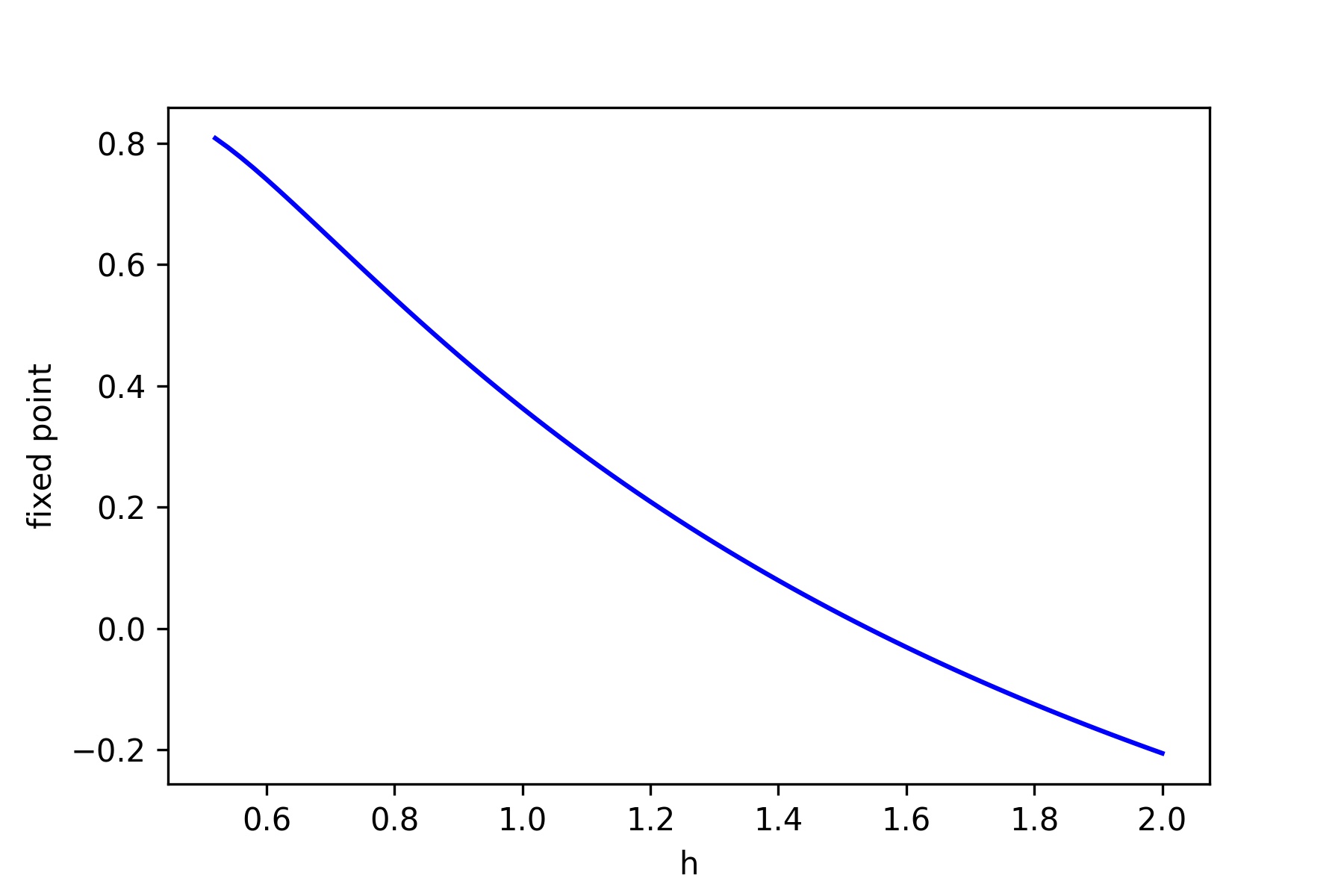}
\caption{Fixed point}
\label{subfig:fp_h}
\end{subfigure}
\caption{Sensitivity w.r.t. $h$}
\label{fig:h}
\end{figure}

\begin{figure}[H]
\centering
\begin{subfigure}[t]{0.48\textwidth}
\centering
\includegraphics[width=\textwidth]{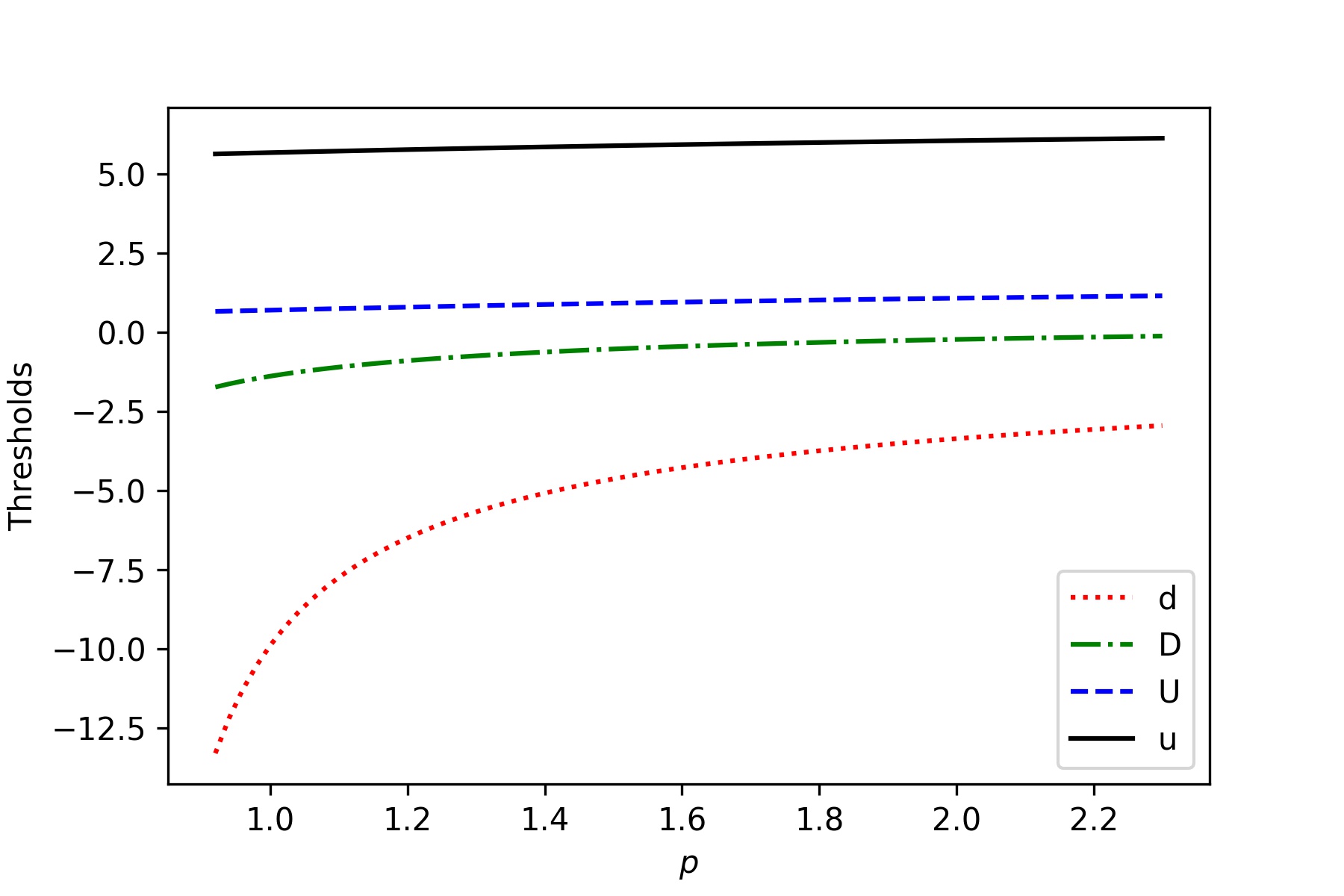}
\caption{All thresholds}
\label{subfig:all_p}
\end{subfigure}
~
\begin{subfigure}[t]{0.48\textwidth}
\centering
\includegraphics[width=\textwidth]{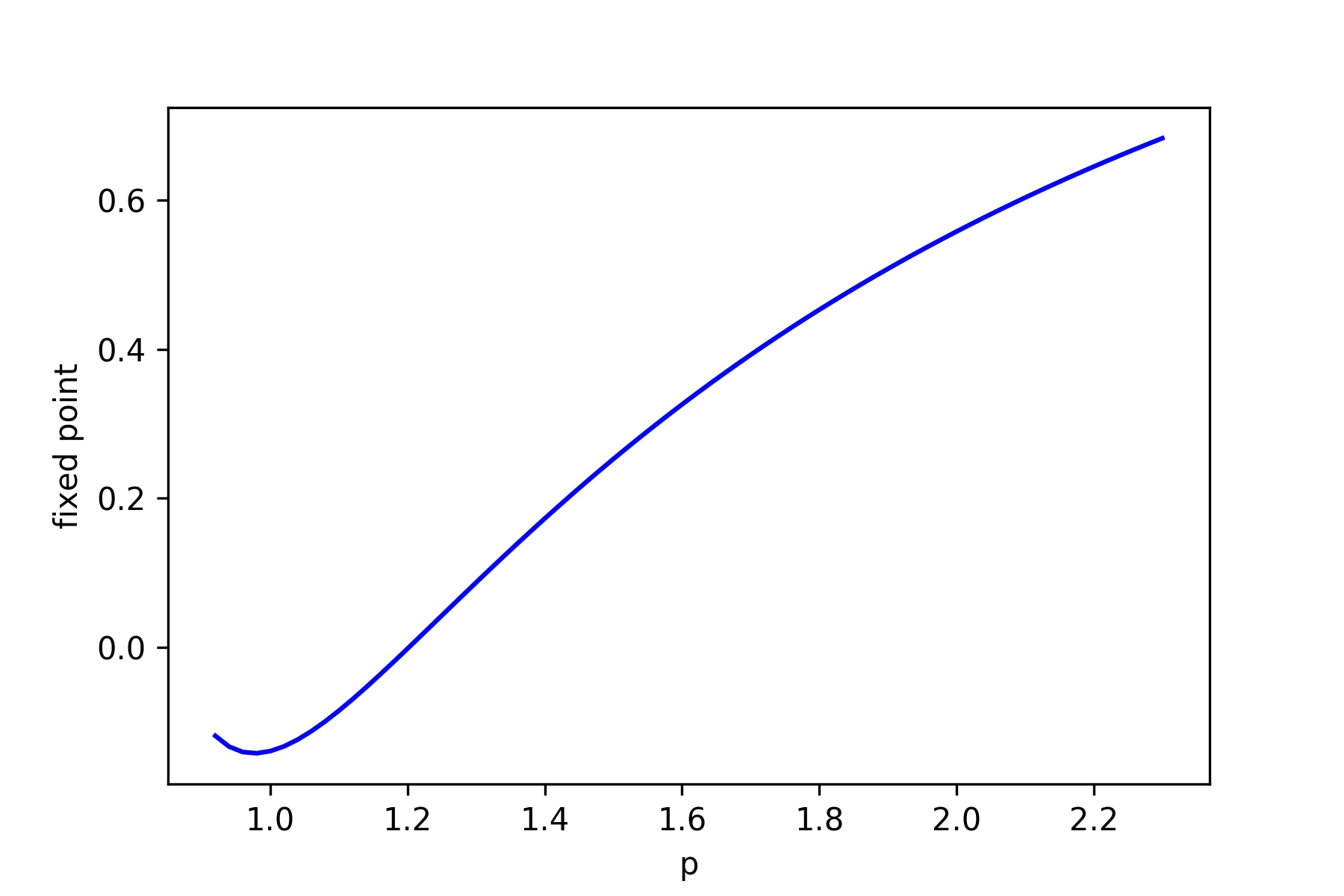}
\caption{Fixed point}
\label{subfig:fp_p}
\end{subfigure}
\caption{Sensitivity w.r.t. $p$}
\label{fig:p}
\end{figure}

\paragraph{Cost of instant decrease $K^-$ and $k^-$.}The parameter $K^-$ is the amount of fixed cost to pay every time a player chooses to decrease the state. A high fixed cost $K^-$ discourages the player from deceasing the state too often. Therefore the players have the incentive to tolerate a high state value compared with the reference point $\alpha$  and once a player chooses to decrease its state, the amount is adjustment will be larger. The more obvious impact is that the distance between $\alpha$ and the upper bound of the non-action region, i.e. $u$, increases; the impact on $U$ is trickier, as the per unit cost of decreasing the state stays unchanged. The impact of the increasing $K^-$ on the rest of the thresholds becomes the trade-of between unchanged holding cost, penalty cost, cost of increasing the state, and the costs associated with potentially more frequent but larger in scales, instantaneous increases in the state variables. 
\begin{figure}[H]
\centering
\begin{subfigure}[t]{0.48\textwidth}
\centering
\includegraphics[width=\textwidth]{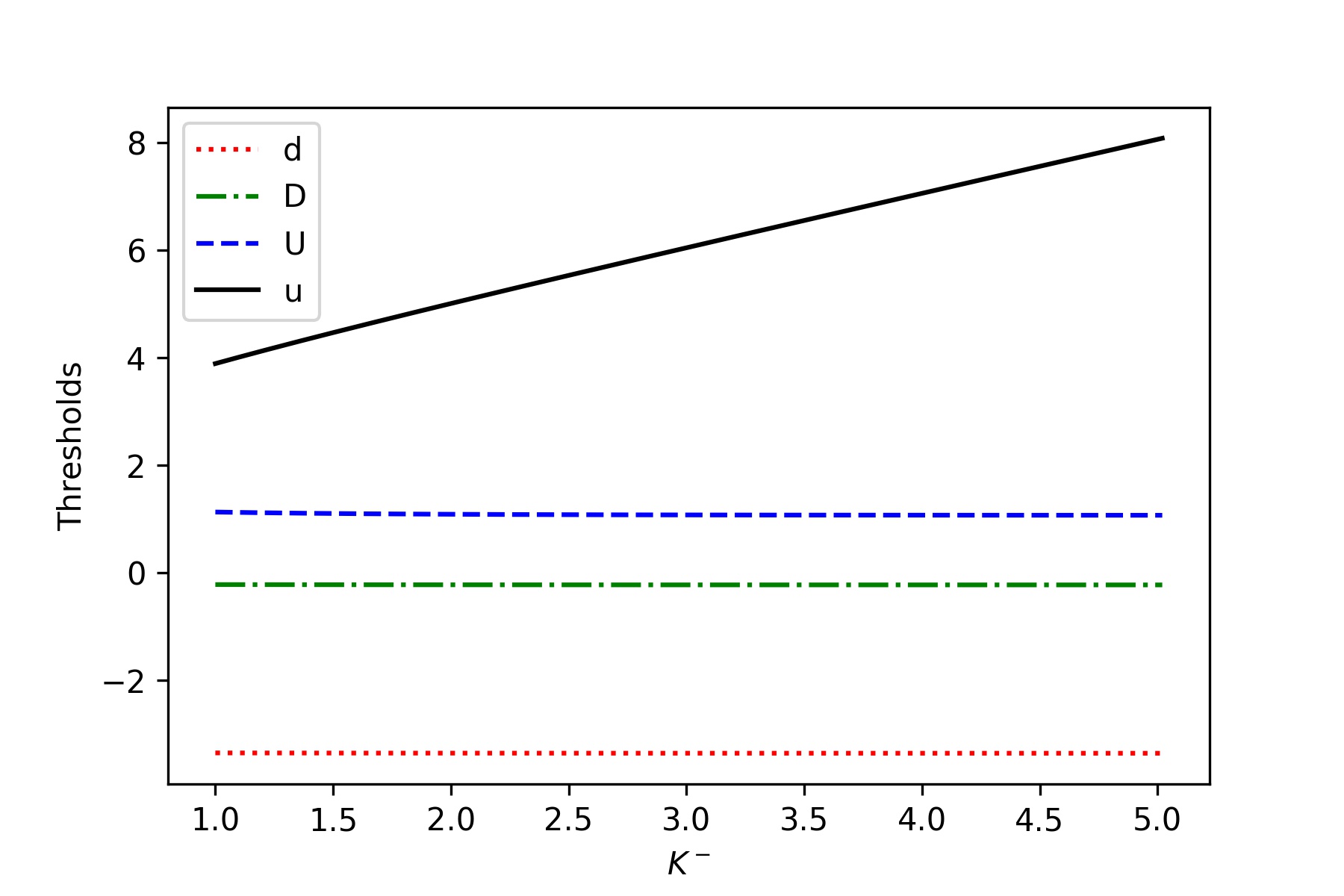}
\caption{All thresholds}
\label{subfig:all_K-}
\end{subfigure}
~
\begin{subfigure}[t]{0.48\textwidth}
\centering
\includegraphics[width=\textwidth]{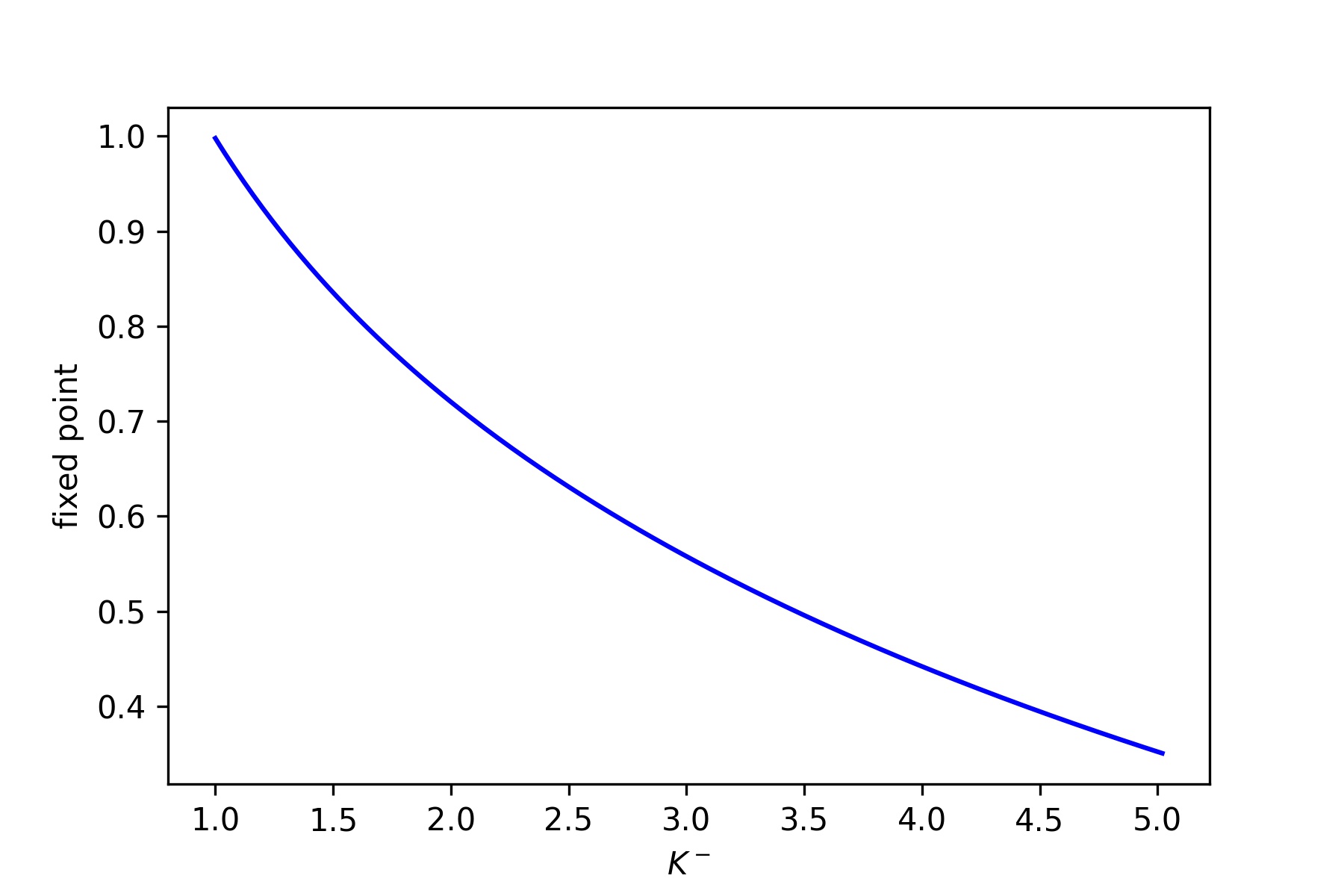}
\caption{Fixed point}
\label{subfig:fp_K-}
\end{subfigure}
\caption{Sensitivity w.r.t. $K^-$}
\label{fig:K-}
\end{figure}
Figure \ref{fig:K-} illustrates the effect of varying $K^-$ from 1 to 5 under fixed values of the rest of the parameters, where $h=1$, $p=2$, $k^-=1$, $K^+=3.25$, $k^+=1.5$, $r=0.5$, $\sigma=1$. We can see from Figure \ref{subfig:all_K-} that as $K^-$ grows, $u$ is affected the most and it increases significantly indicating a decreased frequency of intervention, with a larger jumper size. Similar impact on control policy can be observed with an increasing $k^-$, shown in Figure \ref{fig:k-}. It pushes the equilibrium mean information to a lower level, as shown in Figure $\ref{subfig:fp_K-}$.
\begin{figure}[H]
\centering
\begin{subfigure}[t]{0.48\textwidth}
\centering
\includegraphics[width=\textwidth]{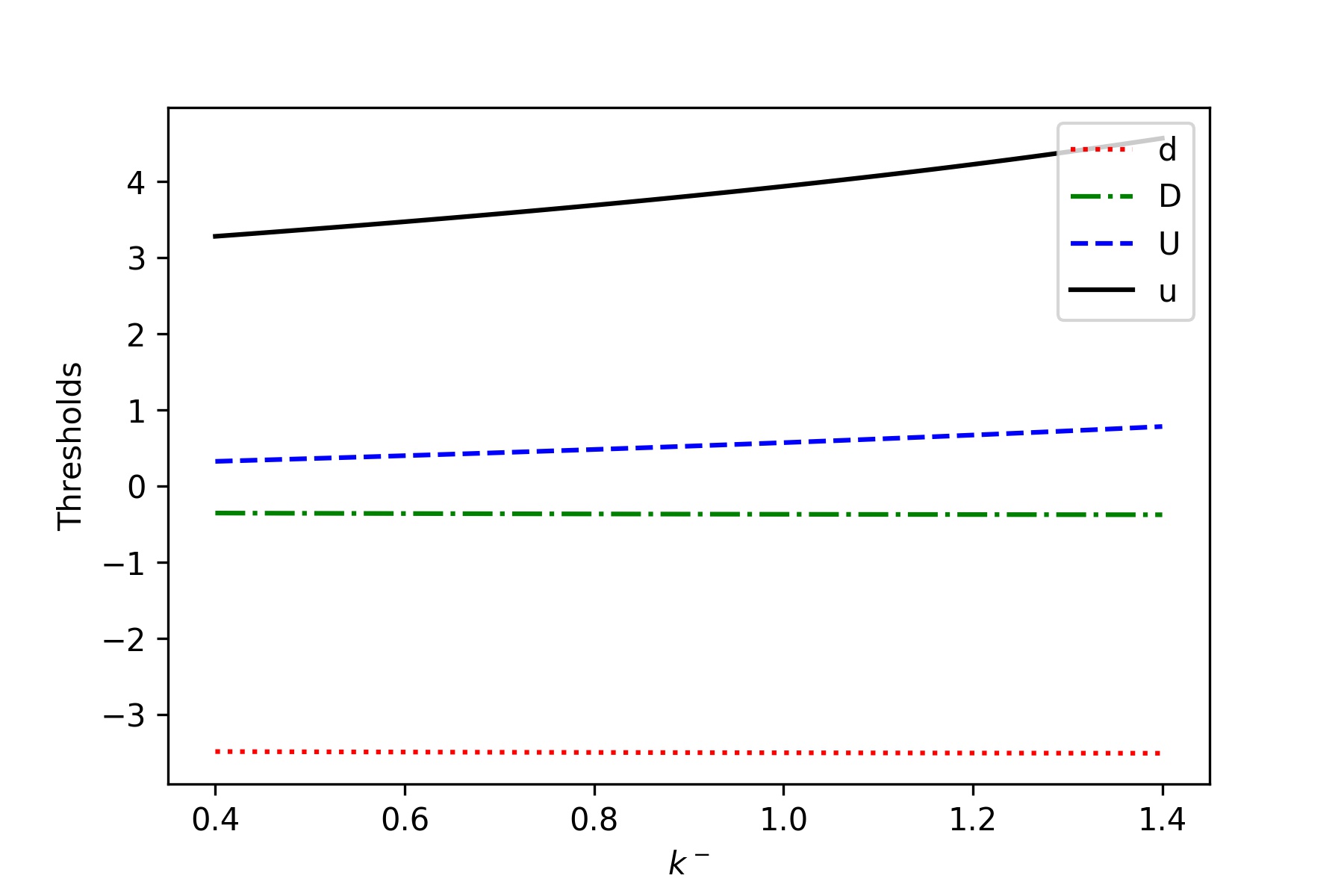}
\caption{All thresholds}
\label{subfig:all_k-}
\end{subfigure}
~
\begin{subfigure}[t]{0.48\textwidth}
\centering
\includegraphics[width=\textwidth]{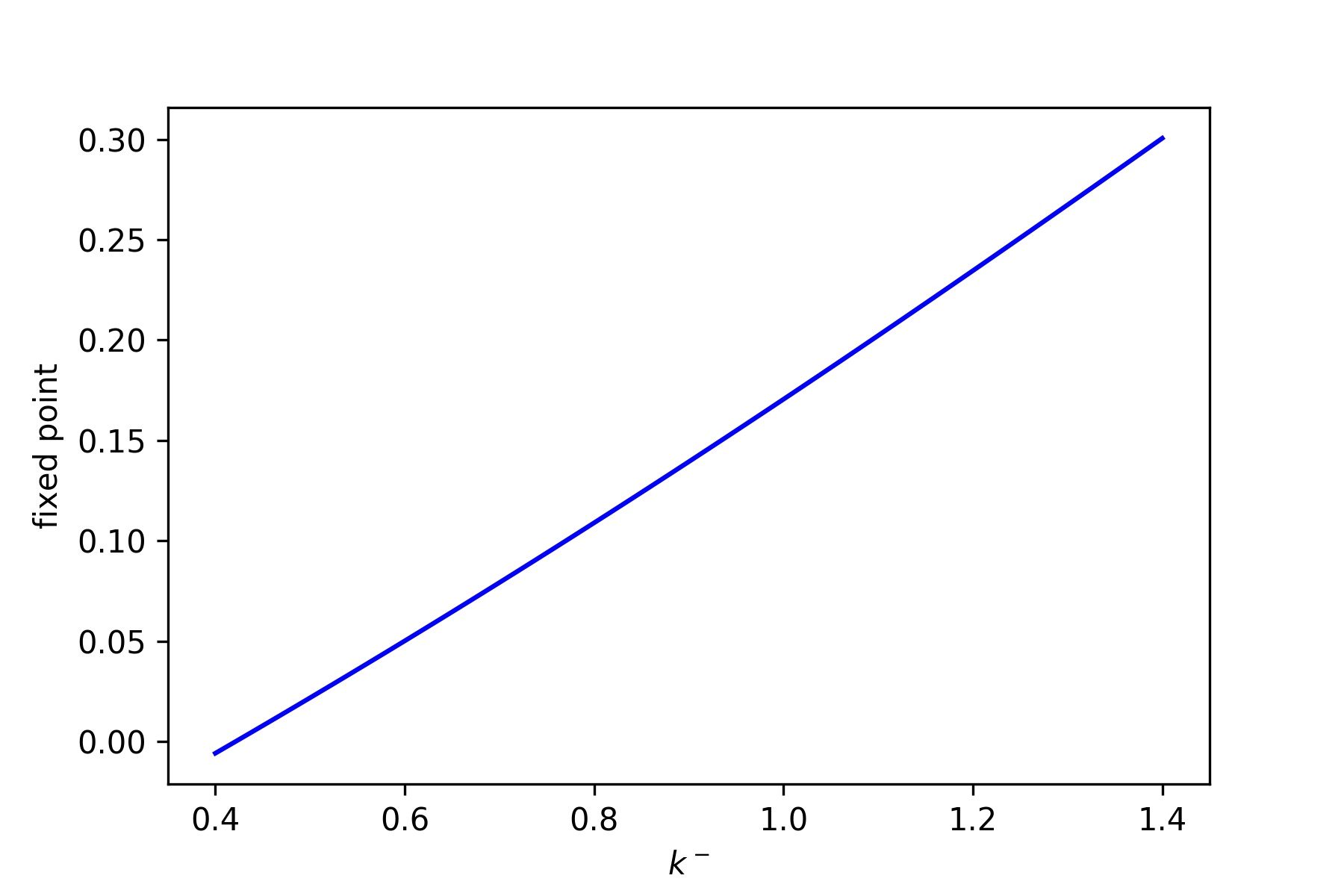}
\caption{Fixed point}
\label{subfig:fp_k-}
\end{subfigure}
\caption{Sensitivity of the thresholds w.r.t. $k^-$}
\label{fig:k-}
\end{figure}
A lower frequency of decreasing can lead the state to move right while larger jump size may push the state to the left. While, as shown in Figure \ref{subfig:fp_K-}, increasing $K^-$ results in a decreasing mean information, increasing $k^-$ causes mean information to increasing, see Figure \ref{subfig:fp_k-}.

\paragraph{Cost of instant increase $K^+$ and $k^+$.}The parameter $K^+$ is the amount of fixed cost to pay every time a player chooses to increase the state. A high fixed cost $K^+$ discourages the player from increasing the state too often. Therefore the players have the incentive to tolerate a low state value compared with the reference point $\alpha$  and once a player chooses to increase its state, the amount is adjustment will be larger. The more obvious impact is that the distance between $\alpha$ and the lower bound of the non-action region, i.e. $d$, decreases; the impact on $D$ is trickier, as the per unit cost of increasing the state stays unchanged. The impact of the increasing $K^+$ on the rest of the thresholds becomes the trade-of between unchanged holding cost, penalty cost, cost of decreasing the state, and the costs associated with potentially more frequent but larger in scales, instantaneous decreases in the state variables. 
\begin{figure}[H]
\centering
\begin{subfigure}[t]{0.48\textwidth}
\centering
\includegraphics[width=\textwidth]{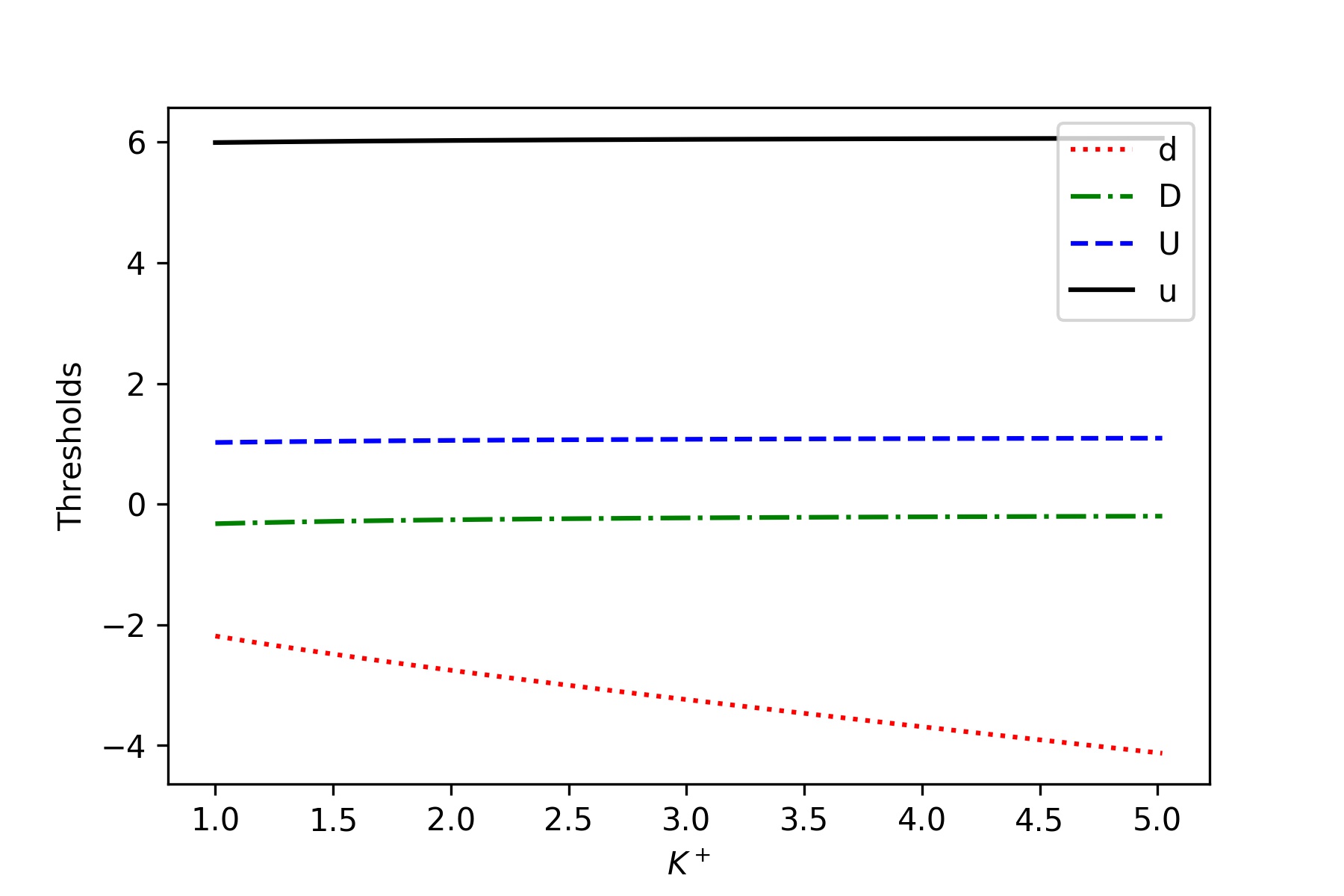}
\caption{All thresholds}
\label{subfig:all_K+}
\end{subfigure}
~
\begin{subfigure}[t]{0.48\textwidth}
\centering
\includegraphics[width=\textwidth]{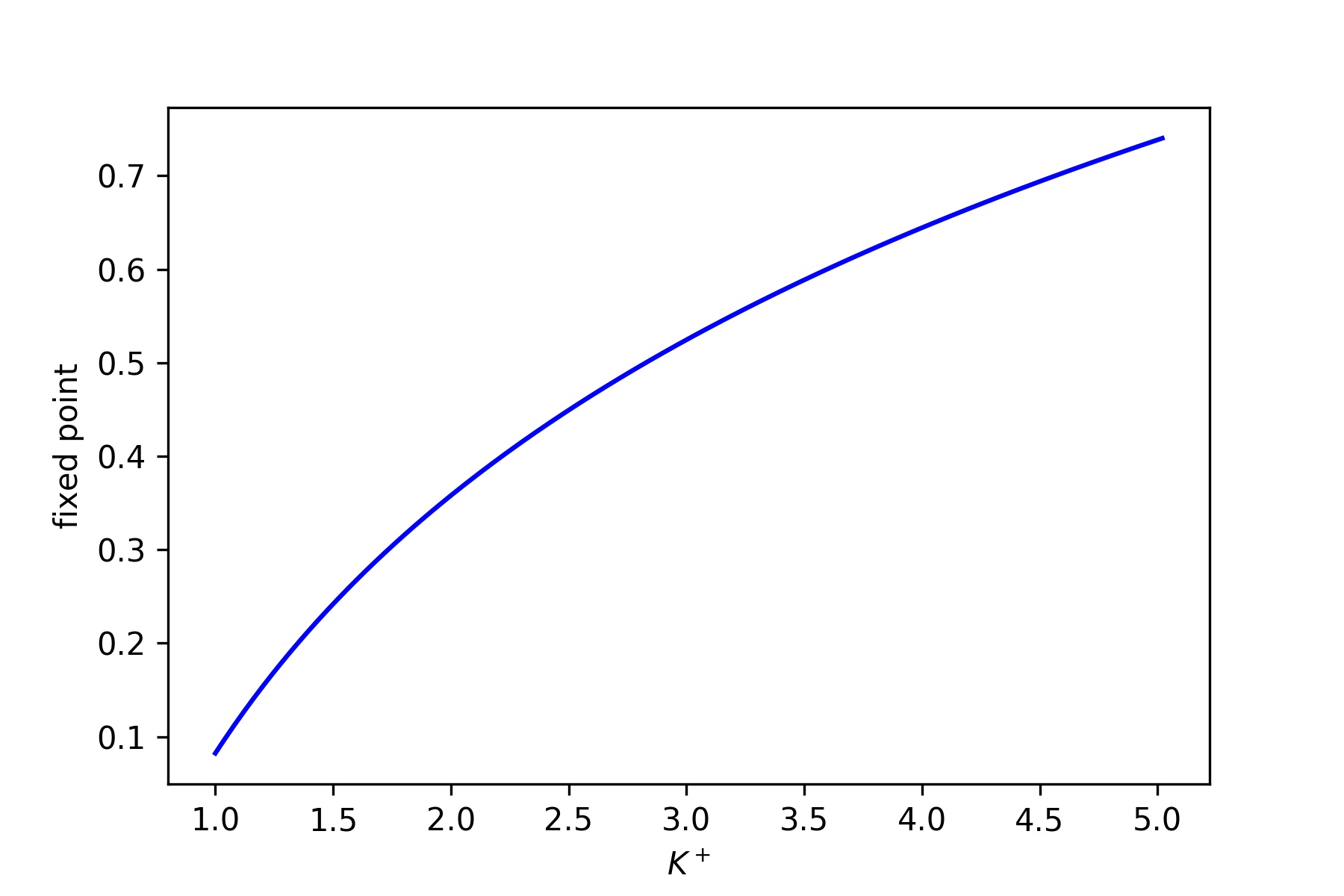}
\caption{Fixed point}
\label{subfig:fp_K+}
\end{subfigure}
\caption{Sensitivity w.r.t. $K^+$}
\label{fig:K+}
\end{figure}
Figure \ref{fig:K+} illustrates the effect of varying $K^+$ from 1 to 5 under fixed values of the rest of the parameters, where $h=1$, $p=2$, $K^-=3$, $k^-=1$, $k^+=1.5$, $r=0.5$, $\sigma=1$. We can see from Figure \ref{subfig:all_K+} that as $K^+$ grows, $d$ gets the most influence and decreases significantly, indicating a lower frequency of intervention with larger jump size. Similar impact on control policy can be seen if $k^+$ increases. 
\begin{figure}[H]
\centering
\begin{subfigure}[t]{0.48\textwidth}
\centering
\includegraphics[width=\textwidth]{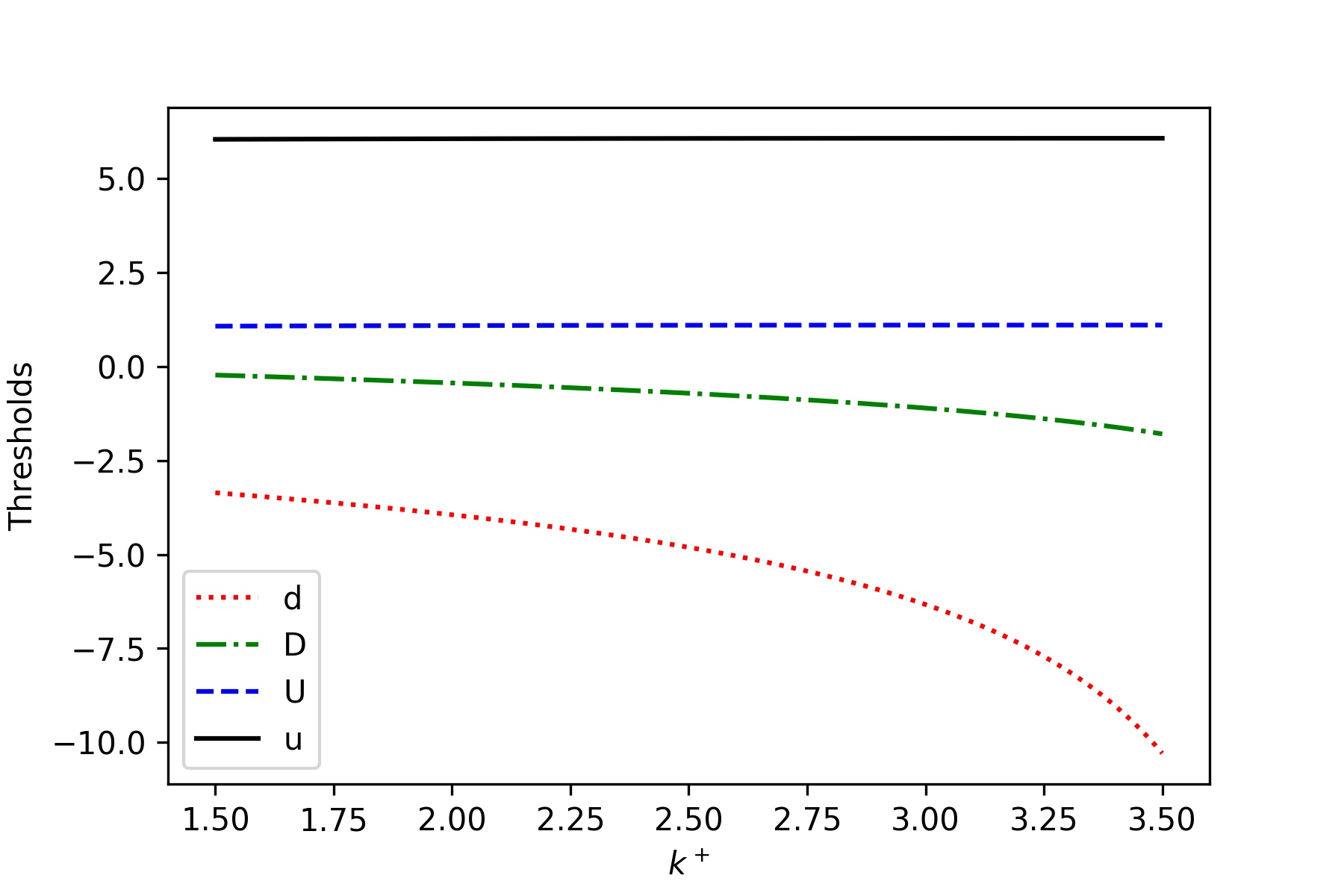}
\caption{All thresholds}
\label{subfig:all_k+}
\end{subfigure}
~
\begin{subfigure}[t]{0.48\textwidth}
\centering
\includegraphics[width=\textwidth]{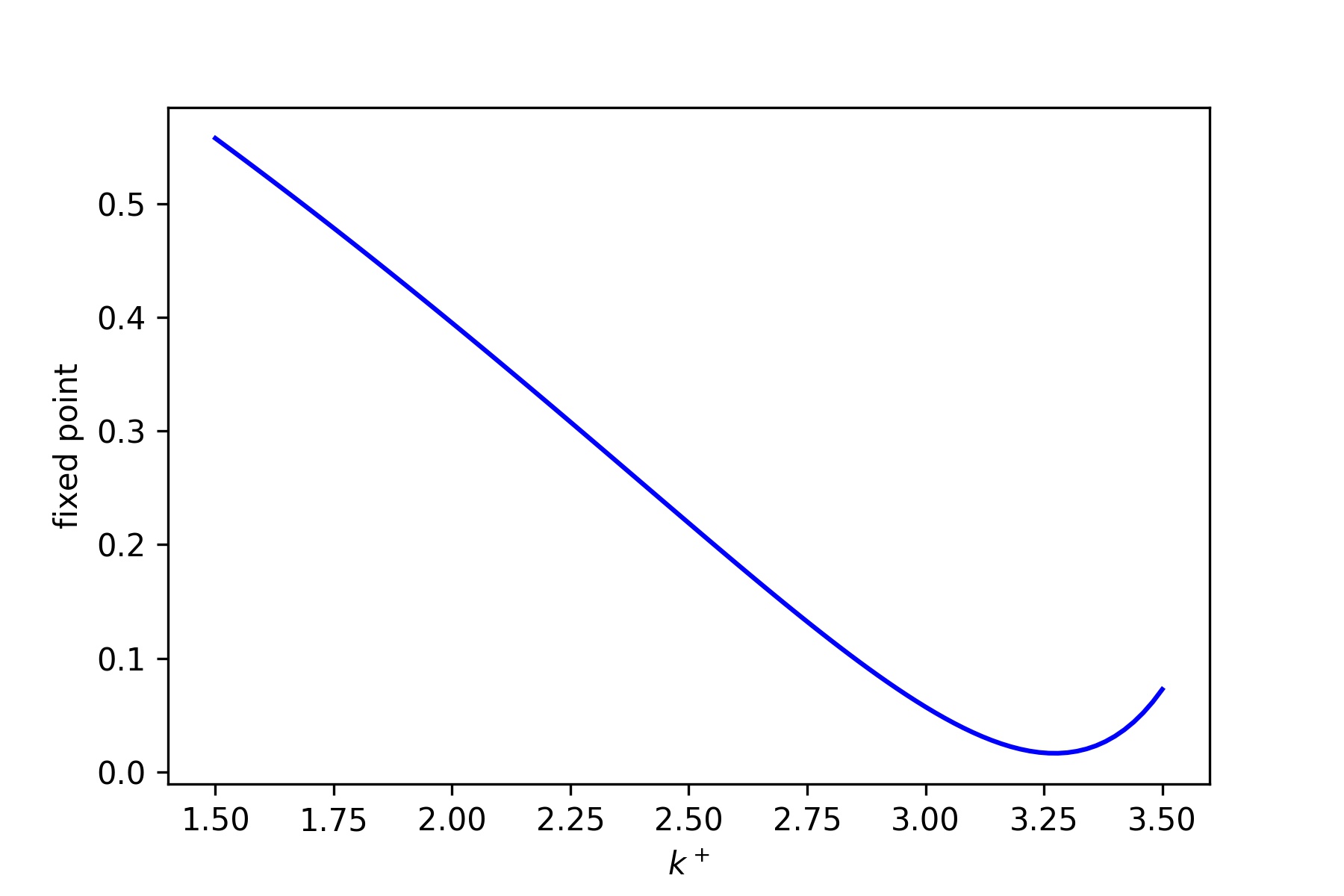}
\caption{Fixed point}
\label{subfig:fp_k+}
\end{subfigure}
\caption{Sensitivity w.r.t. $k^+$}
\label{fig:k+}
\end{figure}
A lower frequency of increase motivates the state to go left while larger jump size can push the state to the right. The impacts on the mean information by varying $K^+$ and $k^+$ are different as shown in Figure \ref{subfig:fp_K+} and Figure \ref{subfig:fp_k+}.

\paragraph{Discount rate $r$.} Figure \ref{fig:r} illustrates the effect of varying $r$ from 0.2 to 1.2 under fixed values of the rest of the parameters, where $h=1.5$, $p=2.5$, $K^-=3$, $k^-=1$, $K^+=3.25$, $k^+=1.5$, $\sigma=1$.
\begin{figure}[h]
\centering
\begin{subfigure}[t]{0.48\textwidth}
\centering
\includegraphics[width=\textwidth]{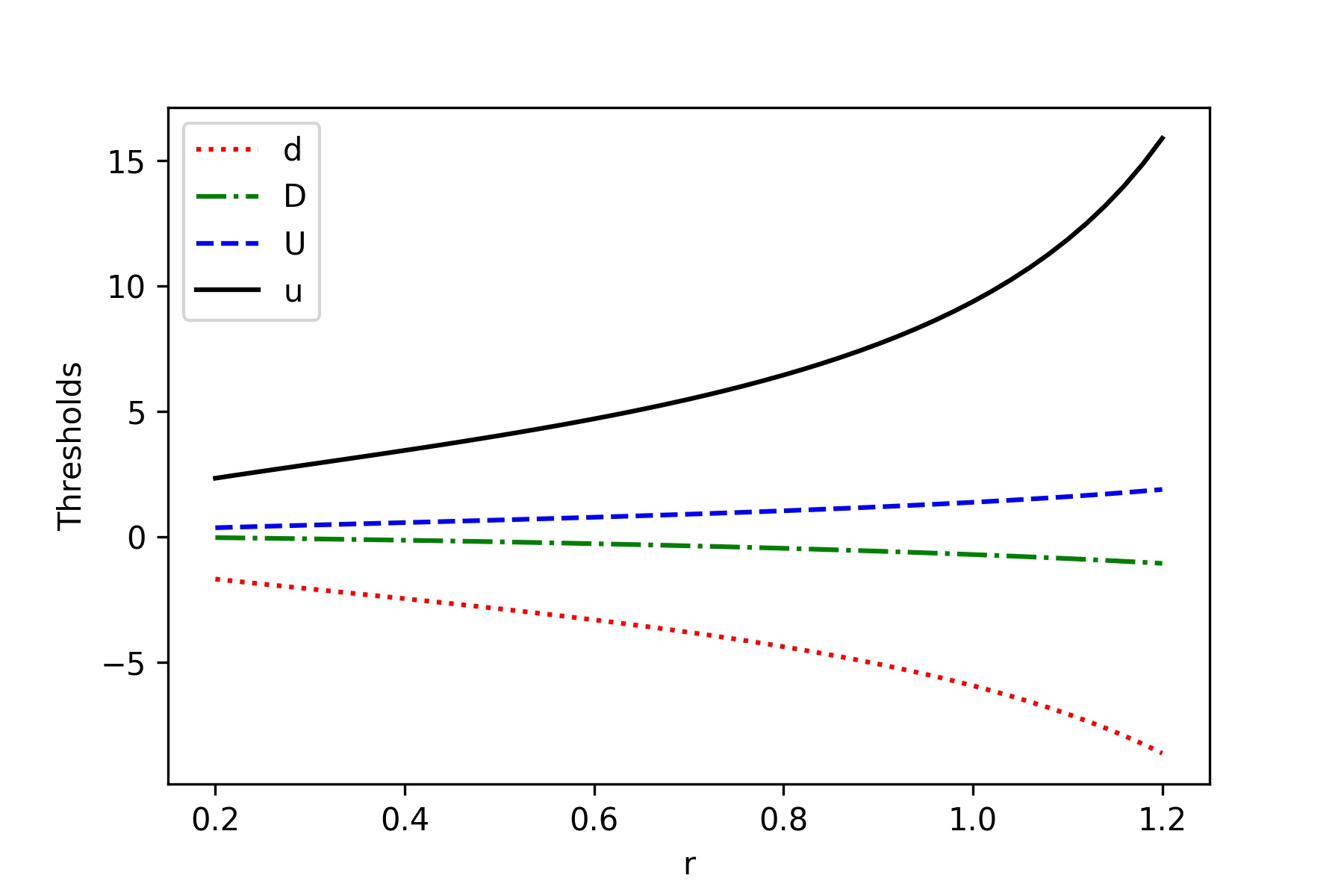}
\caption{All thresholds}
\label{subfig:all_r}
\end{subfigure}
~
\begin{subfigure}[t]{0.48\textwidth}
\centering
\includegraphics[width=\textwidth]{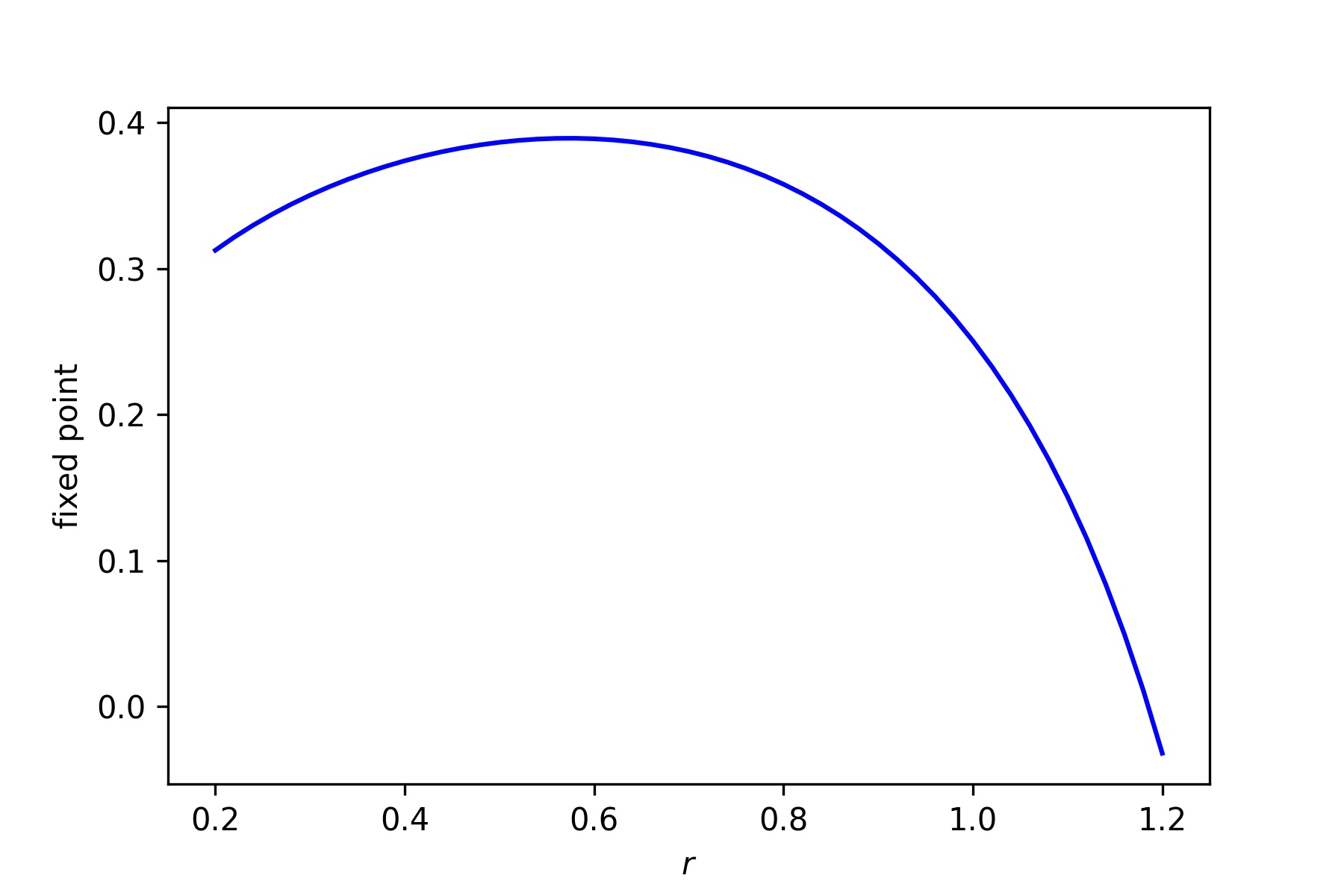}
\caption{Fixed point}
\label{subfig:fp_r}
\end{subfigure}
\caption{Sensitivity w.r.t. $r$}
\label{fig:r}
\end{figure}
We can see in Figure \ref{subfig:all_r} that as $r$ grows, $d$ and $u$ get influenced the most indicating a decreased intervention in both directions, with a larger jump size. A lower decreasing frequency and a higher intensity of instant increases can both push the state the right while a lower increasing frequency and a higher intensity of instant decreases can both cause the state to go left. The influence on the mean information is shown in Figure $\ref{subfig:fp_r}$, that the mean first increases and then decreases as $r$ grows.

\paragraph{Individual volatility $\sigma$.} Figure \ref{fig:sig} illustrates the effect of varying $\sigma$ from 0.5 to 2 under fixed values of the rest of the parameters, where $h=1$, $p=2$, $K^-=3$, $k^-=1$, $K^+=3.25$, $k^+=1.5$, $r=0.5$. 
\begin{figure}[h]
\centering
\begin{subfigure}[t]{0.48\textwidth}
\centering
\includegraphics[width=\textwidth]{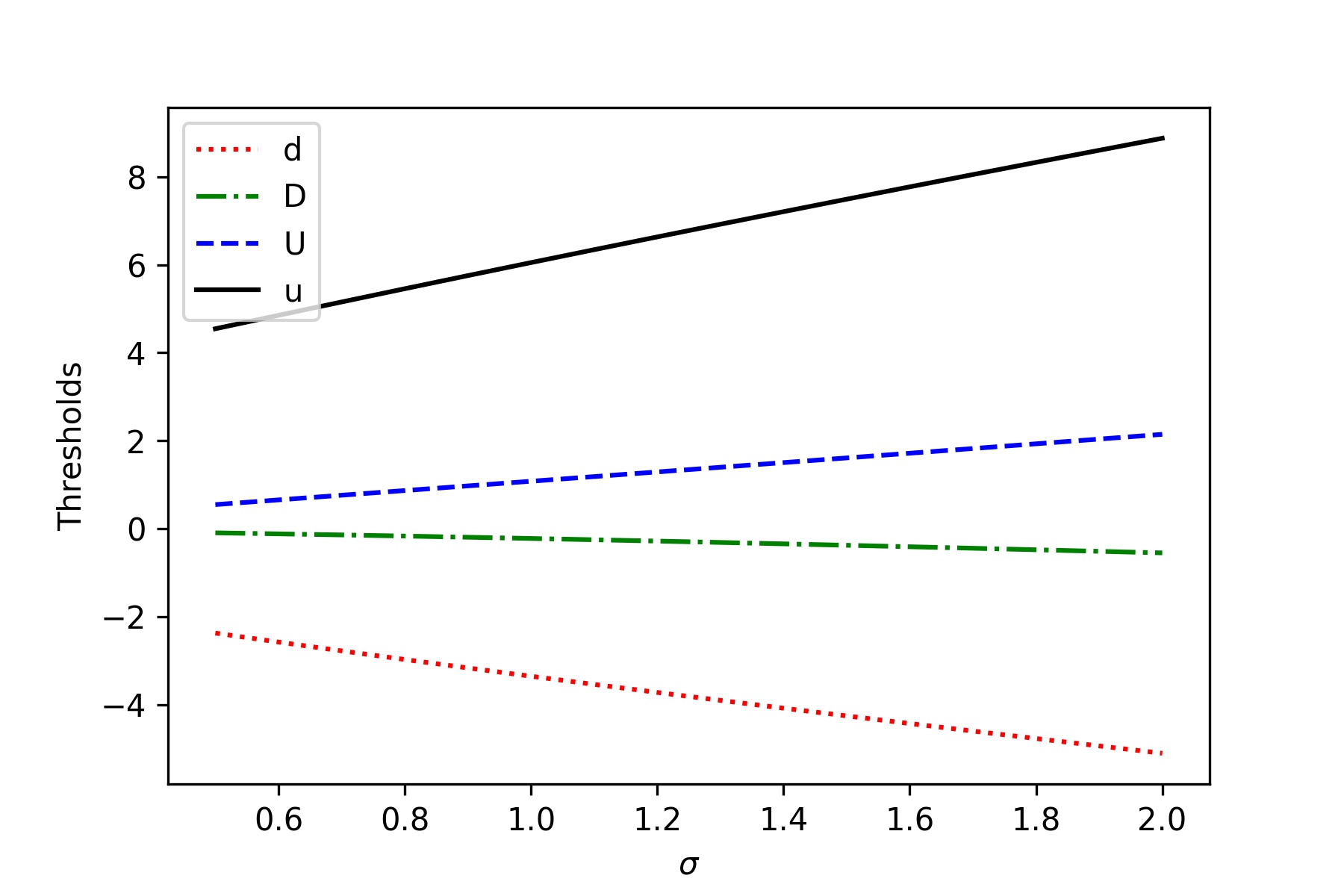}
\caption{All thresholds}
\label{subfig:all_sig}
\end{subfigure}
~
\begin{subfigure}[t]{0.48\textwidth}
\centering
\includegraphics[width=\textwidth]{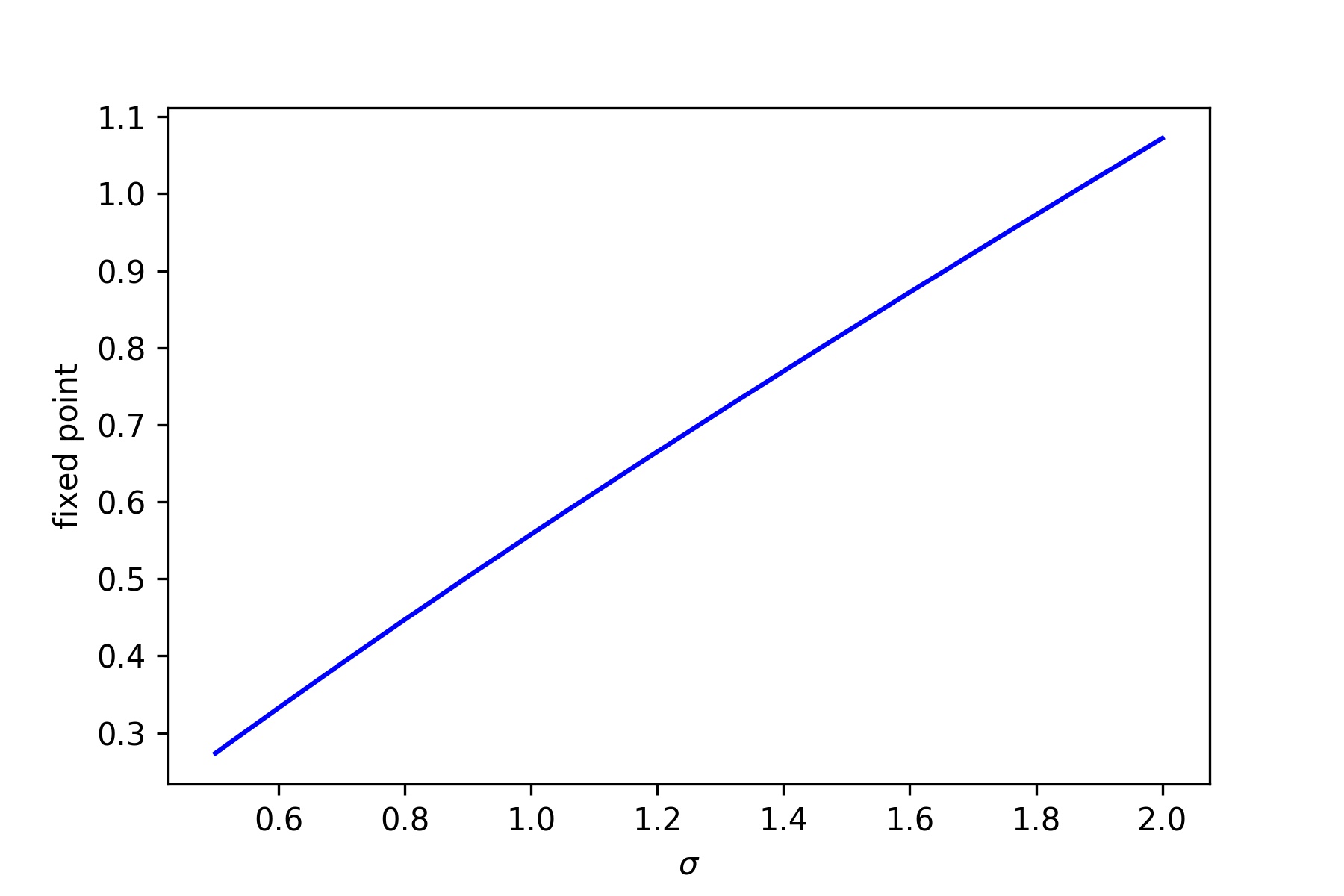}
\caption{Fixed point}
\label{subfig:fp_sig}
\end{subfigure}
\caption{Sensitivity w.r.t. $\sigma$}
\label{fig:sig}
\end{figure}
We can see from Figure \ref{subfig:all_sig} that as $\sigma$ grows, $d$ and $u$ get influenced the most indicating a decreased intervention in both directions, with a larger jump size. The influence on the mean information is shown in Figure $\ref{subfig:fp_sig}$, that the mean increases as $\sigma$ grows.

\bibliographystyle{plain}
\bibliography{impulse-gameUPDATEDjune2020}

\end{document}